%% file: main.tex
\documentclass[a4paper,10pt,bibliography=totoc,listof=totoc,headinclude=true,cleardoublepage=empty,oneside]{scrartcl}

\usepackage[english,ngerman]{babel}
\usepackage[utf8]{inputenc}
\usepackage{fullpage}
\usepackage{ifthen}
\usepackage{color}
\usepackage{amsmath,amsthm,amssymb,amsfonts,bbm}
\usepackage{mathtools}
\usepackage{graphicx, float}
\usepackage{psfrag}
\usepackage{enumitem}
\usepackage{bm}
\usepackage{mparhack}
\usepackage{wasysym}
\usepackage{pdflscape} %landscape mode
\usepackage[unicode,colorlinks=true,pagebackref=false]{hyperref}

\KOMAoptions{footinclude=false} % Fusszeile wird nicht zu Satzspiegel gezaehlt
\KOMAoptions{headsepline=true}  % Trennlinie zwischen Kopfzeile und Text
\KOMAoptions{DIV=12}            % beeinflusst Satzspiegel
\KOMAoptions{BCOR=8mm}          % Bindekorrew upwktur

\recalctypearea % berechne Satzspiegel neu

\numberwithin{equation}{section}
\numberwithin{figure}{section}

\graphicspath{ {./fig/} }

\newcommand\norm[1]{\left\lVert#1\right\rVert}

\newcommand*\curl{\mathop{}\!\mathrm{curl}}
\renewcommand*\div{\mathop{}\!\mathrm{div}}

\newcommand*\HOne{H^1(\Omega)}
\newcommand*\HDirOne{H_D^1(\Omega)}
\newcommand*\HZeroOne{H_0^1(\Omega)}

\newcommand*\HDiv{\pmb{H}(\Omega, \div)}
\newcommand*\HNeuDiv{\pmb{H}_N(\Omega, \div)}
\newcommand*\HZeroDiv{\pmb{H}_0(\Omega, \div)}

\newcommand*\HCurl{\pmb{H}(\Omega, \curl)}
\newcommand*\HNeuCurl{\pmb{H}_N(\Omega, \curl)}
\newcommand*\HZeroCurl{\pmb{H}_0(\Omega, \curl)}

\newcommand*\Sp{S_{p_s}(\mathcal{T}_h)}
\newcommand*\SpZero{S^0_{p_s}(\mathcal{T}_h)}
\newcommand*\SpDir{S^D_{p_s}(\mathcal{T}_h)}

\newcommand*\Nedelec{\pmb{\mathrm{N}}_{p_v}(\mathcal{T}_h)}
\newcommand*\NedelecZero{\pmb{\mathrm{N}}^0_{p_v}(\mathcal{T}_h)}
\newcommand*\NedelecNeu{\pmb{\mathrm{N}}^N_{p_v}(\mathcal{T}_h)}

\newcommand*\BDM{\pmb{\mathrm{BDM}}_{p_v}(\mathcal{T}_h)}
\newcommand*\BDMZero{\pmb{\mathrm{BDM}}^0_{p_v}(\mathcal{T}_h)}
\newcommand*\BDMNeu{\pmb{\mathrm{BDM}}^N_{p_v}(\mathcal{T}_h)}

\newcommand*\RT{\pmb{\mathrm{RT}}_{p_v-1}(\mathcal{T}_h)}
\newcommand*\RTZero{\pmb{\mathrm{RT}}^0_{p_v-1}(\mathcal{T}_h)}
\newcommand*\RTNeu{\pmb{\mathrm{RT}}^N_{p_v-1}(\mathcal{T}_h)}

\newcommand*\RTBDM{\pmb{\mathrm{V}}_{p_v}(\mathcal{T}_h)}
\newcommand*\RTBDMZero{\pmb{\mathrm{V}}^0_{p_v}(\mathcal{T}_h)}
\newcommand*\RTBDMNeu{\pmb{\mathrm{V}}^N_{p_v}(\mathcal{T}_h)}

\newcommand*\Ih{\pmb{I}_h}
\newcommand*\IhZero{\pmb{I}_h^0}

\newcommand\restr[2]{{% we make the whole thing an ordinary symbol
  \left.\kern-\nulldelimiterspace % automatically resize the bar with \right
  #1 % the function
  \vphantom{\big|} % pretend it's a little taller at normal size
  \right|_{#2} % this is the delimiter
  }}

\def\phiuvec{
  \begin{pmatrix}
      \pmb{\varphi} \\
      u
  \end{pmatrix}}

  \def\psivvec{
    \begin{pmatrix}
        \pmb{\psi} \\
        v
    \end{pmatrix}}

\def\productspace{\HNeuDiv \times \HDirOne}

\def\ltwo{\Omega}

\newcommand{\eremk}{\hbox{}\hfill\rule{0.8ex}{0.8ex}}

\newtheorem{theorem}{Theorem}[section]

\newtheorem{lemma}[theorem]{Lemma}
\newtheorem{corollary}[theorem]{Corollary}

\theoremstyle{definition}

\newtheorem{example}[theorem]{Example}
\newtheorem{remark}[theorem]{Remark}

\newtheorem{assumption}[theorem]{Assumption}

\begin{document}

\title{Optimal convergence rates in $L^2$ for a \\ first order system least squares \\ finite element method. \\ Part I: homogeneous boundary conditions}
\author{
  M. Bernkopf
  \thanks{
  Institute for Analysis and Scientific Computing,
  TU Wien,
  A-1040 Vienna,
  \href{mailto:maximilian.bernkopf@tuwien.ac.at}{maximilian.bernkopf@tuwien.ac.at}
  }
  \and
  J.M. Melenk
  \thanks{
  Institute for Analysis and Scientific Computing,
  TU Wien,
  A-1040 Vienna,
  \href{mailto:melenk@tuwien.ac.at}{melenk@tuwien.ac.at}
  }
}

\maketitle
\selectlanguage{english}
\pagenumbering{arabic}

\begin{abstract}
  We analyze a divergence based first order system least squares method applied to a second order elliptic model problem with homogeneous boundary conditions.
  We prove optimal convergence in the $L^2(\Omega)$ norm for the scalar variable.
  Numerical results confirm our findings.
\end{abstract}

\input{00_notation}
\input{01_model_problem}
\input{02_duality_argument}
\input{03_error_analysis}
\input{04_numerical_examples}

\textbf{Acknowledgement}:
MB is grateful for the financial support by the Austrian Science Fund (FWF) through the doctoral school \textit{Dissipation and dispersion in nonlinear PDEs} (grant W1245).
MB thanks Joachim Sch\"oberl (TU Wien) and his group for their support in connection with the numerical experiments.

\appendix
\input{05_appendix}

%%%%%%%%%%%%%%%%%%%%%%%%%%%%%%%%%%%%%%%%%%%%%%%%%%%%%%%%%%%%%%%%%%%%%%%%%%%%%%%%%%%%%%%%%%%%%%%%%%%%%%%%%%%%%%
%%%%%%%%%%%%%%%%%%%%%%%%%%%%%%%%%%%%%%%%%%%%%%%%%%%%%%%%%%%%%%%%%%%%%%%%%%%%%%%%%%%%%%%%%%%%%%%%%%%%%%%%%%%%%%
% LITERATURBERZEICHNIS MIT BIBTEX --> MATHSCINET VERWENDEN
%%%%%%%%%%%%%%%%%%%%%%%%%%%%%%%%%%%%%%%%%%%%%%%%%%%%%%%%%%%%%%%%%%%%%%%%%%%%%%%%%%%%%%%%%%%%%%%%%%%%%%%%%%%%%%
%%%%%%%%%%%%%%%%%%%%%%%%%%%%%%%%%%%%%%%%%%%%%%%%%%%%%%%%%%%%%%%%%%%%%%%%%%%%%%%%%%%%%%%%%%%%%%%%%%%%%%%%%%%%%%

\bibliographystyle{alpha}
\bibliography{literature.bib}

\end{document}

%% file: 00_notation.tex
% !TEX root = main.tex
\section{Introduction}\label{section:introduction}

Least Squares Finite Element Methods (LSFEM) are an important class of numerical methods for the solution of partial differential equations
with a variety of applications.  The main idea of the LSFEM  is to reformulate the partial differential equation
of interest as a minimization problem, for which a variety of tools is available. For example, even for non-symmetric or indefinite problems, the
discretization with the least squares approach leads to  symmetric, positive definite systems, which can be solved with well-established numerical technologies.
Furthermore, the least squares technique is naturally quasi-optimal, albeit in a problem-dependent norm. For second order PDEs, which is the setting of the
present work, the most common least squares approach is that of rewriting the equation as a First Order Least Squares System (FOSLS) that can be
discretized with established finite element techniques. A benefit is that many quantities of interest are approximated directly without the need of postprocessing.
%The corresponding Euler-Lagrange equations then lead to a variational formulation of this minimization problem.
We mention \cite{bochev-gunzburger09} as a classical monograph on the topic as well as the papers \cite{jespersen77, cai-lazarov-manteufel-mccormick94, cai-manteufel-mccormick97, bochev-gunzburger05}.
%\JMM{welche Norm equivalence?} Once a norm equivalence result is established the least squares finite element method features many advantages.
%Among these advantages are:
%The bilinear form and the linear functional defining the variation formulation of the continuous as well as the discrete problem are continuous and the bilinear form is coercive and symmetric.
%This in turn guarantees unique solvability on a discrete and continuous level.
%The resulting linear system is symmetric and positive definite, which is especially convenient when dealing with problems where the standard Galerkin discretization results in indefinite matrices.
%Furthermore, the least squares finite element method features quasi-optimality,
%albeit in some nonstandard residual norms.
%Finally, first order formulations approximate many quantities of interests simultaneously.

The present work considers a Poisson-like second order model problem written as a system of first order equations.
For the discretization, an $\HDiv \times \HOne$-conforming least squares formulation is employed.
Even though our model problem in its standard $\HOne$ formulation is coercive our methods and lines of proof can most certainly be applied to other problems as well,
see \cite{bernkopf-melenk19, chen-qiu17} for an application to the Helmholtz equation.
The LSFEM is typically quasi-optimality in some problem-dependent energy norm, which is, however, somewhat intractable; \textsl {a priori}
error estimates in more familiar norms such as the $L^2(\Omega)$ norm of the scalar variable are thus desirable.
Numerical examples in our previous work \cite{bernkopf-melenk19} suggested convergence rates in standard norms such as the $L^2(\Omega)$-norm which,
to our best knowledge, are not explained by the current theory. In the present work, we develop such a convergence theory
with minimal assumptions on the regularity of the right-hand side.
%We give a more precise overview of the available results below.
%Closing this gap in the literature will be our main focus.
\\
\subsection{Contribution of the present work}
Our main contribution are optimal $L^2(\Omega)$ based convergence result for the least squares approximation $u_h$ to the scalar variable $u$.
Furthermore, we derive $hp$ error estimates for the gradient of the scalar variable $u$, which do not seem to be available in the current literature,
as well as an $hp$ error estimate for the vector variable $\pmb{\varphi}$ in the $L^2(\Omega)$ norm, which is available in the literature for a pure $h$-version.
These optimality results are new in the sense that we achieve optimal convergence rates under minimal regularity assumptions on the data.
%Below we give an overview of the state what to our best knowledge the state of the art is.
%Before doing so, let us clarify the notion of optimality.
Here, we call a method optimal in a certain norm, if the norm of the error made by the method
is of the same order as the best approximation of the employed space.

\subsubsection{Review of related results}
%\begin{itemize}
%  \item
In \cite{jespersen77} the author considered the classical model problem
  $-\Delta u = f$ with inhomogeneous Dirichlet boundary condition $u = g$ in some smooth domain $\Omega$.
  Unlike the present work the least squares formulation employs vector valued $H^1(\Omega)$ functions
  instead of $\HDiv$ for the vector variable.
  The corresponding finite element spaces are chosen such that they satisfy simultaneous approximation properties
  in $L^2(\Omega)$ and $H^1(\Omega)$ for both the scalar variable $u$ and the vector variable $\pmb{\varphi}$.
  Using a duality argument akin to the one used in the present work the author arrived at the error estimate
  \begin{equation*}
    \norm{u-u_h}_{L^2(\Omega)} \lesssim h \norm{(\pmb{\varphi} - \pmb{\varphi}_h , u-u_h )}_b,
  \end{equation*}
  see \cite[Thm.~{4.1}]{jespersen77}, where $\norm{(\cdot, \cdot)}_b$ denotes the corresponding energy norm.
  At this point higher order convergence rates are just a question of approximation properties in $\norm{(\cdot, \cdot)}_b$,
  see \cite[Lemma~{3.1}]{jespersen77} for a precise statement.
  As stated after the proof of \cite[Thm.~{4.1}]{jespersen77}, one can extract optimal convergence rates for sufficiently smooth data $f$ and $g$.
The smoothness of the data is important as the following considerations show:
  For the case of a smooth boundary $\Gamma$ and $f \in L^2(\Omega)$ and $g \in H^{3/2}(\Gamma)$,
  elliptic regularity gives $u \in H^2(\Omega)$.
  Therefore $u$ can be approximated by globally continuous piecewise polynomials of degree greater or equal
  to one with a error $O(h^2)$ in the $L^2(\Omega)$ norm,
  which is achieved by classical FEM, due to the Aubin-Nitsche trick.
  In contrast, the above least squares estimate does not give the desired rate:
  The norm $\norm{( \pmb{\varphi} - \pmb{\varphi}_h, u-u_h )}_b$ contains a term of the form
  \begin{equation*}
    \norm{ \nabla \cdot ( \pmb{\varphi} - \pmb{\varphi}_h ) }_{L^2(\Omega)} = \norm{ f - \nabla \cdot \pmb{\varphi}_h }_{L^2(\Omega)},
  \end{equation*}
  from which no further convergence rate can be extracted, since $f$ is only in $L^2(\Omega)$.

%  \item
In \cite{cai-lazarov-manteufel-mccormick94} (see also \cite{cai-manteufel-mccormick97}) the problem $- \nabla \cdot (A \nabla u) + X u = f$
  with uniformly elliptic diffusion matrix $A$ and $X$ a linear differential operator of order at most one together with
  homogeneous mixed Dirichlet and Neumann boundary conditions was considered.
  The least squares formulation presented therein employs the same spaces as the present work.
  Apart from nontrivial norm equivalence results, see \cite[Thm.~{3.1}]{cai-lazarov-manteufel-mccormick94},
  they also derived the following estimate of the least squares approximation
  \begin{equation*}
    \norm{u - u_h}_{H^1(\Omega)} + \norm{ \pmb{\varphi} - \pmb{\varphi}_h }_{H(\div, \Omega)} \lesssim h^s (\norm{u}_{H^{s+1}(\Omega)} + \norm{\pmb{\varphi}}_{H^{s+1}(\Omega)}),
  \end{equation*}
  assuming $u \in H^{s+1}(\Omega)$ and $\pmb{\varphi} \in \pmb{H}^{s+1}(\Omega)$.
  This result is then optimal in the stated norm, however, the assumed regularity is somewhat unsatisfactory,
  in the sense that if the solution $u \in H^{s+1}(\Omega)$ then the relation $\nabla u + \pmb{\varphi} = 0$ merely provides the
regularity $\pmb{\varphi} \in \pmb{H}^{s}(\Omega)$ and not the assume
regularity $\pmb{\varphi} \in \pmb{H}^{s+1}(\Omega)$.

%  \item
Finally in \cite{bochev-gunzburger05} the same model problem as well as the same least squares formulation is considered.
  The main goal of \cite{bochev-gunzburger05} is to establish $L^2(\Omega)$ error estimates for $u$ and $\pmb{\varphi}$.
  In \cite[Lemma~{3.4}]{bochev-gunzburger05} a result similar to \cite[Thm.~{4.1}]{jespersen77} is obtained.
  This result, however, suffers from the same drawback as elaborated above.
  Furthermore, they prove optimality of the error of the vector variable $\pmb{\varphi}$ in the $L^2(\Omega)$ norm, see \cite[Cor.~{3.7}]{bochev-gunzburger05}.
%\end{itemize}

The main tools for \textsl{a priori} error estimates in more tractable norms such as  $L^2(\Omega)$ instead of the energy norm
in a least squares setting are, as it is done in the present paper and the above literature,
duality arguments, which lead to an estimate of the form
\begin{equation*}
  \norm{u-u_h}_{L^2(\Omega)} \lesssim h \norm{( \pmb{\varphi} - \pmb{\varphi}_h , u-u_h )}_b.
\end{equation*}
As elaborated above it is not possible to extract the desired optimal rate from this estimate directly.
%Usually starting from a duality argument one exploits Galerkin orthogonality and then applies the Cauchy-Schwarz inequality.
In the proof of one of our main result (Theorem~\ref{theorem:e_u_optimal_l2_error_estimate}) we exploit the duality argument in a more delicate way,
which allows us to lower the regularity requirements on $\pmb{\varphi}$ to what could be expected from the regularity of the data $f$.
Key components in the proof are the div-conforming approximation operators $\IhZero$ and $\Ih$ (cf.\ Lemmas~\ref{lemma:Ih_well_defined}, \ref{lemma:properties_of_Ih}),
which are also of independent interest.
%resulting in an improved error estimate, leading to the desired rate of convergence.

\subsubsection{Notation}
Throughout this work, $\Omega$ denotes a bounded simply connected domain in $\mathbb{R}^n$, $n \in\{ 2,3\}$,
with $C^\infty$ boundary $\Gamma \coloneqq \partial \Omega$ and outward unit normal vector $\pmb{n}$.
Let $\Gamma$ consist of two disjoint parts $\Gamma_D$ and $\Gamma_N$.
We consider the following spaces:

\noindent\begin{minipage}{.48\linewidth}
\begin{align*}
  \HOne      &= \{ u \in L^2(\Omega) \colon \nabla u \in \pmb{L}^2(\Omega) \}, \\
  \HDirOne   &= \{ u \in \HOne       \colon u = 0 \text{ on } \Gamma_D \}, \\
  \HZeroOne  &= \{ u \in \HOne       \colon u = 0 \text{ on } \Gamma \},
\end{align*}
\end{minipage}%
\noindent\begin{minipage}{.48\linewidth}
  \begin{align*}
    \HCurl     &= \{ \pmb{\varphi} \in \pmb{L}^2(\Omega) \colon \nabla \times \pmb{\varphi} \in \pmb{L}^2(\Omega) \}, \\
    \HNeuCurl  &= \{ \pmb{\varphi} \in \HCurl            \colon \pmb{n} \times \pmb{\varphi} = 0 \text{ on } \Gamma_N \}, \\
    \HZeroCurl &= \{ \pmb{\varphi} \in \HCurl            \colon \pmb{n} \times \pmb{\varphi} = 0 \text{ on } \Gamma \},
  \end{align*}
\end{minipage}
\\
\begin{minipage}{\linewidth}
  \begin{align*}
    \HDiv      &= \{ \pmb{\varphi} \in \pmb{L}^2(\Omega) \colon \nabla \cdot \pmb{\varphi} \in L^2(\Omega) \}, \\
    \HNeuDiv   &= \{ \pmb{\varphi} \in \HDiv             \colon \pmb{\varphi} \cdot \pmb{n} = 0 \text{ on } \Gamma_N \}, \\
    \HZeroDiv  &= \{ \pmb{\varphi} \in \HDiv             \colon \pmb{\varphi} \cdot \pmb{n} = 0 \text{ on } \Gamma \}.
  \end{align*}
\end{minipage}
\\
For further detail and references see \cite{monk03,boffi-brezzi-fortin13}.
Since we will look at a first order system formulation we have two finite element spaces to choose,
one for the scalar variable $u$ and one for the vector variable $\pmb{\varphi}$.
We consider the following finite element spaces:

\noindent\begin{minipage}{.25\linewidth}
\begin{align*}
  \Sp     &\subseteq \HOne, \\
  \SpDir  &\subseteq \HDirOne, \\
  \SpZero &\subseteq \HZeroOne,
\end{align*}
\end{minipage}%
\begin{minipage}{.25\linewidth}
\begin{align*}
  \Nedelec     &\subseteq \HCurl, \\
  \NedelecNeu  &\subseteq \HNeuCurl, \\
  \NedelecZero &\subseteq \HZeroCurl,
\end{align*}
\end{minipage}%
\begin{minipage}{.50\linewidth}
\begin{align*}
  \RT     \subseteq \BDM     &\subseteq \HDiv, \\
  \RTNeu  \subseteq \BDMNeu  &\subseteq \HNeuDiv, \\
  \RTZero \subseteq \BDMZero &\subseteq \HZeroDiv,
\end{align*}
\end{minipage}
\\
\\
\noindent
where the polynomial approximation of the scalar and vector variable is denoted by $p_s \geq 1$ and $p_v \geq 1$ respectively.
For brevity we also denote by $\RTBDM$ either the space $\RT$ or $\BDM$.
The spaces $\RTBDMNeu$ and $\RTBDMZero$ are denoted analogously.
Furthermore, the N\'ed\'elec space $\Nedelec$ is either of type one or two, depending on the choice of $\RTBDM$.
The same convention applies to the spaces with boundary conditions. See again \cite{monk03} for further details as well as Section~\ref{section:error_analysis}.
Further notational conventions will be:
\begin{itemize}
  \item lower case roman letters like $u$ and $v$ will be reserved for scalar valued functions;
  \item lower case boldface greek letters like $\pmb{\varphi}$ and $\pmb{\psi}$ will be reserved for vector valued functions;
  \item $K$ denotes the physical element and $\widehat{K}$ denotes the reference element;
%  \item \JMM{brauchen wir ueberhaupt den pull-back aufs Referenzelement in einer signifikanten weise?} quantities without a \, $\widehat{\cdot}$ \, will be either global quantities or quantities defined on the physical element $K$, whereas quantities with a \, $\widehat{\cdot}$ \, are related to the reference element $\widehat{K}$,
  \item quantities like $u_h$ and $\pmb{\varphi}_h$ will be reserved for functions from the corresponding finite element space, again scalar and vector valued respectively;
  \item if not stated otherwise discrete functions without a \, $\tilde{\cdot}$ \, will be in some sense fixed, e.g., resulting from a certain discretization scheme,
        whereas functions with a \, $\tilde{\cdot}$ \, will be arbitrary, e.g., when dealing with quasi-optimality results.
\end{itemize}

\subsubsection{Outline}
The outline of this paper is as follows.
In Section~\ref{section:model_problem} we introduce the model problem,
the first order system least squares (FOSLS) method itself and prove norm equivalence results,
which in turn guarantee unique solvability of the continuous as well as the discrete least squares formulation.
Section~\ref{section:duality_argument} is devoted to the proof of duality results for the scalar variable, the gradient of the scalar variable as well as the vector variable.
In the beginning of Section~\ref{section:error_analysis} we first exploit the duality result of Section~\ref{section:duality_argument} in order to prove $L^2(\Omega)$ error estimates for the scalar variable of the primal as well as the dual problem.
We then argue first heuristically that these results are actually suboptimal and can be further improved.
To that end we introduce an approximation operator that also satisfies certain orthogonality relations
and prove best approximation results for this operator,
which are then used to prove our main result (Theorem~\ref{theorem:e_u_optimal_l2_error_estimate}).
Furthermore, we derive $L^2(\Omega)$ error estimates for the gradient of the scalar variable as well as the vector variable.
In Section~\ref{section:numerical_examples} we present numerical examples showcasing the proved convergence rates,
focusing especially on the case of finite Sobolev regularity.

%% file: 01_model_problem.tex
% !TEX root = main.tex
\section{Model problem}\label{section:model_problem}

%In the following we are going to introduce our first model problem.
Let $\Gamma = \partial\Omega$ consist of two disjoint parts $\Gamma_D$ and $\Gamma_N$ and let $f \in L^2(\Omega)$.
(Later, we will focus on the special cases $\Gamma = \Gamma_D$ and $\Gamma=  \Gamma_N$.)
For $\gamma > 0$ fixed we consider the following model problem
%\begin{subequations}
\begin{align}
\label{eq:model_problem}
\begin{alignedat}{2}
  - \Delta u + \gamma u &= f \quad   &&\text{in } \Omega,\\
                      u &= 0         &&\text{on } \Gamma_D, \\
           \partial_n u &= 0         &&\text{on } \Gamma_N.
\end{alignedat}
\end{align}
%\end{subequations}
We formulate (\ref{eq:model_problem}) a first order system.
Introducing the new variable $\pmb{\varphi} = -\nabla u$ we formally arrive at the system
\begin{subequations}
\label{eq:model_problem_first_order_system}
\begin{alignat}{2}
%\begin{alignedat}{2}
  \nabla \cdot \pmb{\varphi} + \gamma u &= f \quad   &&\text{in } \Omega,   \\
               \nabla u + \pmb{\varphi} &= \pmb{0}   &&\text{in } \Omega,   \\
                                      u &= 0         &&\text{on } \Gamma_D, \\
            \pmb{\varphi} \cdot \pmb{n} &= 0         &&\text{on } \Gamma_N.
%\end{alignedat}
\end{alignat}
\end{subequations}
Introducing the differential operator $\mathcal{L} \colon \productspace \to L^2(\Omega) \times \pmb{L}^2(\Omega)$, given by
\begin{equation*}
  \mathcal{L} \phiuvec =
  \begin{pmatrix} \nabla \cdot & \gamma \\ 1 & \nabla \end{pmatrix} \phiuvec =
    \begin{pmatrix} \nabla \cdot \pmb{\varphi} + \gamma u  \\ \nabla u + \pmb{\varphi} \end{pmatrix},
\end{equation*}
we want to solve the equation
\begin{equation*}
  \mathcal{L} \phiuvec = \begin{pmatrix} f \\ \pmb{0} \end{pmatrix}.
\end{equation*}
The least squares approach to this problem is to find $(\pmb{\varphi}, u) \in \productspace$ such that
\begin{equation*}
  \left( \mathcal{L} \phiuvec, \mathcal{L} \psivvec \right)_{\ltwo}
  =
  \left(\begin{pmatrix} f \\ \pmb{0} \end{pmatrix} ,  \mathcal{L} \psivvec \right)_{\ltwo}
  \quad \forall \, (\pmb{\psi}, v) \in \productspace,
\end{equation*}
where $(\cdot , \cdot)_{\ltwo}$ denotes the usual $L^2(\Omega)$ scalar product.
Introducing now the bilinear form $b$ and the linear functional $F$ by
\begin{align}
\label{eq:b}
	b( (\pmb{\varphi}, u), (\pmb{\psi}, v) )
  &\coloneqq
  ( \nabla \cdot \pmb{\varphi} + \gamma u , \nabla \cdot \pmb{\psi} + \gamma v )_{\ltwo}
  +
  ( \nabla u + \pmb{\varphi} , \nabla v + \pmb{\psi} )_{\ltwo}, \\
\label{eq:F}
	F((\pmb{\psi}, v)) &\coloneqq ( f,  \nabla \cdot \pmb{\psi} + \gamma v )_{\ltwo},
\end{align}
we can state the mixed weak least squares formulation: Find $(\pmb{\varphi}, u) \in \productspace$ such that
\begin{equation}
\label{eq:fosls_method_continuous}
  b( (\pmb{\varphi}, u), (\pmb{\psi}, v) ) = F((\pmb{\psi}, v)) ~~~ \forall \, (\pmb{\psi}, v) \in \productspace.
\end{equation}
To see solvability of \eqref{eq:fosls_method_continuous},
let $u \in \HDirOne$ be the unique solution of (\ref{eq:model_problem}). In view of $f \in L^2(\Omega)$ the pair
$(-\nabla u,u)$ is a solution of (\ref{eq:fosls_method_continuous}).
  Uniqueness follows if one can show that
  $b( (\pmb{\varphi}, u), (\pmb{\psi}, v) ) = 0$
  for all
  $(\pmb{\psi}, v) \in \productspace$
  implies $(\pmb{\varphi}, u) = (\pmb{0},0)$.
  To that end we introduce the (yet to be verified) norm $\norm{\cdot}_b$ induced by $b$:
  \begin{equation}
\label{eq:b-norm}
    \norm{ (\pmb{\varphi}, u) }_b \coloneqq \sqrt{   b( (\pmb{\varphi}, u), (\pmb{\varphi}, u) ) }.
  \end{equation}
  A general approach would be to show norm equivalence.
  In our case:
  \begin{equation*}
    \norm{u}_{H^1(\Omega)} + \norm{\pmb{\varphi}}_{\HDiv}
    \lesssim \norm{( \pmb{\varphi}, u)}_b
    \lesssim \norm{u}_{H^1(\Omega)} + \norm{\pmb{\varphi}}_{\HDiv}.
  \end{equation*}
  We will employ methods similar to a duality argument in the following Theorem~\ref{theorem:norm_equivalence} to prove such a norm equivalence.

\begin{theorem}[Norm equivalence]\label{theorem:norm_equivalence}
  For all $(\pmb{\varphi}, u) \in \productspace$ there holds the norm equivalence
  \begin{equation}\label{eq:norm_equivalence}
    \norm{u}_{H^1(\Omega)}^2 + \norm{\pmb{\varphi}}_{\HDiv}^2
    \lesssim b( (\pmb{\varphi}, u), (\pmb{\varphi}, u) )
    \lesssim \norm{u}_{H^1(\Omega)}^2 + \norm{\pmb{\varphi}}_{\HDiv}^2.
  \end{equation}
\end{theorem}

\begin{proof}
  First note that by definition
  \begin{equation*}
    b( (\pmb{\varphi}, u), (\pmb{\varphi}, u) ) =
    \| \underbrace{ \nabla \cdot \pmb{\varphi} + \gamma u}_{\eqqcolon w} \|_{L^2(\Omega)}^2 + \| \underbrace{ \nabla u + \pmb{\varphi} }_{\eqqcolon \pmb{\eta}} \|_{L^2(\Omega)}^2,
  \end{equation*}
  from which the second inequality in \eqref{eq:norm_equivalence} is obvious.
  For the first inequality, we will now split $\pmb{\varphi}$ and $u$ as follows:
  \\
  \\
  \noindent\begin{minipage}{.5\linewidth}
  \begin{equation*}
    \begin{alignedat}{2}
      \nabla \cdot \pmb{\varphi}_1 + \gamma u_1 &= w \qquad  &&\text{in } \Omega,   \\
                   \nabla u_1 + \pmb{\varphi}_1 &= \pmb{0}   &&\text{in } \Omega,   \\
                                            u_1 &= 0         &&\text{on } \Gamma_D, \\
                  \pmb{\varphi}_1 \cdot \pmb{n} &= 0         &&\text{on } \Gamma_N,
    \end{alignedat}
  \end{equation*}
  \end{minipage}%
  \begin{minipage}{.5\linewidth}
    \begin{equation*}
      \begin{alignedat}{2}
        \nabla \cdot \pmb{\varphi}_2 + \gamma u_2 &= 0 \qquad   &&\text{in } \Omega,   \\
                     \nabla u_2 + \pmb{\varphi}_2 &= \pmb{\eta} &&\text{in } \Omega,   \\
                                              u_2 &= 0          &&\text{on } \Gamma_D, \\
                    \pmb{\varphi}_2 \cdot \pmb{n} &= 0          &&\text{on } \Gamma_N,
      \end{alignedat}
    \end{equation*}
  \end{minipage}
  \\
  \\
  with yet to be determined functions $\pmb{\varphi}_1$, $\pmb{\varphi}_2$, $u_1$ and $u_2$.
  We observe that $\pmb{\varphi} = \pmb{\varphi}_1 + \pmb{\varphi}_2$ and $u = u_1+u_2$ since
  the difference solves (\ref{eq:model_problem_first_order_system}) with zero right-hand side, which is only
  solved by the trivial solution.
 % At this point it is not clear yet that $\pmb{\varphi} = \pmb{\varphi}_1 + \pmb{\varphi}_2$ and $u = u_1 + u_2$.
 % However, since the difference would solve system (\ref{eq:model_problem_first_order_system}) with zero right-hand side,
 % which is only solved by the trivial solution,
 % we conclude that this splitting is indeed valid.
  Simply eliminating $\pmb{\varphi}_1$ and $\pmb{\varphi}_2$ in the above equations, we expect $u_1$ and $u_2$ to be solutions to
  \\
  \\
  \noindent\begin{minipage}{.5\linewidth}
  \begin{equation*}
    \begin{alignedat}{2}
                       -\Delta u_1 + \gamma u_1 &= w \qquad  &&\text{in } \Omega,   \\
                                            u_1 &= 0         &&\text{on } \Gamma_D, \\
                                 \partial_n u_1 &= 0         &&\text{on } \Gamma_N,
    \end{alignedat}
  \end{equation*}
  \end{minipage}%
  \begin{minipage}{.5\linewidth}
    \begin{equation*}
      \begin{alignedat}{2}
                         -\Delta u_2 + \gamma u_2 &= - \nabla \cdot \pmb{\eta} \qquad  &&\text{in } \Omega,   \\
                                              u_2 &= 0                              &&\text{on } \Gamma_D, \\
                                   \partial_n u_2 &= 0                              &&\text{on } \Gamma_N,
      \end{alignedat}
    \end{equation*}
  \end{minipage}
  \\
  \\
  where $- \nabla \cdot \pmb{\eta}$ is to be understood as an element of $(\HDirOne)^\prime$ given by $F: v \mapsto (\pmb{\eta}, \nabla v)_{\ltwo}$.
  Both equations are therefore uniquely solvable.
  This then determines the desired functions $u_1$, $u_2$ and consequently the functions $\pmb{\varphi}_1$, $\pmb{\varphi}_2$, using the second equation in the first order systems.

  Let us show that $(\pmb{\varphi}_1,u_1)$ solves the above system. By construction it satisfies the differential equations and
  furthermore, since $\pmb{\varphi}_1 = - \nabla u_1$, we have by standard regularity theory
  $\pmb{\varphi}_1 \cdot \pmb{n} = - \nabla u_1 \cdot \pmb{n} = - \partial_n u_1 = 0$.

  Let us show that $(\pmb{\varphi}_2,u_2)$ satisfies the above system.
  Let $v \in C_0^\infty(\Omega)$ be arbitrary. Integration by parts and exploiting the weak formulation gives
  \begin{equation*}
    (\nabla \cdot \pmb{\varphi}_2, v)_{\ltwo}
    = -(\pmb{\varphi}_2, \nabla v)_{\ltwo}
    = -(\pmb{\eta}, \nabla v)_{\ltwo} + (\nabla u_2, \nabla v)_{\ltwo}
    = -(\gamma u_2, v)_{\ltwo}.
  \end{equation*}
  Therefore the $\div$-equation is satisfied.
  To verify the boundary conditions we calculate for any $v \in \HDirOne$
  \begin{align*}
    \langle \pmb{\varphi}_2 \cdot \pmb{n} , v \rangle_{H^{-1/2}(\Gamma) \times H^{1/2}(\Gamma)}
    &= (\pmb{\varphi}_2 , \nabla v)_{\ltwo} + (\nabla \cdot \pmb{\varphi}_2 , v)_{\ltwo} \\
    &= (- \nabla u_2 + \pmb{\eta} , \nabla v)_{\ltwo} + (\nabla \cdot \pmb{\varphi}_2 , v)_{\ltwo}  = 0,
  \end{align*}
  where we first used Green's theorem, then the equations of the first order system and at last the weak formulation for $u_2$.
  The \textit{a priori} estimate of the Lax-Milgram theorem gives
  \begin{align*}
    \norm{u_1}_{H^1(\Omega)} &\lesssim \norm{w}_{(\HDirOne)^\prime} \leq \norm{w}_{L^2(\Omega)}, \\
    \norm{u_2}_{H^1(\Omega)} &\lesssim \norm{F}_{(\HDirOne)^\prime} \leq \norm{\pmb{\eta}}_{L^2(\Omega)}.
  \end{align*}
  Due to the splitting $u = u_1 + u_2$ it is now obvious that
  \begin{equation*}
    \norm{u}_{H^1(\Omega)}^2 \lesssim \norm{w}_{L^2(\Omega)}^2 + \norm{\pmb{\eta}}_{L^2(\Omega)}^2.
  \end{equation*}
  We now estimate the $\HDiv$ norms of $\pmb{\varphi}_1$ and $\pmb{\varphi}_2$ as follows
  \begin{align*}
    \norm{ \pmb{\varphi}_1}_{ \HDiv }^2 &=
    \norm{ \pmb{\varphi}_1 }_{L^2(\Omega)}^2 + \norm{ \nabla \cdot \pmb{\varphi}_1 }_{L^2(\Omega)}^2 =
    \norm{ - \nabla u_1 }_{L^2(\Omega)}^2 + \norm{ w-\gamma u_1 }_{L^2(\Omega)}^2 \lesssim
    \norm{ w }_{L^2(\Omega)}^2,
    \\
    \norm{ \pmb{\varphi}_2}_{ \HDiv }^2 &=
    \norm{ \pmb{\varphi}_2 }_{L^2(\Omega)}^2 + \norm{ \nabla \cdot \pmb{\varphi}_2 }_{L^2(\Omega)}^2 =
    \norm{ \pmb{\eta} - \nabla u_2 }_{L^2(\Omega)}^2 + \norm{ - \gamma u_2  }_{L^2(\Omega)}^2 \lesssim
    \norm{ \pmb{\eta} }_{L^2(\Omega)}^2,
  \end{align*}
  which completes the proof.
\end{proof}

\begin{remark}
  Theorem~\ref{theorem:norm_equivalence} (norm equivalence) does not hold on all of $\pmb{H}(\Omega, \div) \times H^1(\Omega)$ since one can construct non-trivial solutions to the system
  \begin{equation*}
    \begin{alignedat}{2}
      \nabla \cdot \pmb{\varphi} + \gamma u &= 0 \qquad   &&\text{in } \Omega,   \\
                   \nabla u + \pmb{\varphi} &= \pmb{0}    &&\text{in } \Omega,
    \end{alignedat}
  \end{equation*}
  due to the missing boundary conditions, even though $\norm{(\pmb{\varphi}, u)}_b = 0$ by construction.
\eremk
\end{remark}

\begin{remark}\label{remark:weaker_coercivity_estimate}
  Theorem~\ref{theorem:norm_equivalence} (norm equivalence)
  is in fact much stronger than what we need to establish unique solvability of the system (\ref{eq:fosls_method_continuous}):
  The weaker coercivity estimate
%  \begin{equation*}
$
\displaystyle     \norm{u}_{L^2(\Omega)}^2 + \norm{\pmb{\varphi}}_{L^2(\Omega)}^2 \lesssim b( (\pmb{\varphi}, u), (\pmb{\varphi}, u) )
$
%  \end{equation*}
  suffices to establish uniqueness.
\eremk
\end{remark}

\begin{remark}
  In the literature there are two main ideas for showing unique solvability when working in a least squares setting concerning a first order system derived from a second order equation:
  \begin{itemize}
    \item The first one deduces solvability from the second order equation and uses some weaker coercivity estimates to establish uniqueness,
as sketched in Remark~\ref{remark:weaker_coercivity_estimate}.  See also \cite{chen-qiu17, bernkopf-melenk19} for these kind of arguments for the Helmholtz equation.
    \item The second approach is to establish a stronger coercivity estimate as in Theorem~\ref{theorem:norm_equivalence} and directly apply the Lax-Milgram theorem to (\ref{eq:fosls_method_continuous}),
    where the right-hand side is a suitable continuous linear functional.
    See also \cite{cai-lazarov-manteufel-mccormick94, cai-manteufel-mccormick97} concerning the model problem in question and also \cite{cai-manteufel-mccormick97stokes} for the Stokes equation.
\eremk
  \end{itemize}

\end{remark}

%% file: 02_duality_argument.tex
% !TEX root = main.tex
\section{Duality argument}\label{section:duality_argument}

The current section is devoted to duality arguments
that are later used for the analysis of the $L^2(\Omega)$ norms of $u - u_h$, $\nabla (u - u_h)$, and  $\pmb{\varphi} - \pmb{\varphi}_h$.
Since these duality arguments rely heavily on the elliptic shift theorem, we restrict ourself to either the pure Neumann or Dirichlet boundary conditions,
i.e., $\Gamma = \Gamma_N$ or $\Gamma = \Gamma_D$ respectively.
In contrast, when considering mixed boundary conditions one has to expect a singularity at the interface between the Dirichlet and Neumann condition, which has
to be properly accounted for in the numerical analysis by graded meshes for both the primal and dual problem.
This is beyond the scope of the present work.
Our overall agenda is to derive regularity results for the dual solutions, always denoted by $(\pmb{\psi}, v)$.
For $w \in \HOne$ and $\eta \in \HZeroDiv$ we prove the existence of dual solutions such that:
\begin{itemize}
  \item $\norm{w}_{L^2(\Omega)}^2 = b( (\pmb{\varphi}, w), (\pmb{\psi}, v) )$, see Theorem~\ref{theorem:duality_argument},
  \item $\norm{\nabla w}_{L^2(\Omega)}^2 = b( (\pmb{\varphi}, w), (\pmb{\psi}, v) )$, see Theorem~\ref{theorem:duality_argument_grad_u},
  \item $\norm{\pmb{\eta}}_{L^2(\Omega)}^2 = b( (\pmb{\eta}, u), (\pmb{\psi}, v) )$, see Theorem~\ref{theorem:duality_argument_phi}.
\end{itemize}
These results are exploited in Section~\ref{section:error_analysis} with the special choices of $w = u - u_h$ and $\pmb{\eta} = \pmb{\varphi} - \pmb{\varphi}_h$, respectively.

\begin{theorem}[Duality argument for the scalar variable]\label{theorem:duality_argument}
  Let $\Gamma$ be smooth. Then there holds:
  \begin{enumerate}[label=(\roman*)]
    \item\label{item:duality_argument_neumann} For $\Gamma = \Gamma_N$ and any $(\pmb{\varphi}, w) \in \HZeroDiv \times \HOne$ there exists $(\pmb{\psi}, v) \in \HZeroDiv \times \HOne$
    such that $\norm{w}_{L^2(\Omega)}^2 = b( (\pmb{\varphi}, w), (\pmb{\psi}, v) )$.
    Furthermore, $\pmb{\psi} \in \pmb{H}^3(\Omega)$, $\nabla \cdot \pmb{\psi} \in H^2(\Omega)$, and $v \in H^2(\Omega)$.
    Additionally the following estimates hold:
    \begin{align*}
      \norm{v}_{H^2(\Omega)}                       &\lesssim \norm{w}_{L^2(\Omega)}, \\
      \norm{\pmb{\psi}}_{H^3(\Omega)}              &\lesssim \norm{w}_{L^2(\Omega)}, \\
      \norm{\nabla \cdot \pmb{\psi}}_{H^2(\Omega)} &\lesssim \norm{w}_{L^2(\Omega)}.
    \end{align*}
    \item\label{item:duality_argument_dirichlet} For $\Gamma = \Gamma_D$ and any $(\pmb{\varphi}, w) \in \HDiv \times \HZeroOne$ there exists $(\pmb{\psi}, v) \in \HDiv \times \HZeroOne$
    such that $\norm{w}_{L^2(\Omega)}^2 = b( (\pmb{\varphi}, w), (\pmb{\psi}, v) )$.
    The same regularity results and estimates as in \ref{item:duality_argument_neumann} hold.
  \end{enumerate}
\end{theorem}

\begin{proof}
We prove \ref{item:duality_argument_neumann}.
Theorem~\ref{theorem:norm_equivalence} give the existence of a unique
  $(\pmb{\psi}, v) \in  \HZeroDiv \times \HOne$ satisfying
  \begin{equation}
\label{eq:theorem:duality_argument-10}
    (u, w)_{\ltwo} = b( (\pmb{\varphi}, u), (\pmb{\psi}, v) ) ~~~ \forall \, (\pmb{\varphi}, u) \in  \HZeroDiv \times \HOne.
  \end{equation}
For the regularity assertions, we introduce the auxiliary functions $z$ and $\pmb{\mu}$ by
  \begin{equation}\label{eq:actual_dual_problem_u}
    \begin{alignedat}{2}
      \nabla \cdot \pmb{\psi} + \gamma v   &= z            \qquad  &\text{in } \Omega, \\
      \nabla v + \pmb{\psi}                &= \pmb{\mu}    \qquad  &\text{in } \Omega, \\
      \pmb{\psi} \cdot \pmb{n}             &= 0            \qquad  &\text{on } \Gamma.
    \end{alignedat}
  \end{equation}

  \textbf{Regularity properties of $z$ and $\pmb{\mu}$:}
  Regularity properties of $z$ are inferred from a scalar elliptic equation satisfied by $z$. To that end, we note that
  (\ref{eq:theorem:duality_argument-10}) is equivalent to
  \begin{equation}\label{eq:axiliary_actual_dual_problem_u}
    ( u, w)_{\ltwo} = (\nabla u + \pmb{\varphi}, \pmb{\mu})_{\ltwo} + (\nabla \cdot \pmb{\varphi} + \gamma u, z)_{\ltwo} ~~~ \forall \, (\pmb{\varphi}, u) \in  \HZeroDiv \times \HOne.
  \end{equation}
  For $u = 0$ and integrating by parts we find
  \begin{equation*}
    0 = (\pmb{\varphi}, \pmb{\mu})_{\ltwo} + (\nabla \cdot \pmb{\varphi}, z)_{\ltwo} = (\pmb{\varphi}, \pmb{\mu} - \nabla z)_{\ltwo}
\qquad \forall \pmb{\varphi} \in \HZeroDiv,
  \end{equation*}
  which gives $z \in H^1(\Omega)$ as well as $\pmb{\mu} = \nabla z$. Inserting $\pmb{\mu} = \nabla z$ and setting $\pmb{\varphi} = 0$
  in (\ref{eq:axiliary_actual_dual_problem_u}) we find
  \begin{equation*}
    ( u, w)_{\ltwo} = (\nabla u, \nabla z)_{\ltwo} + (\gamma u, z)_{\ltwo} ~~~ \forall \, u \in \HOne.
  \end{equation*}
  Therefore $z$ satisfies, in strong form,
  \begin{equation}\label{eq:auxiliary_dual_problem}
  \begin{alignedat}{2}
    - \Delta z + \gamma z &= w \quad   &&\text{in } \Omega,\\
             \partial_n z &= 0         &&\text{on } \Gamma,
  \end{alignedat}
  \end{equation}
  and the shift theorem immediately give $z \in H^2(\Omega)$ with the estimate $\norm{z}_{H^2(\Omega)} \lesssim \norm{w}_{L^2(\Omega)}$.

  \textbf{Regularity properties of $v$:}
  Eliminating $\pmb{\psi}$ in (\ref{eq:actual_dual_problem_u}), we discover that $v$ satisfies
  %\begin{equation*}
  %  -\Delta v + \gamma v = -\Delta z + z  = w + (1-\gamma)z.
  %\end{equation*}
  %Solving now
  \begin{equation}\label{eq:auxiliary_dual_problem-v}
  \begin{alignedat}{2}
    - \Delta v + \gamma v &= w + (1-\gamma)z \quad   &&\text{in } \Omega,\\
             \partial_n v &= 0                       &&\text{on } \Gamma.
  \end{alignedat}
  \end{equation}
  By elliptic regularity $v \in H^2(\Omega)$ with the \textsl{a priori} estimate
  \begin{equation*}
    \norm{v}_{H^2(\Omega)} \lesssim \norm{w + (1-\gamma)z}_{L^2(\Omega)} \lesssim \norm{w}_{L^2(\Omega)}.
  \end{equation*}

  \textbf{Regularity properties of $\pmb{\psi}$:}
  Setting $\pmb{\psi} = \nabla (z-v)$, we have found the desired pair $(\pmb{\psi}, v) \in \HZeroDiv \times \HOne$.
  Since $\pmb{\psi} = \nabla (z-v)$ we first look at the regularity of $z-v$.
  Subtracting the equations (\ref{eq:auxiliary_dual_problem}), (\ref{eq:auxiliary_dual_problem-v})
  satisfied by $z$ and $v$ respectively we obtain
  \begin{equation*}
  \begin{alignedat}{2}
    - \Delta (z-v) + \gamma (z-v) &= (\gamma-1)z     \quad   &&\text{in } \Omega,\\
                 \partial_n (z-v) &= 0                       &&\text{on } \Gamma,
  \end{alignedat}
  \end{equation*}
  which gives $z-v \in H^4(\Omega)$ with the estimate
  \begin{equation*}
    \norm{z-v}_{H^4(\Omega)} \lesssim \norm{(\gamma-1)z}_{H^2(\Omega)} \lesssim \norm{w}_{L^2(\Omega)}.
  \end{equation*}
  We can therefore conclude
  \begin{equation*}
    \norm{\pmb{\psi}}_{H^3(\Omega)} = \norm{ \nabla (z-v) }_{H^3(\Omega)} \leq \norm{z-v}_{H^4(\Omega)} \lesssim \norm{w}_{L^2(\Omega)},
  \end{equation*}
  and since $\nabla \cdot \pmb{\psi} = z - \gamma v$, we have
  \begin{equation*}
    \norm{\nabla \cdot \pmb{\psi}}_{H^2(\Omega)} = \norm{z - \gamma v}_{H^2(\Omega)} \lesssim \norm{w}_{L^2(\Omega)},
  \end{equation*}
  which concludes the proof of \ref{item:duality_argument_neumann}.
  For the Dirichlet case \ref{item:duality_argument_dirichlet} the proof is completely analogous by replacing every Neumann boundary condition with a Dirichlet one.
\end{proof}

\begin{theorem}[Duality argument for the gradient of the scalar variable]\label{theorem:duality_argument_grad_u}
  Let $\Gamma$ be smooth. Then there holds:
  \begin{enumerate}[label=(\roman*)]
    \item\label{item:duality_argument_grad_u_neumann} For $\Gamma = \Gamma_N$ and any $(\pmb{\varphi}, w) \in \HZeroDiv \times \HOne$ there exists $(\pmb{\psi}, v) \in \HZeroDiv \times \HOne$
    such that $\norm{\nabla w}_{L^2(\Omega)}^2 = b( (\pmb{\varphi}, w), (\pmb{\psi}, v) )$.
    Furthermore, $\pmb{\psi} \in \pmb{H}^2(\Omega)$, $\nabla \cdot \pmb{\psi} \in H^1(\Omega)$, and $v \in H^1(\Omega)$.
    Additionally the following estimates hold:
    \begin{align*}
      \norm{v}_{H^1(\Omega)}                       &\lesssim \norm{\nabla w}_{L^2(\Omega)}, \\
      \norm{\pmb{\psi}}_{H^2(\Omega)}              &\lesssim \norm{\nabla w}_{L^2(\Omega)}, \\
      \norm{\nabla \cdot \pmb{\psi}}_{H^1(\Omega)} &\lesssim \norm{\nabla w}_{L^2(\Omega)}.
    \end{align*}
    \item\label{item:duality_argument_grad_u_dirichlet} For $\Gamma = \Gamma_D$ and any $(\pmb{\varphi}, w) \in \HDiv \times \HZeroOne$ there exists $(\pmb{\psi}, v) \in \HDiv \times \HZeroOne$
    such that  $\norm{\nabla w}_{L^2(\Omega)}^2 = b( (\pmb{\varphi}, w), (\pmb{\psi}, v) )$.
    The same regularity results and estimates as in \ref{item:duality_argument_grad_u_neumann} hold.
  \end{enumerate}
\end{theorem}

\begin{proof}
We prove  \ref{item:duality_argument_grad_u_neumann}.
Theorem~\ref{theorem:norm_equivalence} give the existence of a unique
  $(\pmb{\psi}, v) \in  \HZeroDiv \times \HOne$ satisfying
  \begin{equation}
  \label{eq:actual_dual_problem_grad_u-10}
    (\nabla u, \nabla w)_{\ltwo} = b( (\pmb{\varphi}, u), (\pmb{\psi}, v) ) ~~~ \forall \, (\pmb{\varphi}, u) \in  \HZeroDiv \times \HOne.
  \end{equation}
For the regularity assertion, we introduce the auxiliary functions $z$ and $\pmb{\mu}$ by
  \begin{equation}\label{eq:actual_dual_problem_grad_u}
    \begin{alignedat}{2}
      \nabla \cdot \pmb{\psi} + \gamma v   &= z            \qquad  &\text{in } \Omega, \\
      \nabla v + \pmb{\psi}                &= \pmb{\mu}    \qquad  &\text{in } \Omega, \\
      \pmb{\psi} \cdot \pmb{n}             &= 0            \qquad  &\text{on } \Gamma.
    \end{alignedat}
  \end{equation}

\textbf{Regularity properties of $z$ and $\pmb{\mu}$:} We note that
  (\ref{eq:actual_dual_problem_grad_u-10}) is equivalent to
  \begin{equation}\label{eq:axiliary_actual_dual_problem_grad_u}
    (\nabla u, \nabla w)_{\ltwo} = (\nabla u + \pmb{\varphi}, \pmb{\mu})_{\ltwo} + (\nabla \cdot \pmb{\varphi} + \gamma u, z)_{\ltwo} ~~~ \forall \, (\pmb{\varphi}, u) \in  \HZeroDiv \times \HOne.
  \end{equation}
  For $u = 0$ and integrating by parts we find
  \begin{equation*}
    0 = (\pmb{\varphi}, \pmb{\mu})_{\ltwo} + (\nabla \cdot \pmb{\varphi}, z)_{\ltwo} = (\pmb{\varphi}, \pmb{\mu} - \nabla z)_{\ltwo}
  \end{equation*}
  which gives $\pmb{\mu} = \nabla z$. Inserting $\pmb{\mu} = \nabla z$ and setting $\pmb{\varphi} = 0$ in (\ref{eq:axiliary_actual_dual_problem_grad_u})
we find
  \begin{equation*}
    (\nabla u, \nabla w)_{\ltwo} = (\nabla u, \nabla z)_{\ltwo} + (\gamma u, z)_{\ltwo} ~~~ \forall \, u \in \HOne,
  \end{equation*}
  which can be solved for $z \in \HOne$ with the \textsl{a priori} estimate $\norm{z}_{H^1(\Omega)} \lesssim \norm{\nabla w}_{L^2(\Omega)}$.
  Formally, $z$ satisfies
  \begin{equation}\label{eq:auxiliary_dual_problem_grad_u}
  \begin{alignedat}{2}
    - \Delta z + \gamma z &= - \nabla \cdot \nabla w \quad   &&\text{in } \Omega,\\
             \partial_n z &= 0                               &&\text{on } \Gamma.
  \end{alignedat}
  \end{equation}
  where $- \nabla \cdot \nabla w \in (H^1(\Omega))^\prime$ is to be understood as the mapping $ u  \mapsto (\nabla u, \nabla w)_{\ltwo}$.

  \textbf{Regularity of $v$:} Eliminating $\pmb{\psi}$ from (\ref{eq:actual_dual_problem_grad_u}) and using
$\pmb{\mu} = \nabla z$, we discover that $v$ satisfies
  \begin{equation*}
  \begin{alignedat}{2}
    - \Delta v + \gamma v &= (1-\gamma)z - \nabla \cdot \nabla w \quad   &&\text{in } \Omega,\\
             \partial_n v &= 0                                           &&\text{on } \Gamma,
  \end{alignedat}
  \end{equation*}
  By the Lax-Milgram theorem we find that $v \in H^1(\Omega)$ as well as
  \begin{equation*}
    \norm{v}_{H^1(\Omega)} \lesssim \norm{(1-\gamma)z - \nabla \cdot \nabla w}_{(H^1(\Omega))^\prime} \lesssim \norm{\nabla w}_{L^2(\Omega)}.
  \end{equation*}

  \textbf{Regularity of $\pmb{\psi}$:}
  Upon setting $\pmb{\psi} = \nabla (z-v)$, we have found the solution $(\pmb{\psi}, v) \in \HZeroDiv \times \HOne$ of
  (\ref{eq:actual_dual_problem_grad_u-10}).
  To prove the estimates and regularity results for $\pmb{\psi}$ first note that
  \begin{equation*}
  \begin{alignedat}{2}
    - \Delta (z-v) + \gamma (z-v) &= (1-\gamma)z \quad   &&\text{in } \Omega,\\
                 \partial_n (z-v) &= 0                   &&\text{on } \Gamma,
  \end{alignedat}
  \end{equation*}
  and therefore by elliptic regularity $z-v \in H^3(\Omega)$ with the estimate $\norm{z-v}_{H^3(\Omega)} \lesssim \norm{(1-\gamma)z}_{H^1(\Omega)} \lesssim \norm{\nabla w}_{L^2(\Omega)}$.
  Finally since $\pmb{\psi} = \nabla (z-v)$ the regularity assertion for $\pmb{\psi} \in \pmb{H}^2(\Omega)$ follows.
  For the Dirichlet case \ref{item:duality_argument_grad_u_dirichlet} the proof is completely analogous by replacing every Neumann boundary condition with a Dirichlet one.
\end{proof}

\begin{theorem}[Duality argument for the vector valued variable]\label{theorem:duality_argument_phi}
  Let $\Gamma$ be smooth. Then there holds:
  \begin{enumerate}[label=(\roman*)]
    \item\label{item:duality_argument_phi_neumann} For $\Gamma = \Gamma_N$ and any $(\pmb{\eta}, u) \in \HZeroDiv \times \HOne$ there exists $(\pmb{\psi}, v) \in \HZeroDiv \times \HOne$
    such that $\norm{\pmb{\eta}}_{L^2(\Omega)}^2 = b( (\pmb{\eta}, u), (\pmb{\psi}, v) )$.
    Furthermore, $\pmb{\psi} \in \pmb{L}^2(\Omega)$, $\nabla \cdot \pmb{\psi} \in H^1(\Omega)$ and $v \in H^3(\Omega)$.
    Additionally the following estimates hold:
    \begin{align*}
      \norm{v}_{H^3(\Omega)}                       &\lesssim \norm{\pmb{\eta}}_{L^2(\Omega)}, \\
      \norm{\pmb{\psi}}_{L^2(\Omega)}              &\lesssim \norm{\pmb{\eta}}_{L^2(\Omega)}, \\
      \norm{\nabla \cdot \pmb{\psi}}_{H^1(\Omega)} &\lesssim \norm{\pmb{\eta}}_{L^2(\Omega)}.
    \end{align*}
    \item\label{item:duality_argument_phi_dirichlet} For $\Gamma = \Gamma_D$ and any $(\pmb{\eta}, u) \in \HDiv \times \HZeroOne$ there exists $(\pmb{\psi}, v) \in \HDiv \times \HZeroOne$
    such that  $\norm{\pmb{\eta}}_{L^2(\Omega)}^2 = b( (\pmb{\eta}, u), (\pmb{\psi}, v) )$.
    The same regularity results and estimates as in \ref{item:duality_argument_phi_neumann} hold.
  \end{enumerate}
\end{theorem}

\begin{proof}
    We prove \ref{item:duality_argument_phi_neumann}.
Theorem~\ref{theorem:norm_equivalence} give the existence of a unique
  $(\pmb{\psi}, v) \in  \HZeroDiv \times \HOne$ such that
  \begin{equation}
  \label{eq:theorem:duality_argument_phi-10}
    (\pmb{\varphi}, \pmb{\eta})_{\ltwo} = b( (\pmb{\varphi}, u), (\pmb{\psi}, v) ) ~~~ \forall \, (\pmb{\varphi}, u) \in  \HZeroDiv \times \HOne.
  \end{equation}
  For the regularity assertions, we introduce the auxiliary
  functions $z$ and $\pmb{\mu}$ by
  \begin{equation}\label{eq:actual_dual_problem_phi}
    \begin{alignedat}{2}
      \nabla \cdot \pmb{\psi} + \gamma v   &= z            \qquad  &\text{in } \Omega, \\
      \nabla v + \pmb{\psi}                &= \pmb{\mu}    \qquad  &\text{in } \Omega, \\
      \pmb{\psi} \cdot \pmb{n}             &= 0            \qquad  &\text{on } \Gamma.
    \end{alignedat}
  \end{equation}

  \textbf{Regularity of $z$ and $\pmb{\mu}$}:
  (\ref{eq:theorem:duality_argument_phi-10})
  is equivalent to
  \begin{equation}\label{eq:axiliary_actual_dual_problem_phi}
    (\pmb{\varphi}, \pmb{\eta})_{\ltwo} = (\nabla u + \pmb{\varphi}, \pmb{\mu})_{\ltwo} + (\nabla \cdot \pmb{\varphi} + \gamma u, z)_{\ltwo} ~~~ \forall \, (\pmb{\varphi}, u) \in  \HZeroDiv \times \HOne.
  \end{equation}
  For $u = 0$ and integrating by parts we find
  \begin{equation*}
    (\pmb{\varphi}, \pmb{\eta})_{\ltwo} = (\pmb{\varphi}, \pmb{\mu})_{\ltwo} + (\nabla \cdot \pmb{\varphi}, z)_{\ltwo} = (\pmb{\varphi}, \pmb{\mu} - \nabla z)_{\ltwo}
  \end{equation*}
  which gives $\pmb{\mu} - \nabla z = \pmb{\eta}$. Inserting $\pmb{\mu} = \pmb{\eta} + \nabla z$ and setting $\pmb{\varphi} = 0$ in
(\ref{eq:theorem:duality_argument_phi-10}) we find
  \begin{equation*}
    0 = (\nabla u, \pmb{\eta} + \nabla z)_{\ltwo} + (\gamma u, z)_{\ltwo} ~~~ \forall \, u \in \HOne.
  \end{equation*}
Hence, with the understanding
  that $\nabla \cdot \pmb{\eta}$ means $u \mapsto (\nabla u, \pmb{\eta})$, the function $z$ solves
  \begin{equation}\label{eq:auxiliary_dual_problem_phi}
  \begin{alignedat}{2}
    - \Delta z + \gamma z &= \nabla \cdot \pmb{\eta} \quad   &&\text{in } \Omega,\\
             \partial_n z &= 0                               &&\text{on } \Gamma.
  \end{alignedat}
  \end{equation}
Thus,  $z \in \HOne$ and setting $\pmb{\mu} = \pmb{\eta} + \nabla z$ we find (\ref{eq:axiliary_actual_dual_problem_phi}) to be satisfied.
  Furthermore, note that
  \begin{equation*}
    \norm{z}_{H^1(\Omega)} \lesssim \norm{\nabla \cdot \eta}_{(H^1(\Omega))^\prime} \leq \norm{\eta}_{L^2(\Omega)},
  \end{equation*}
  where the last inequality following from integration by parts and exploiting the boundary condition $\eta \in \HZeroDiv$.

  \textbf{Regularity of $v$:}
  By eliminating $\pmb{\psi}$ we find that $v$ solves
  \begin{equation*}
  \begin{alignedat}{2}
    - \Delta v + \gamma v &= (1-\gamma)z \quad   &&\text{in } \Omega,\\
             \partial_n v &= 0                       &&\text{on } \Gamma,
  \end{alignedat}
  \end{equation*}
  Again by elliptic regularity we find that $v \in H^3(\Omega)$ as well as
  \begin{equation*}
    \norm{v}_{H^3(\Omega)} \lesssim \norm{(1-\gamma)z}_{H^1(\Omega)} \lesssim \norm{\eta}_{L^2(\Omega)}.
  \end{equation*}

  \textbf{Regularity of $\pmb{\psi}$:}
  We have $\pmb{\psi} = \pmb{\eta} + \nabla (z-v)$, and the regularity of $\pmb{\psi}$ follows from that of $z$ of $v$.
  For the Dirichlet case \ref{item:duality_argument_phi_dirichlet} the proof is completely analogous by replacing every Neumann boundary condition with a Dirichlet one.
\end{proof}

%% file: 03_error_analysis.tex
% !TEX root = main.tex
\section{Error analysis}\label{section:error_analysis}

The goal of the present section is to establish optimal convergence rates for an $hp$ version of the FOSLS method for the scalar variable, the gradient of the scalar variable as well as the vector variable, all measured in the $L^2(\Omega)$ norm, as long as the polynomial degree of the other variable is chosen appropriately.
\subsection{Notation, assumptions, and road map of the current section}
Throughout we denote by $( \pmb{\varphi}_h, u_h )$ the least squares approximation of $( \pmb{\varphi} , u )$.
Furthermore, let $e^u = u-u_h$ and $\pmb{e}^{\pmb{\varphi}} = \pmb{\varphi}-\pmb{\varphi}_h$ denote the corresponding error terms.
For simplicity we also assume $\Gamma = \Gamma_N$, i.e., $\Gamma_D = \emptyset$.
Furthermore, $p$ will denote the minimum of the two polynomial degrees $p_s$ and $p_v$, i.e., $p = \min(p_s, p_v)$.
The overall agenda of the present section is as follows:
\begin{enumerate}
  \item
  We start off by proving \cite[Lemma~{3.4}]{bochev-gunzburger05} in an $hp$ setting using our duality argument, i.e., the (in our sense) suboptimal $L^2(\Omega)$ estimate
  \begin{equation*}
    \norm{e^u}_{L^2(\Omega)} \lesssim h/p \norm{( \pmb{e}^{\pmb{\varphi}} , e^u )}_b.
  \end{equation*}
  This is done in Lemma~\ref{lemma:e_u_suboptimal_l2_error_estimate}.
  In Remark~\ref{remark:heuristic_arguments} we present heuristic arguments that suggest the possibility of optimal $L^2(\Omega)$ convergence rates.
  These arguments suggest to construct an $\HZeroDiv$ conforming approximation operator $\IhZero$ with additional orthogonality properties.
  \item
  In Lemma~\ref{lemma:Ih_well_defined} we prove that the operator $\IhZero$ is in fact well defined.
  As a tool of independent interest we derive certain continuous and discrete Helmholtz decompositions in Lemmata~\ref{lemma:helmholtz_decomp} and \ref{lemma:helmholtz_decomp_zero}.
  These decompositions are then used in Lemma~\ref{lemma:properties_of_Ih} to analyze the $L^2(\Omega)$ error of the operator $\IhZero$.
  \item
  Next we prove an $hp$ version of \cite[Lemma~{3.6}]{bochev-gunzburger05} (an $h$ analysis of $\pmb{e}^{\pmb{\varphi}}$ in the $L^2(\Omega)$ norm).
  \item
  In Theorem~\ref{theorem:grad_e_u_optimal_l2_error_estimate} we exploit the results of Lemma~\ref{lemma:convergence_of_dual_solution_grad_u}, which analyzes the convergence rate of the FOSLS approximation of the dual solution for the gradient of the scalar variable, in order to prove new optimal $L^2(\Omega)$ error estimates for $\nabla e^u$.
  \item
  We analyze the convergence rate of the FOSLS approximation of the dual solution in various norms in
  Lemma~\ref{lemma:convergence_of_dual_solution}.
  Finally we prove our main result, Theorem~\ref{theorem:e_u_optimal_l2_error_estimate}, which analyzes the convergence of $e^u$ in the $L^2(\Omega)$ norm.
  \item
  Closing this section we derive Corollary~\ref{corollary:summary_of_estimates_for_f_in_higher_order_sobolev_space},
  which summarizes the results for general right-hand side $f \in H^s(\Omega)$,
  by exploiting the estimates given by the Theorems~\ref{theorem:e_phi_optimal_l2_error_estimate}, \ref{theorem:grad_e_u_optimal_l2_error_estimate} and \ref{theorem:e_u_optimal_l2_error_estimate}
  together with the approximation properties of the employed finite element spaces.
\end{enumerate}

Since we are dealing with smooth boundaries we employ curved elements.
We make the following assumptions on the triangulation.
\begin{assumption}[quasi-uniform regular meshes]\label{assumption:quasi_uniform_regular_meshes}
	Let $\widehat{K}$ be the reference simplex.
  Each element map $F_K \colon \widehat{K} \to K$ can be written as $F_K = R_K \circ A_K$, where $A_K$ is an affine map and the maps $R_K$ and $A_K$ satisfy, for constants $C_\mathrm{affine}, C_\mathrm{metric}, \rho > 0$ independent of $K$:
  \begin{equation*}
  	\begin{alignedat}{2}
        &\norm{A^\prime_K}_{L^\infty( \widehat{K} )}        \leq C_\mathrm{affine} h_K, \qquad &&\norm{ (A^\prime_K)^{-1} }_{L^\infty( \widehat{K} )} \leq C_\mathrm{affine} h^{-1}_K, \\
        &\norm{ (R^\prime_K)^{-1} }_{L^\infty( \tilde{K} )} \leq C_\mathrm{metric},     \qquad &&\norm{ \nabla^n R_K }_{L^\infty( \tilde{K} )} \leq C_\mathrm{metric} \rho^n n! \qquad \forall n \in \mathbb{N}_0.
  	\end{alignedat}
  \end{equation*}
  Here, $\tilde{K} = A_K(\widehat{K})$ and $h_K > 0$ denotes the element diameter.
\end{assumption}

On the reference element $\widehat{K}$ we introduce the Raviart-Thomas and Brezzi-Douglas-Marini elements:
\begin{align*}
  \mathcal{P}_{p}(\widehat{K}) &\coloneqq \mathrm{span}\left\{ \pmb{x}^{\pmb \alpha} \colon |\pmb{\alpha}| \leq p \right\}, \\
	\pmb{\mathrm{BDM}}_{p}(\widehat{K}) &\coloneqq \mathcal{P}_{p}(\widehat{K})^d, \\
  \pmb{\mathrm{RT}}_{p-1}(\widehat{K}) &\coloneqq \left\{ \pmb{p} + \pmb{x}q \colon \pmb{p} \in \mathcal{P}_{p-1}(\widehat{K})^d, q \in \mathcal{P}_{p-1}(\widehat{K}) \right\}.
\end{align*}
Note that trivially $\pmb{\mathrm{RT}}_{p-1}(\widehat{K}) \subset \pmb{\mathrm{BDM}}_p(\widehat{K}) \subset \pmb{\mathrm{RT}}_p(\widehat{K})$.
We also recall the classical Piola transformation, which is the appropriate change of variables for $\HDiv$.
For a function $\pmb{\varphi} : K \to \mathbb{R}^d$ and the element map $F_K \colon \widehat{K} \to K$
its Piola transform $\widehat{\pmb{\varphi}} : \widehat{K} \to \mathbb{R}^d$ is given by
\begin{equation*}
	\widehat{\pmb{\varphi}} = (\det F_K^\prime ) (F_K^\prime)^{-1} \pmb{\varphi} \circ F_K.
\end{equation*}
The spaces $S_p(\mathcal{T}_h)$, $\pmb{\mathrm{BDM}}_p(\mathcal{T}_h)$, and $\pmb{\mathrm{RT}}_{p-1}(\mathcal{T}_h)$ are given by standard transformation and (contravariant) Piola transformation of functions on the reference element:
\begin{align*}
  S_{p}(\mathcal{T}_h) &\coloneqq \left\{ u \in H^1(\Omega) \colon \left.\kern-\nulldelimiterspace{u}\vphantom{\big|} \right|_{K} \circ F_K \in \mathcal{P}_{p}(\widehat{K}) \text{ for all } K \in \mathcal{T}_h \right\}, \\
	\pmb{\mathrm{BDM}}_p(\mathcal{T}_h) &\coloneqq \left\{ \pmb{\varphi} \in \pmb{H}(\operatorname{div}, \Omega) \colon (\det F_K^\prime) (F_K^\prime)^{-1} \left.\kern-\nulldelimiterspace{\pmb{\varphi}}\vphantom{\big|} \right|_{K} \circ F_K \in \pmb{\mathrm{BDM}}_p(\widehat{K}) \text{ for all } K \in \mathcal{T}_h \right\}, \\
  \pmb{\mathrm{RT}}_{p-1}(\mathcal{T}_h) &\coloneqq \left\{ \pmb{\varphi} \in \pmb{H}(\operatorname{div}, \Omega) \colon (\det F_K^\prime) (F_K^\prime)^{-1} \left.\kern-\nulldelimiterspace{\pmb{\varphi}}\vphantom{\big|} \right|_{K} \circ F_K \in \pmb{\mathrm{RT}}_{p-1}(\widehat{K}) \text{ for all } K \in \mathcal{T}_h \right\}.
\end{align*}

For the approximation properties of the $\HDiv$ conforming finite element spaces see \cite[Proposition~{2.5.4}]{boffi-brezzi-fortin13} as a standard reference for non-curved elements and without the $p$-aspect. For an analysis of the $hp$-version under Assumption~\ref{assumption:quasi_uniform_regular_meshes}
we refer to \cite[Section~{4}]{bernkopf-melenk19}.
%-----------------------------------
\subsection{The standard duality argument}
Before formulating various duality arguments,
we recall that the conforming least squares approximation $(\pmb{\varphi}_h,u_h)$
is the best approximation in the $\|\cdot\|_b$ norm:
\begin{equation}
\label{eq:cea-lemma}
\|(\pmb{\varphi} - \pmb{\varphi}_h,u - u_h)\|_{b}  = \min_{\substack{\tilde{u}_h \in \Sp , \\ \tilde{\pmb{\varphi}}_h \in \RTBDMZero}} \| (\pmb{\varphi} - \tilde{\pmb{\varphi}}_h , u - \tilde{u}_h )\|_{b}.
\end{equation}
\begin{lemma}\label{lemma:e_u_suboptimal_l2_error_estimate}
  Let $\Gamma$ be smooth and $( \pmb{\varphi}_h , u_h )$ be the least squares approximation of $( \pmb{\varphi} , u )$. Furthermore, let  $e^u = u-u_h$ and $\pmb{e}^{\pmb{\varphi}} = \pmb{\varphi}-\pmb{\varphi}_h$. Then,
  for any $\tilde{u}_h \in \Sp$, $\tilde{\pmb{\varphi}}_h \in \RTBDMZero$,
  \begin{align*}
    \norm{e^u}_{L^2(\Omega)}
    &\lesssim  \frac{h}{p}  \| ( \pmb{e}^{\pmb{\varphi}} , e^u ) \|_b \\
    &\lesssim \frac{h}{p}  \norm{u - \tilde{u}_h}_{H^1(\Omega)}
      + \frac{h}{p}  \norm{\pmb{\varphi} - \tilde{\pmb{\varphi}}_h}_{L^2(\Omega)}
      + \frac{h}{p}  \norm{\nabla \cdot (\pmb{\varphi} - \tilde{\pmb{\varphi}}_h)}_{L^2(\Omega)}.
  \end{align*}
\end{lemma}

\begin{proof}
  Apply Theorem~\ref{theorem:duality_argument} (duality argument for the scalar variable) with $w = e^u$.
  For any $\tilde{v}_h \in \Sp$, $\tilde{\pmb{\psi}}_h \in \RTBDMZero$,
  we find due to the Galerkin orthogonality and the Cauchy-Schwarz inequality:
  \begin{equation}\label{eq:galerkin_orthogonality_in_lemma}
    \begin{alignedat}{1}
      \norm{e^u}_{L^2(\Omega)}^2
      &= b( ( \pmb{e}^{\pmb{\varphi}} , e^u ) , ( \pmb{\psi} , v ) ) \\
      &= b( ( \pmb{e}^{\pmb{\varphi}} , e^u ) , ( \pmb{\psi} - \tilde{\pmb{\psi}}_h , v - \tilde{v}_h ) ) \\
      &\leq \| ( \pmb{e}^{\pmb{\varphi}} , e^u ) \|_b \| ( \pmb{\psi} - \tilde{\pmb{\psi}}_h , v - \tilde{v}_h ) \|_b.
    \end{alignedat}
  \end{equation}
  Using Theorem~\ref{theorem:norm_equivalence} (norm equivalence), and exploiting the regularity results and estimates of Theorem~\ref{theorem:duality_argument}
  as well as the $\HOne$ and $\HDiv$ conforming operators in \cite{melenk-rojik18}, we can find $\tilde{v}_h \in \Sp$, $\tilde{\pmb{\psi}}_h \in \RTBDMZero$, such that
  \begin{align*}
    \| ( \pmb{\psi} - \tilde{\pmb{\psi}}_h , v - \tilde{v}_h ) \|_b
    &\lesssim \| v - \tilde{v}_h \|_{H^1(\Omega)} + \| \pmb{\psi} - \tilde{\pmb{\psi}}_h \|_{\HDiv} \\
    &\lesssim h/p \left( \norm{v}_{H^2(\Omega)} + \norm{\pmb{\psi}}_{\pmb{H}^1(\Omega, \div)} \right) \\
    &\lesssim h/p \norm{e^u}_{L^2(\Omega)},
  \end{align*}
  where we exploited the regularity for $(\pmb{\psi}, v)$ and the \textsl{a priori} estimates of Theorem~\ref{theorem:duality_argument}, which proves the first estimate.
  The second one follows by the fact that the least squares solution is the projection with respect to the scalar product $b$.
  Therefore $\| ( \pmb{e}^{\pmb{\varphi}} , e^u ) \|_b \leq \| ( \pmb{\varphi} - \tilde{\pmb{\varphi}}_h , u - \tilde{u}_h ) \|_b$.
  The result follows by applying the norm equivalence given in Theorem~\ref{theorem:norm_equivalence}.
\end{proof}

\begin{remark}[Heuristic arguments for improved $L^2(\Omega)$ convergence]\label{remark:heuristic_arguments}
  We present an argument why improved convergence of the scalar variable $u$ can be expected.
  We again start by applying our duality argument and exploit the Galerkin orthogonality as in (\ref{eq:galerkin_orthogonality_in_lemma}) in the proof of Lemma~\ref{lemma:e_u_suboptimal_l2_error_estimate}.
  Instead of immediately applying the Cauchy-Schwarz inequality we investigate the terms in the $b$ scalar product
  and analyze the best rate we can expect from the regularity of the dual problem:
  \begin{align*}
    \norm{e^u}_{L^2(\Omega)}^2
    &= b( ( \pmb{e}^{\pmb{\varphi}} , e^u ), ( \pmb{\psi} - \tilde{\pmb{\psi}}_h , v - \tilde{v}_h ) ) \\
    &=
    ( \underbrace{\nabla \cdot \pmb{e}^{\pmb{\varphi}} + \gamma e^u}_{\frownie{}} ,
      \underbrace{\nabla \cdot (\pmb{\psi} - \tilde{\pmb{\psi}}_h)}_{\sim h^2} + \gamma \underbrace{(v - \tilde{v}_h)}_{\sim h^2} )_{\ltwo}
    +
    ( \underbrace{\nabla e^u + \pmb{e}^{\pmb{\varphi}}}_{\frownie{}} ,
      \underbrace{\nabla (v - \tilde{v}_h)}_{\sim h} + \underbrace{\pmb{\psi} - \tilde{\pmb{\psi}}_h}_{\sim h^3}  )_{\ltwo}. \\
  \end{align*}
  Note that the terms are not equilibrated and we cannot expect any rate from the terms marked by $\frownie{}$.
  However choosing $( \tilde{\pmb{\psi}}_h , \tilde{v}_h )$ to be the least squares approximation $( \pmb{\psi}_h , v_h )$ of $( \pmb{\psi} , v )$ and again exploiting the Galerkin orthogonality we have for any
  $( \tilde{\pmb{\varphi}}_h , \tilde{u}_h )$:
  \begin{align*}
    \norm{e^u}_{L^2(\Omega)}^2
    &= b( ( \pmb{e}^{\pmb{\varphi}} , e^u ), ( \pmb{e}^{\pmb{\psi}} , e^v ) ) \\
    &= b( ( \pmb{\varphi} - \tilde{\pmb{\varphi}}_h , u - \tilde{u}_h ), ( \pmb{e}^{\pmb{\psi}} , e^v ) ) \\
    &=
    ( \underbrace{\nabla \cdot (\pmb{\varphi} - \tilde{\pmb{\varphi}}_h)}_{\frownie{}} + \gamma \underbrace{(u - \tilde{u}_h)}_{\sim h^2} ,
      \underbrace{\nabla \cdot \pmb{e}^{\pmb{\psi}}}_{\sim h} + \gamma \underbrace{e^v}_{\sim h^2} )_{\ltwo}
    + ( \underbrace{\nabla (u - \tilde{u}_h)}_{\sim h} + \underbrace{\pmb{\varphi} - \tilde{\pmb{\varphi}}_h}_{\sim h} , \underbrace{\nabla e^v + \pmb{e}^{\pmb{\psi}}}_{\sim h} )_{\ltwo}.
  \end{align*}
  The improved convergence of the dual solution will be shown in Lemma~\ref{lemma:convergence_of_dual_solution}.
  From a best approximation viewpoint the $\nabla \cdot$ term involving $\pmb{\varphi}$ still has no rate.
  To be more precise, the second term has the right powers of $h$ resulting in an overall $h^2$.
  Since the term $\gamma (u - \tilde{u}_h)$ already has order $h^2$ we have no problem with that one.
  The term with the worst rate is
  \begin{equation*}
    ( \nabla \cdot (\pmb{\varphi} - \tilde{\pmb{\varphi}}_h) , \nabla \cdot \pmb{e}^{\pmb{\psi}} )_{\ltwo} \sim h.
  \end{equation*}
  Out of the box we cannot find an extra $h$ to get optimal convergence, even though $\pmb{\psi}$ has far more regularity, which we did not exploit yet.
  We now want to construct an operator $\IhZero$ mapping into the conforming finite element space of the vector variable.
  To exploit the regularity of $\pmb{\psi}$ we insert any $\tilde{\pmb{\psi}}_h \in \RTBDMZero$.
  We have
  \begin{equation*}
    ( \nabla \cdot (\pmb{\varphi} - \IhZero \pmb{\varphi}) , \nabla \cdot \pmb{e}^{\pmb{\psi}} )_{\ltwo} =
    ( \nabla \cdot (\pmb{\varphi} - \IhZero \pmb{\varphi}) , \nabla \cdot ( \pmb{\psi} - \tilde{\pmb{\psi}}_h ) )_{\ltwo} +
    ( \nabla \cdot (\pmb{\varphi} - \IhZero \pmb{\varphi}) , \nabla \cdot ( \tilde{\pmb{\psi}}_h - \pmb{\psi}_h ) )_{\ltwo}.
  \end{equation*}
  Note that $\tilde{\pmb{\psi}}_h - \pmb{\psi}_h$ is a discrete object.
  If we assume $\IhZero$ to satisfy the orthogonality condition
  \begin{equation*}
    ( \nabla \cdot (\pmb{\varphi} - \IhZero \pmb{\varphi}) , \nabla \cdot \pmb{\chi}_h )_{\ltwo} = 0 , \qquad \forall \pmb{\chi}_h \in \RTBDMZero
  \end{equation*}
  we arrive at
  \begin{equation*}
    ( \nabla \cdot (\pmb{\varphi} - \IhZero \pmb{\varphi}) , \nabla \cdot \pmb{e}^{\pmb{\psi}} )_{\ltwo} =
    ( \nabla \cdot (\pmb{\varphi} - \IhZero \pmb{\varphi}) , \underbrace{\nabla \cdot ( \pmb{\psi} - \tilde{\pmb{\psi}}_h )}_{h^2} )_{\ltwo} \sim h^2.
  \end{equation*}
  Therefore the operator $\IhZero$ should satisfy the aforementioned orthogonality condition and have good approximation properties in $L^2(\Omega)$, as needed above.
  In the following we will construct operators $\IhZero$ and $\Ih$ acting on $\HZeroDiv$ and $\HDiv$ respectively.
\eremk
\end{remark}

\noindent
%-----------------------------
\subsection{The operators $\IhZero$ and $\Ih$ }
%-----------------------------
In the spirit of Remark~\ref{remark:heuristic_arguments} a natural choice for the operator $\IhZero$ is the following constrained minimization problem
\begin{equation*}
  \IhZero \pmb{\varphi} = \underset{\pmb{\varphi}_h \in \RTBDMZero }{\mathrm{argmin}} \frac{1}{2}\norm{ \pmb{\varphi} - \pmb{\varphi}_h }_{L^2(\Omega)}^2
  \qquad \text{s.t.} \qquad
  ( \nabla \cdot (\pmb{\varphi} - \IhZero \pmb{\varphi}) , \nabla \cdot \pmb{\chi}_h )_{\ltwo} = 0 \qquad \forall \pmb{\chi}_h \in \RTBDMZero.
\end{equation*}
The corresponding Lagrange function is
\begin{equation*}
  L(\pmb{\varphi}_h, \pmb{\lambda}_h) = \frac{1}{2}\norm{ \pmb{\varphi}_h - \pmb{\varphi} }_{L^2(\Omega)}^2 + ( \nabla \cdot (\pmb{\varphi}_h - \pmb{\varphi}) , \nabla \cdot \pmb{\lambda}_h )_{\ltwo}
\end{equation*}
and the associated saddle point problem is to find $(\pmb{\varphi}_h, \pmb{\lambda}_h) \in \RTBDMZero \times \RTBDMZero$ such that
\begin{subequations}
%\begin{equation*}
%\begin{alignedat}{2}
\begin{alignat}{2}
\label{eq:IhZero-a}
  &( \pmb{\varphi}_h - \pmb{\varphi} , \pmb{\mu}_h )_{\ltwo} + ( \nabla \cdot \pmb{\mu}_h , \nabla \cdot \pmb{\lambda}_h )_{\ltwo}  &&= 0 \qquad \forall \pmb{\mu}_h  \in \RTBDMZero,\\
\label{eq:IhZero-b}
  &( \nabla \cdot ( \pmb{\varphi}_h - \pmb{\varphi} ), \nabla \cdot \pmb{\eta}_h )_{\ltwo}                                                &&= 0 \qquad \forall \pmb{\eta}_h \in \RTBDMZero.
\end{alignat}
%\end{alignedat}
%\end{equation*}
\end{subequations}
Uniqueness is not given since only the divergence of the Lagrange parameter appears.
However, by focussing on the divergence of the Lagrange parameter, we can formulate it in the following way:
Find $(\pmb{\varphi}_h, \lambda_h) \in \RTBDMZero \times \nabla \cdot \RTBDMZero$ such that
\begin{subequations}
%\begin{equation*}
%\begin{alignedat}{3}
\begin{alignat}{3}
\label{eq:Ih-a}
  &( \pmb{\varphi}_h , \pmb{\mu}_h )_{\ltwo} + ( \nabla \cdot \pmb{\mu}_h , \lambda_h )_{\ltwo}  &&= ( \pmb{\varphi} , \pmb{\mu}_h )_{\ltwo} \quad &&\forall \pmb{\mu}_h    \in \RTBDMZero,\\
\label{eq:Ih-b}
  &( \nabla \cdot \pmb{\varphi}_h , \eta_h )_{\ltwo}                                                   &&= ( \nabla \cdot \pmb{\varphi} , \eta_h )_{\ltwo} \qquad    &&\forall \eta_h         \in \nabla \cdot \RTBDMZero.
\end{alignat}
%\end{alignedat}
%\end{equation*}
\end{subequations}
The construction of $\Ih$ is completely analogous, one just drops the zero boundary conditions everywhere.

To see that the operator $\IhZero$ is well-defined,
we have to check the Babu{\v s}ka–Brezzi conditions, see \cite{boffi-brezzi-fortin13}.
Let us first verify solvability on the continuous level.
\\
\noindent
\textbf{Coercivity on the kernel}:
Let $\pmb{\mu} \in \left\{ \pmb{\psi} \in \HZeroDiv \colon (\nabla \cdot \pmb{\psi}, \eta)_{\ltwo} = 0, \forall \eta \in \nabla \cdot \HZeroDiv \right\}$ be given.
The coercivity is trivial since by construction $ (\nabla \cdot \pmb{\mu}, \nabla \cdot \pmb{\mu} )_{\ltwo} = 0 $ and therefore
\begin{equation*}
  (\pmb{\mu}, \pmb{\mu})_{\ltwo} = \norm{\pmb{\mu}}_{L^2(\Omega)}^2 = \norm{\pmb{\mu}}_{L^2(\Omega)}^2 + \norm{\nabla \cdot \pmb{\mu}}_{L^2(\Omega)}^2 = \norm{\pmb{\mu}}_{\HDiv}^2.
\end{equation*}
\\
\noindent
\textbf{inf-sup condition}:
Let $\eta \in \nabla \cdot \HZeroDiv$ be given.
First let $u \in H^1(\Omega)$ with zero average solve
\begin{equation*}
\begin{alignedat}{2}
    - \Delta u &= \eta \quad   &&\text{in } \Omega,\\
  \partial_n u &= 0            &&\text{on } \Gamma.
\end{alignedat}
\end{equation*}
By elliptic regularity we have $\norm{u}_{H^2(\Omega)} \lesssim \norm{\eta}_{L^2(\Omega)}$ and upon defining $\pmb{\mu} = - \nabla u$ we also have $\norm{\pmb{\mu}}_{\HDiv} \lesssim \norm{\eta}_{L^2(\Omega)}$.
Note that by construction $\pmb{\mu} \in \HZeroDiv$ as well as
\begin{equation*}
  (\nabla \cdot \pmb{\mu}, \eta)_{\ltwo} =
  (\eta, \eta)_{\ltwo} =
  \norm{\eta}_{L^2(\Omega)} \norm{\eta}_{L^2(\Omega)} \gtrsim
  \norm{\eta}_{L^2(\Omega)} \norm{\pmb{\mu}}_{\HDiv},
\end{equation*}
which proves the inf-sup condition.
\\
\noindent
\textbf{Coercivity on the kernel - discrete}:
The coercivity is again trivial by the same argument as above.
\\
\noindent
\textbf{inf-sup condition - discrete}:
Let $\lambda_h \in \nabla \cdot \RTBDMZero$ be given.
As above in the continuous case we solve the Poisson problem
\begin{equation*}
\begin{alignedat}{2}
    - \Delta u &= \lambda_h \quad   &&\text{in } \Omega,\\
  \partial_n u &= 0                 &&\text{on } \Gamma.
\end{alignedat}
\end{equation*}
Let $\pmb{\Lambda} = - \nabla u$ and again we have $\norm{\pmb{\Lambda}}_{\HDiv} \leq \norm{\pmb{\Lambda}}_{H^1(\Omega)} \leq \norm{u}_{H^2(\Omega)} \lesssim \norm{\lambda_h}_{L^2(\Omega)}$.
We now employ the commuting projection based interpolation operators defined in \cite{melenk-rojik18},
especially the global operator $\pmb{\Pi}^{\div}_{p}$ given in \cite[Remark~{2.10}]{melenk-rojik18}, see also \cite[Section~{4.8}]{rojik19} in the case $\RTBDMZero = \BDMZero$.
Let therefore $\pmb{\Pi}^{\div, \star}_{p_v}$ denote either the operator $\pmb{\Pi}^{\div}_{p_v - 1}$ if $\RTBDMZero = \RTZero$ or the analogous operator $\pmb{\Pi}^{\div}_{p_v}$ in the case  $\RTBDMZero = \BDMZero$.
We use this operator to project $\pmb{\Lambda}$ onto the conforming subspace.
With $\pmb{\Lambda}_h \coloneqq \pmb{\Pi}^{\div, \star}_{p_v} \pmb{\Lambda}$ we find
\begin{equation*}
  \nabla \cdot \pmb{\Lambda}_h = \nabla \cdot \pmb{\Pi}^{\div, \star}_{p_v} \pmb{\Lambda} = \pmb{\Pi}^{L^2}_{p_v} \nabla \cdot \pmb{\Lambda} =  \pmb{\Pi}^{L^2}_{p_v} \lambda_h = \lambda_h,
\end{equation*}
where $\pmb{\Pi}^{L^2}_{p_v}$ denotes the $L^2$ orthogonal projection on $\nabla \cdot \RTBDMZero$.
Using \cite[Theorem~{2.8 (vi)}]{melenk-rojik18} we can estimate
\begin{equation*}
  \| \pmb{\Lambda} - \pmb{\Pi}^{\div, \star}_{p_v} \pmb{\Lambda} \|_{\HDiv}
  \lesssim \norm{\pmb{\Lambda}}_{H^1(\Omega)}
  \lesssim \norm{\lambda_h}_{L^2(\Omega)},
\end{equation*}
which finally leads to
\begin{equation*}
  \| \pmb{\Lambda}_h \|_{\HDiv}
  = \| \pmb{\Pi}^{\div, \star}_{p_v} \pmb{\Lambda} \|_{\HDiv}
  \lesssim \| \pmb{\Lambda} - \pmb{\Pi}^{\div, \star}_{p_v} \pmb{\Lambda} \|_{\HDiv} + \| \pmb{\Lambda} \|_{\HDiv}
  \lesssim  \norm{\lambda_h}_{L^2(\Omega)}.
\end{equation*}
For any $\lambda_h \in \nabla \cdot \RTBDMZero$ we estimate
\begin{equation*}
  \sup_{\pmb{\varphi}_h \in \RTBDMZero } \frac{(\nabla \cdot \pmb{\varphi}_h, \lambda_h)_{\ltwo}}{\norm{\pmb{\varphi}_h}_{\HDiv} \norm{\lambda_h}_{L^2(\Omega)}}
  \geq \frac{(\nabla \cdot \pmb{\Lambda}_h, \lambda_h)_{\ltwo}}{\norm{\pmb{\Lambda}_h}_{\HDiv} \norm{\lambda_h}_{L^2(\Omega)}}
  = \frac{ \norm{\lambda_h}_{L^2(\Omega)} }{\norm{\pmb{\Lambda}_h}_{\HDiv}}
  \gtrsim 1,
\end{equation*}
which proves the discrete inf-sup condition.
The above arguments can be modified in a straightforward manner when replacing $\RTBDMZero$ with $\RTBDM$ and $\HZeroDiv$ with $\HDiv$.
The only caveate is the fact that one has to replace the homogeneous Neumann boundary condition in the auxiliary problem,
used in the verification of the inf-sup condition, by a homogeneous Dirichlet boundary condition.
We have therefore proven
\begin{lemma}\label{lemma:Ih_well_defined}
For any mesh $\mathcal{T}_h$ satisfying Assumption~\ref{assumption:quasi_uniform_regular_meshes},
the operators $\IhZero: \HZeroDiv \rightarrow \RTBDMZero$ and $\Ih : \HDiv \rightarrow \RTBDM$ are well defined
with bounds independent of the mesh size $h$ and the polynomial degree $p$.
They are projections.
\end{lemma}

We are now going to analyze the approximation properties of the operator $\IhZero$ and $\Ih$ in the $L^2(\Omega)$ norm.
To that end we need certain decompositions on a continuous as well as a discrete level.
\begin{lemma}[Continuous and discrete Helmholtz-like decomposition - no boundary conditions]\label{lemma:helmholtz_decomp}
  The operators
  $\pmb{\Pi}^{\mathrm{curl}} \colon \HDiv \to \nabla \times \HCurl$
  and
  $\pmb{\Pi}^{\mathrm{curl}}_h \colon \RTBDM \to \nabla \times \Nedelec$
  given by
  \begin{align}
    ( \pmb{\Pi}^{\mathrm{curl}}   \pmb{\varphi}  , \nabla \times \pmb{\mu} )_{\ltwo} &= ( \pmb{\varphi}  , \nabla \times \pmb{\mu} )_{\ltwo} \quad \forall \pmb{\mu} \in \HCurl, \label{eq:continuous_helmholtz_decomp} \\
    ( \pmb{\Pi}^{\mathrm{curl}}_h \pmb{\varphi}_h, \nabla \times \pmb{\mu} )_{\ltwo} &= ( \pmb{\varphi}_h, \nabla \times \pmb{\mu} )_{\ltwo} \quad \forall \pmb{\mu} \in \Nedelec \label{eq:discrete_helmholtz_decomp}
  \end{align}
  are well defined.
  Furthermore, the remainder $\pmb{r}$ of the continuous decomposition $\pmb{\varphi} = \pmb{\Pi}^{\mathrm{curl}} \pmb{\varphi} + \pmb{r}$ satisfies
  \begin{equation*}
  \begin{alignedat}{2}
    \nabla \cdot  \pmb{r}  &= \nabla \cdot \pmb{\varphi} \quad   &&\text{in } \Omega, \\
    \nabla \times \pmb{r}  &= 0                                  &&\text{in } \Omega, \\
    \pmb{n} \times \pmb{r} &= 0                                  &&\text{on } \Gamma,
  \end{alignedat}
  \end{equation*}
  as well as $\pmb{r} \in \pmb{H}^1(\Omega)$.
  Additionally there exists $R \in H^2(\Omega) \cap \HZeroOne$ such that $\pmb{r} = \nabla R$, where $R$ satisfies
  \begin{equation*}
  \begin{alignedat}{2}
    \Delta R  &= \nabla \cdot \pmb{\varphi} \quad   &&\text{in } \Omega, \\
            R &= 0                                  &&\text{on } \Gamma.
  \end{alignedat}
  \end{equation*}
  Finally, the estimate $\norm{ R }_{H^2(\Omega)} \lesssim \norm{\pmb{r}}_{H^1(\Omega)} \lesssim \norm{\nabla \cdot \pmb{\varphi}}_{L^2(\Omega)}$ holds.
\end{lemma}

\begin{proof}
  For unique solvability of the variational definition of the operators,
  just note that they are the $L^2(\Omega)$ orthogonal projection on $\nabla \times \HCurl$ and $\nabla \times \Nedelec$ respectively.
  By construction we have
  \begin{equation*}
    (\pmb{r}, \nabla \times \pmb{\mu} )_{\ltwo} = 0 \quad \forall \pmb{\mu} \in \HCurl
  \end{equation*}
  which by definition gives $\nabla \times \pmb{r} = 0$.
  Furthermore, by the characterization of $\HZeroCurl$ given in \cite[Thm.~{3.33}]{monk03} we have $\pmb{n} \times \pmb{r} = 0$.
  Since $\pmb{\Pi}^{\curl} \pmb{\varphi} \in \nabla \times \HCurl$ we immediately have $\nabla \cdot  \pmb{r} = \nabla \cdot \pmb{\varphi}$.
  Exploiting the exact sequence property of the following de Rahm complex
  \begin{equation*}
    \left\{ 0 \right\} \stackrel{\mathrm{id}}{\longrightarrow}
    \HZeroOne          \stackrel{\nabla}{\longrightarrow}
    \HZeroCurl         \stackrel{\nabla \times}{\longrightarrow}
    \HZeroDiv          \stackrel{\nabla \cdot}{\longrightarrow}
    L^2_0(\Omega)      \stackrel{0}{\longrightarrow}
    \left\{ 0 \right\}
  \end{equation*}
  in the case that both $\Omega$ and $\Gamma$ are simply connected,
  we can find $R \in \HZeroOne$ such that $\pmb{r} = \nabla R$.
  Therefore $R$ solves the asserted equation.
  The Friedrichs inequality and elliptic regularity theory then give the desired results.
\end{proof}

By nearly the same arguments we also have a version for zero boundary conditions:

\begin{lemma}[Continuous and discrete Helmholtz-like decomposition - zero boundary conditions]\label{lemma:helmholtz_decomp_zero}
  The operators
  $\pmb{\Pi}^{\mathrm{curl},0} \colon \HZeroDiv \to \nabla \times \HZeroCurl$
  and
  $\pmb{\Pi}^{\mathrm{curl},0}_h \colon \RTBDMZero \to \nabla \times \NedelecZero$
  given by
  \begin{align}
    ( \pmb{\Pi}^{\mathrm{curl},0}   \pmb{\varphi}  , \nabla \times \pmb{\mu} )_{\ltwo} &= ( \pmb{\varphi}  , \nabla \times \pmb{\mu} )_{\ltwo} \quad \forall \pmb{\mu} \in \HZeroCurl \label{eq:continuous_helmholtz_decomp_zero} \\
    ( \pmb{\Pi}^{\mathrm{curl},0}_h \pmb{\varphi}_h, \nabla \times \pmb{\mu} )_{\ltwo} &= ( \pmb{\varphi}_h, \nabla \times \pmb{\mu} )_{\ltwo} \quad \forall \pmb{\mu} \in \NedelecZero \label{eq:discrete_helmholtz_decomp_zero}
  \end{align}
  are well defined.
  Furthermore, the remainder $\pmb{r}$ of the continuous decomposition $\pmb{\varphi} = \pmb{\Pi}^{\mathrm{curl},0} \pmb{\varphi} + \pmb{r}$ satisfies
  \begin{equation*}
    \begin{alignedat}{2}
      \nabla \cdot  \pmb{r}  &= \nabla \cdot \pmb{\varphi} \quad   &&\text{in } \Omega, \\
      \nabla \times \pmb{r}  &= 0                                  &&\text{in } \Omega, \\
      \pmb{r} \cdot \pmb{n}  &= 0                                  &&\text{on } \Gamma,
    \end{alignedat}
  \end{equation*}
  as well as $\pmb{r} \in \pmb{H}^1(\Omega)$.
  Additionally there exists an $R \in H^2(\Omega) \cap H^1(\Omega)/\mathbb{R}$ such that $\pmb{r} = \nabla R$, where $R$ satisfies
  \begin{equation*}
  \begin{alignedat}{2}
    \Delta R     &= \nabla \cdot \pmb{\varphi} \quad   &&\text{in } \Omega, \\
    \partial_n R &= 0                                  &&\text{on } \Gamma.
  \end{alignedat}
  \end{equation*}
  Finally, the estimate $\norm{ R }_{H^2(\Omega)} \lesssim \norm{\pmb{r}}_{H^1(\Omega)} \lesssim \norm{\nabla \cdot \pmb{\varphi}}_{L^2(\Omega)}$ holds.
\end{lemma}

\begin{proof}
  Unique solvability as well as $\nabla \times \pmb{r} = 0$ and $\nabla \cdot  \pmb{r} = \nabla \cdot \pmb{\varphi}$ follows by the same arguments as in the proof of Lemma~\ref{lemma:helmholtz_decomp}.
  Since $\pmb{\varphi} \in \HZeroDiv$ and $\pmb{\Pi}^{\mathrm{curl},0}   \pmb{\varphi} \in \nabla \times \HZeroCurl \subset \HZeroDiv$ we find
  \begin{equation*}
    \pmb{r} \cdot \pmb{n} = \pmb{\varphi} \cdot \pmb{n} - \pmb{\Pi}^{\mathrm{curl},0}   \pmb{\varphi} \cdot \pmb{n} = 0.
  \end{equation*}
  Again by the exact sequence
  \begin{equation*}
    \mathbb{R}       \stackrel{\mathrm{id}}{\longrightarrow}
    \HOne            \stackrel{\nabla}{\longrightarrow}
    \HCurl           \stackrel{\nabla \times}{\longrightarrow}
    \HDiv            \stackrel{\nabla \cdot}{\longrightarrow}
    L^2(\Omega)      \stackrel{0}{\longrightarrow}
    \left\{ 0 \right\}
  \end{equation*}
  we can find $R \in \HOne$ such that $\pmb{r} = \nabla R$.
  Finally since $\partial_n R = \nabla R \cdot \pmb{n} = \pmb{r} \cdot \pmb{n} = 0$,
  we find that $R$ solves the asserted equation.
  The Poincar\'e inequality and elliptic regularity theory then give the desired results.
\end{proof}

\begin{lemma}\label{lemma:properties_of_Ih}
  The operator $\IhZero$ satisfies
  for arbitrary $\tilde{\pmb{\varphi}}_h \in \RTBDMZero$ the estimates
  \begin{align}
\label{lemma:properties_of_Ih-a}
    \norm{\pmb{\varphi} - \IhZero \pmb{\varphi}}_{L^2(\Omega)} &\lesssim \norm{\pmb{\varphi} - \tilde{\pmb{\varphi}}_h}_{L^2(\Omega)} + \frac{h}{p_v} \norm{\nabla \cdot (\pmb{\varphi} - \tilde{\pmb{\varphi}}_h) }_{L^2(\Omega)}, \\
\label{lemma:properties_of_Ih-b}
    \norm{\nabla \cdot (\pmb{\varphi} - \IhZero \pmb{\varphi}) }_{L^2(\Omega)} &\leq \norm{\nabla \cdot (\pmb{\varphi} - \tilde{\pmb{\varphi}}_h) }_{L^2(\Omega)}.
  \end{align}
  The same estimates hold true for the operator $\Ih$ for arbitrary $\tilde{\pmb{\varphi}}_h \in \RTBDM$.
\end{lemma}

\begin{proof}
  Let $\tilde{\pmb{\varphi}}_h \in \RTBDMZero$ be arbitrary.
  Due to the orthogonality relation satisfied by the operator $\IhZero$ the estimate (\ref{lemma:properties_of_Ih-b}) is obvious.
  We have with $\pmb{e} = \pmb{\varphi} - \IhZero \pmb{\varphi}$
  \begin{equation*}
    \norm{\pmb{e}}_{L^2(\Omega)}^2 = (\pmb{e}, \pmb{\varphi} - \tilde{\pmb{\varphi}}_h )_{\ltwo} + (\pmb{e}, \tilde{\pmb{\varphi}}_h - \IhZero \pmb{\varphi})_{\ltwo}.
  \end{equation*}
  In order to treat the second term we apply Lemma~\ref{lemma:helmholtz_decomp_zero} and split the discrete object $\tilde{\pmb{\varphi}}_h - \IhZero \pmb{\varphi} \in \RTBDMZero$ on a discrete and a continuous level. That is,
  \begin{align*}
    \tilde{\pmb{\varphi}}_h - \IhZero \pmb{\varphi} &= \nabla \times \pmb{\mu} + \pmb{r},   \\
    \tilde{\pmb{\varphi}}_h - \IhZero \pmb{\varphi} &= \nabla \times \pmb{\mu}_h + \pmb{r}_h
  \end{align*}
  for certain $\pmb{\mu} \in \HZeroCurl$, $\pmb{r} \in \HZeroDiv$, $\pmb{\mu}_h \in \NedelecZero$, and $\pmb{r}_h \in \RTBDMZero$.
  Since $\nabla \cdot \nabla \times = 0$ we have
  \begin{equation*}
    (\pmb{\varphi} - \IhZero \pmb{\varphi} , \nabla \times \pmb{\mu}_h)_{\ltwo} = 0
  \end{equation*}
  by definition of the operator $\IhZero$ and consequently
  \begin{equation*}
    (\pmb{e}, \tilde{\pmb{\varphi}}_h - \IhZero \pmb{\varphi})_{\ltwo} = (\pmb{e}, \nabla \times \pmb{\mu}_h + \pmb{r}_h )_{\ltwo} = (\pmb{e}, \pmb{r}_h )_{\ltwo} = (\pmb{e}, \pmb{r}_h - \pmb{r} )_{\ltwo} + (\pmb{e}, \pmb{r} )_{\ltwo} \eqqcolon T_1 + T_2.
  \end{equation*}
  \noindent
  \textbf{Treatment of $T_1$}:
  To estimate $T_1$ we first need one of the commuting projection based interpolation operators defined in \cite{melenk-rojik18}.
  Especially the global operator $\pmb{\Pi}^{\div}_{p}$ given in \cite[Remark~{2.10}]{melenk-rojik18}, see also \cite{rojik19}.
  Let therefore $\pmb{\Pi}^{\div, \star}_{p_v}$ denote either the operator $\pmb{\Pi}^{\div}_{p_v - 1}$ if $\RTBDMZero = \RTZero$ or the analogous operator $\pmb{\Pi}^{\div}_{p_v}$ in the case  $\RTBDMZero = \BDMZero$.
  First note that $\nabla \cdot \pmb{r} = \nabla \cdot \pmb{r}_h \in \nabla \cdot \RTBDMZero$.
  By the commuting diagram property of the operator $\pmb{\Pi}^{\div, \star}_{p_v}$ as well as the projection property we therefore have
  \begin{equation*}
    \nabla \cdot (\pmb{\Pi}^{\div, \star}_{p_v} \pmb{r} - \pmb{r}_h) = \pmb{\Pi}^{L^2}_{p_v} (\nabla \cdot \pmb{r}) - \nabla \cdot \pmb{r}_h = \nabla \cdot \pmb{r} - \nabla \cdot \pmb{r}_h = 0.
  \end{equation*}
  By the exact sequence property we therefore have $\pmb{\Pi}^{\div, \star}_{p_v} \pmb{r} - \pmb{r}_h \in \nabla \times \NedelecZero$.
  Furthermore, the definition of $\pmb{r}$ and $\pmb{r}_h$ in Lemma~\ref{lemma:helmholtz_decomp_zero} gives the orthogonality relation $\pmb{r} - \pmb{r}_h \perp \nabla \times \NedelecZero$.
  Putting it all together we have
  \begin{equation*}
    \norm{ \pmb{r} - \pmb{r}_h }_{L^2(\Omega)}^2
    = (\pmb{r} - \pmb{r}_h, \pmb{r} - \pmb{\Pi}^{\div, \star}_{p_v} \pmb{r} )_{\ltwo} + (\pmb{r} - \pmb{r}_h, \pmb{\Pi}^{\div, \star}_{p_v} \pmb{r} - \pmb{r}_h )_{\ltwo}
    = (\pmb{r} - \pmb{r}_h, \pmb{r} - \pmb{\Pi}^{\div, \star}_{p_v} \pmb{r} )_{\ltwo},
  \end{equation*}
  which by the Cauchy-Schwarz inequality gives
  \begin{equation*}
    \norm{ \pmb{r} - \pmb{r}_h }_{L^2(\Omega)} \leq \| \pmb{r} - \pmb{\Pi}^{\div, \star}_{p_v} \pmb{r} \|_{L^2(\Omega)}.
  \end{equation*}
  Since $\nabla \cdot \pmb{r} = \nabla \cdot \pmb{r}_h$ is discrete we may apply \cite[Thm.~{2.8} (vi)]{melenk-rojik18} as well as perform a simple scaling argument to arrive at
  \begin{equation*}
    \| \pmb{r} - \pmb{\Pi}^{\div, \star}_{p_v} \pmb{r} \|_{L^2(\Omega)} \lesssim \frac{h}{p_v} \norm{\pmb{r}}_{H^1(\Omega)}  \lesssim \frac{h}{p_v} \norm{\nabla \cdot (\tilde{\pmb{\varphi}}_h - \IhZero \pmb{\varphi}) }_{L^2(\Omega)},
  \end{equation*}
  where the last estimate is due to the \textsl{a priori} estimate of Lemma~\ref{lemma:helmholtz_decomp_zero}.
  Summarizing we have
  \begin{equation*}
    T_1 \lesssim \frac{h}{p_v} \norm{\pmb{e}}_{L^2(\Omega)} \norm{\nabla \cdot (\tilde{\pmb{\varphi}}_h - \IhZero \pmb{\varphi}) }_{L^2(\Omega)}  \lesssim \frac{h}{p_v} \norm{\pmb{e}}_{L^2(\Omega)} \norm{\nabla \cdot (\pmb{\varphi} - \tilde{\pmb{\varphi}}_h) }_{L^2(\Omega)},
  \end{equation*}
  where the last estimate follows by adding and subtracting $\pmb{\varphi}$, the triangle inequality as well as the second inequality of the present lemma.
  \\
  \noindent
  \textbf{Treatment of $T_2$}:
  The term $T_2$ is treated with a duality argument.
  We select $\pmb{\psi} \in \HDiv$ such that
  \begin{equation*}
    (\nabla \cdot \pmb{v}, \nabla \cdot \pmb{\psi})_{\ltwo} = (\pmb{v}, \pmb{r})_{\ltwo} \qquad \forall \pmb{v} \in \HZeroDiv.
  \end{equation*}
To that end, we note that by Lemma~\ref{lemma:helmholtz_decomp_zero} we have $\pmb{r} = \nabla R$ for some $R \in H^2(\Omega)$.
  Therefore for $\pmb{v} \in \HZeroDiv$ we have
  \begin{equation*}
    (\nabla \cdot \pmb{v}, \nabla \cdot \pmb{\psi})_{\ltwo} =
    (\pmb{v}, \pmb{r})_{\ltwo} =
    (\pmb{v}, \nabla R)_{\ltwo} =
    - (\nabla \cdot \pmb{v}, R)_{\ltwo}
  \end{equation*}
so that the desired $\pmb{\psi}$ is found as $\pmb{\psi}=\nabla w$ with $w$ solving
  \begin{equation*}
  \begin{alignedat}{2}
    - \Delta w &= R  \quad   &&\text{in } \Omega,\\
             w &= 0          &&\text{on } \Gamma. \\
  \end{alignedat}
\end{equation*}
Furthermore, since $R \in H^2(\Omega)$,  elliptic regularity gives $w \in H^4(\Omega)$ and therefore $\pmb{\psi} \in \pmb{H}^3(\Omega)$.
Finally the following estimates hold
\begin{equation}\label{eq:regularity_estimate_psi_p_geq_2}
  \norm{\nabla \cdot \pmb{\psi}}_{H^2(\Omega)}
  \leq \norm{\pmb{\psi}}_{H^3(\Omega)}
  \leq \norm{w}_{H^4(\Omega)}
  \lesssim \norm{R}_{H^2(\Omega)}
  \lesssim \norm{\pmb{r}}_{H^1(\Omega)}
  \lesssim \norm{\nabla \cdot (\tilde{\pmb{\varphi}}_h - \IhZero \pmb{\varphi})}_{L^2(\Omega)},
\end{equation}
due to elliptic regularity and the results of Lemma~\ref{lemma:helmholtz_decomp_zero}.
We therefore have for any $\pmb{\psi}_h \in \RTBDMZero$
\begin{equation*}
  T_2
  = (\pmb{e}, \pmb{r})_{\ltwo}
  = (\nabla \cdot \pmb{e}, \nabla \cdot \pmb{\psi})_{\ltwo}
  = (\nabla \cdot \pmb{e}, \nabla \cdot (\pmb{\psi} - \pmb{\psi}_h))_{\ltwo}
  \leq \norm{\nabla \cdot \pmb{e}}_{L^2(\Omega)} \norm{\nabla \cdot (\pmb{\psi} - \pmb{\psi}_h)}_{L^2(\Omega)},
\end{equation*}
where we used the definition of $T_2$, the duality argument elaborated above, the orthogonality relation of $\IhZero$ to insert any $\pmb{\psi}_h \in \RTBDMZero$, and the Cauchy-Schwarz inequality.
Finally exploiting the \textsl{a priori} estimate of $\pmb{\psi}$ in \eqref{eq:regularity_estimate_psi_p_geq_2} we find for $p_v > 1$ that
\begin{align*}
  T_2
  &\leq
  \norm{\nabla \cdot \pmb{e}}_{L^2(\Omega)} \cdot \inf_{\pmb{\psi}_h \in \RTBDMZero} \norm{\nabla \cdot (\pmb{\psi} - \pmb{\psi}_h)}_{L^2(\Omega)}
  \lesssim  \norm{\nabla \cdot \pmb{e}}_{L^2(\Omega)} (h/p_v)^2 \norm{\nabla \cdot \pmb{\psi}}_{H^2(\Omega)} \\
  &\lesssim \norm{\nabla \cdot \pmb{e}}_{L^2(\Omega)} (h/p_v)^2 \norm{\nabla \cdot (\tilde{\pmb{\varphi}}_h - \IhZero \pmb{\varphi})}_{L^2(\Omega)}.
\end{align*}
In the lowest order case $p_v = 1$ we cannot fully exploit the regularity.
However, we find
\begin{equation}\label{eq:regularity_estimate_psi_p_geq_1}
  \norm{\nabla \cdot \pmb{\psi}}_{H^1(\Omega)}
  \leq \norm{\pmb{\psi}}_{H^2(\Omega)}
  \leq \norm{w}_{H^3(\Omega)}
  \lesssim \norm{R}_{H^1(\Omega)}
  \lesssim \norm{\nabla \cdot (\tilde{\pmb{\varphi}}_h - \IhZero \pmb{\varphi})}_{(H^{1}(\Omega))^\prime},
\end{equation}
Proceeding as above and using estimate \eqref{eq:regularity_estimate_psi_p_geq_1} we find
\begin{align*}
  T_2
  &\leq
  \norm{\nabla \cdot \pmb{e}}_{L^2(\Omega)} \cdot \inf_{\pmb{\psi}_h \in \RTBDMZero} \norm{\nabla \cdot (\pmb{\psi} - \pmb{\psi}_h)}_{L^2(\Omega)}
  \lesssim  \norm{\nabla \cdot \pmb{e}}_{L^2(\Omega)} h/p_v \norm{\nabla \cdot \pmb{\psi}}_{H^1(\Omega)} \\
  &\lesssim \norm{\nabla \cdot \pmb{e}}_{L^2(\Omega)} h/p_v \norm{\nabla \cdot (\tilde{\pmb{\varphi}}_h - \IhZero \pmb{\varphi})}_{(H^{1}(\Omega))^\prime}
  \lesssim \norm{\nabla \cdot \pmb{e}}_{L^2(\Omega)} h/p_v \norm{\tilde{\pmb{\varphi}}_h - \IhZero \pmb{\varphi}}_{L^2(\Omega)}.
\end{align*}
The last last estimate is due to integration by parts and the boundary condition of $\tilde{\pmb{\varphi}}_h - \IhZero \pmb{\varphi}$; in fact
\begin{align*}
  \norm{\nabla \cdot (\tilde{\pmb{\varphi}}_h - \IhZero \pmb{\varphi})}_{(H^{1}(\Omega))^\prime}
  &=
  \sup_{v \in H^1(\Omega)} \frac{|(\nabla \cdot (\tilde{\pmb{\varphi}}_h - \IhZero \pmb{\varphi}), v)_{\ltwo}|}{\norm{v}_{H^1(\Omega)}}
  =
  \sup_{v \in H^1(\Omega)} \frac{|( \tilde{\pmb{\varphi}}_h - \IhZero \pmb{\varphi}, \nabla v)_{\ltwo}|}{\norm{v}_{H^1(\Omega)}} \\
  &\leq
  \norm{\tilde{\pmb{\varphi}}_h - \IhZero \pmb{\varphi}}_{L^2(\Omega)}
\end{align*}
holds.
Putting everything together
we have for $p_v > 1$
\begin{align*}
  \norm{\pmb{e}}_{L^2(\Omega)}^2
  &= (\pmb{e}, \pmb{\varphi} - \tilde{\pmb{\varphi}}_h )_{\ltwo} + (\pmb{e}, \tilde{\pmb{\varphi}}_h - \IhZero \pmb{\varphi})_{\ltwo} \\
  &= (\pmb{e}, \pmb{\varphi} - \tilde{\pmb{\varphi}}_h )_{\ltwo} + T_1 + T_2 \\
  &\lesssim \norm{\pmb{e}}_{L^2(\Omega)} \norm{\pmb{\varphi} - \tilde{\pmb{\varphi}}_h}_{L^2(\Omega)}
    + \frac{h}{p_v} \norm{\pmb{e}}_{L^2(\Omega)} \norm{\nabla \cdot (\pmb{\varphi} - \tilde{\pmb{\varphi}}_h) }_{L^2(\Omega)} \\
    &\qquad + \frac{h^2}{p_v^2} \norm{\nabla \cdot \pmb{e}}_{L^2(\Omega)} \norm{\nabla \cdot (\tilde{\pmb{\varphi}}_h - \IhZero \pmb{\varphi})}_{L^2(\Omega)} \\
  &\lesssim \norm{\pmb{e}}_{L^2(\Omega)} \norm{\pmb{\varphi} - \tilde{\pmb{\varphi}}_h}_{L^2(\Omega)}
    + \frac{h}{p_v} \norm{\pmb{e}}_{L^2(\Omega)} \norm{\nabla \cdot (\pmb{\varphi} - \tilde{\pmb{\varphi}}_h) }_{L^2(\Omega)}
    + \frac{h^2}{p_v^2} \norm{\nabla \cdot (\pmb{\varphi} - \tilde{\pmb{\varphi}}_h)}_{L^2(\Omega)}^2,
\end{align*}
where the last estimate again follows from inserting $\pmb{\varphi}$ and using the second estimate of the present lemma.
Young's inequality then yields the result for the operator $\IhZero$.
The lowest order case is treated analogous.
For the operator $\Ih$ the only difference is that one applies Lemma~\ref{lemma:helmholtz_decomp}
instead of Lemma~\ref{lemma:helmholtz_decomp_zero} and perform the duality argument on all of $\HDiv$ instead of $\HZeroDiv$.
Here it is important to note that the potential $R$ given by Lemma~\ref{lemma:helmholtz_decomp} satisfies homogeneous boundary conditions,
so that the boundary term vanishes in the partial integration.
\end{proof}

\begin{remark}
  $\HDiv$-conforming approximation operators similar to $\Ih$ and $\IhZero$ are presented in
  \cite{ern-gudi-smears-vohralik19}, where the focus is on a patchwise construction rather than the
  (global) orthogonalities (\ref{eq:IhZero-b}), (\ref{eq:Ih-b}).
  \eremk
\end{remark}

\begin{theorem}\label{theorem:e_phi_optimal_l2_error_estimate}
  Let $\Gamma$ be smooth and $( \pmb{\varphi}_h , u_h )$ be the least squares approximation of $( \pmb{\varphi} , u )$. Furthermore, let  $e^u = u-u_h$ and $\pmb{e}^{\pmb{\varphi}} = \pmb{\varphi}-\pmb{\varphi}_h$. Then, for any $\tilde{u}_h \in \Sp$, $\tilde{\pmb{\varphi}}_h \in \RTBDMZero$,
    \begin{align*}
      \norm{\pmb{e}^{\pmb{\varphi}}}_{L^2(\Omega)}
      &\lesssim \frac{h}{p} \| ( \pmb{e}^{\pmb{\varphi}} , e^u ) \|_b
        + \| \pmb{\varphi} - \tilde{\pmb{\varphi}}_h \|_{L^2(\Omega)}
        + \frac{h}{p} \norm{\nabla \cdot (\pmb{\varphi} - \tilde{\pmb{\varphi}}_h) }_{L^2(\Omega)} \\
      &\lesssim
        \frac{h}{p} \norm{u - \tilde{u}_h}_{H^1(\Omega)} +
        \| \pmb{\varphi} - \tilde{\pmb{\varphi}}_h \|_{L^2(\Omega)} +
        \frac{h}{p} \| \nabla \cdot (\pmb{\varphi} - \tilde{\pmb{\varphi}}_h) \|_{L^2(\Omega)}.
    \end{align*}
\end{theorem}

\begin{proof}
  Let $(\pmb{\psi}, v) \in \HZeroDiv \times \HOne$ denote the dual solution given by Theorem~\ref{theorem:duality_argument_phi} applied to $\pmb{\eta} = \pmb{e}^{\pmb{\varphi}}$.
  Theorem~\ref{theorem:duality_argument_phi} gives $\pmb{\psi} \in \pmb{L}^2(\Omega)$, $\nabla \cdot \pmb{\psi} \in H^1(\Omega)$, and $v \in H^3(\Omega)$.
  Due to the Galerkin orthogonality we have for any $( \tilde{\pmb{\psi}}_h , \tilde{v}_h  )$
  \begin{equation*}
      \norm{\pmb{e}^{\pmb{\varphi}}}_{L^2(\Omega)}^2
      = b( ( \pmb{e}^{\pmb{\varphi}} , e^u ) , ( \pmb{\psi} , v ) )
      = b( ( \pmb{e}^{\pmb{\varphi}} , e^u ) , ( \pmb{\psi} - \tilde{\pmb{\psi}}_h , v - \tilde{v}_h ) ).
  \end{equation*}
  We now estimate all terms in the above:
  \begin{equation*}
  \begin{alignedat}{2}
    ( \nabla e^u + \pmb{e}^{\pmb{\varphi}} , \nabla (v - \tilde{v}_h) )_{\ltwo}
      &\leq \| ( \pmb{e}^{\pmb{\varphi}} , e^u ) \|_b  \| \nabla (v - \tilde{v}_h) \|_{L^2(\Omega)}, \\
    ( \nabla \cdot \pmb{e}^{\pmb{\varphi}} + \gamma e^u , \nabla \cdot (\pmb{\psi} - \tilde{\pmb{\psi}}_h) + \gamma (v - \tilde{v}_h) )_{\ltwo}
      &\lesssim \| ( \pmb{e}^{\pmb{\varphi}} , e^u ) \|_b  \left[ \| \nabla \cdot ( \pmb{\psi} - \tilde{\pmb{\psi}}_h )  \|_{L^2(\Omega)} + \norm{ v - \tilde{v}_h }_{L^2(\Omega)} \right], \\
    ( \nabla e^u , \pmb{\psi} - \tilde{\pmb{\psi}}_h )_{\ltwo}
      = - ( e^u , \nabla \cdot ( \pmb{\psi} - \tilde{\pmb{\psi}}_h ) )_{\ltwo}
      &\leq \norm{e^u}_{L^2(\Omega)} \| \nabla \cdot ( \pmb{\psi} - \tilde{\pmb{\psi}}_h  ) \|_{L^2(\Omega)}.
  \end{alignedat}
  \end{equation*}
  Therefore, we conclude that
  \begin{equation}\label{eq:e_phi_intermediate_estimate_1}
      \norm{\pmb{e}^{\pmb{\varphi}}}_{L^2(\Omega)}^2
      \lesssim \| ( \pmb{e}^{\pmb{\varphi}} , e^u ) \|_b  \left[ \| \nabla \cdot ( \pmb{\psi} - \tilde{\pmb{\psi}}_h ) \|_{L^2(\Omega)} + \norm{ v - \tilde{v}_h }_{H^1(\Omega)} \right]
      + (\pmb{e}^{\pmb{\varphi}} , \pmb{\psi} - \tilde{\pmb{\psi}}_h )_{\ltwo},
  \end{equation}
  the limiting term being for now the last one.
  To overcome the lack of regularity of $\pmb{\psi}$ we perform a Helmholtz decomposition.
  In fact since $\pmb{\psi} \in \HZeroDiv$ as well as $\nabla \cdot \pmb{\psi} \in \HOne$
  there exist $\pmb{\rho} \in \HZeroCurl$ and $z \in H^3(\Omega)$ such that $\pmb{\psi} = \nabla \times \pmb{\rho} + \nabla z$.
  The construction is as follows:
  Let $z \in \HOne$ solve
  \begin{equation*}
  \begin{alignedat}{2}
    - \Delta z     &= - \nabla \cdot \pmb{\psi} \quad  &&\text{in } \Omega, \\
      \partial_n z &= 0                                   &&\text{on } \Gamma.
  \end{alignedat}
  \end{equation*}
  Since $\nabla \cdot (\pmb{\psi} - \nabla z) = 0$ as well as $(\pmb{\psi} - \nabla z) \cdot \pmb{n} = 0$ by construction,
  the exact sequence property of the employed spaces allows for the existence of $\pmb{\rho} \in \HZeroCurl$ such that $\pmb{\psi} - \nabla z = \nabla \times \pmb{\rho}$.
  Finally the following estimates hold due to the \textsl{a priori} estimate of the Lax-Milgram theorem and partial integration for the first estimate, elliptic regularity theory for the second, and the triangle inequality together with the first estimate for the third one:
  \begin{align*}
    \norm{z}_{H^1(\Omega)}                        &\lesssim \norm{\nabla \cdot \pmb{\psi}}_{(H^1(\Omega))^\prime} \leq \norm{\pmb{\psi}}_{L^2(\Omega)}, \\
    \norm{z}_{H^3(\Omega)}                        &\lesssim \norm{\nabla \cdot \pmb{\psi}}_{H^1(\Omega)}, \\
    \norm{\nabla \times \pmb{\rho}}_{L^2(\Omega)} &\leq \norm{\pmb{\psi}}_{L^2(\Omega)} + \norm{\nabla z}_{L^2(\Omega)} \lesssim \norm{\pmb{\psi}}_{L^2(\Omega)}.
  \end{align*}
  We now continue estimating (\ref{eq:e_phi_intermediate_estimate_1}) by applying the Helmholtz decomposition.
  For any $\tilde{\pmb{\psi}}_h^c$,  $\tilde{\pmb{\psi}}_h^g \in \RTBDMZero$ we have with $\tilde{\pmb{\psi}}_h = \tilde{\pmb{\psi}}_h^c + \tilde{\pmb{\psi}}_h^g$
  \begin{equation*}
    (\pmb{e}^{\pmb{\varphi}} , \pmb{\psi} - \tilde{\pmb{\psi}}_h )_{\ltwo}
    = (\pmb{e}^{\pmb{\varphi}} , \nabla \times \pmb{\rho} - \tilde{\pmb{\psi}}_h^c )_{\ltwo} + (\pmb{e}^{\pmb{\varphi}} , \nabla z - \tilde{\pmb{\psi}}_h^g )_{\ltwo}
    \eqqcolon T^c + T^g.
  \end{equation*}
  \noindent
  \textbf{Treatment of $T^g$}:
  By the Cauchy-Schwarz inequality we have
  \begin{equation*}
    T^g = (\pmb{e}^{\pmb{\varphi}} , \nabla z - \tilde{\pmb{\psi}}_h^g )_{\ltwo}
    \leq \norm{\pmb{e}^{\pmb{\varphi}}}_{L^2(\Omega)} \| \nabla z - \tilde{\pmb{\psi}}_h^g \|_{L^2(\Omega)}.
  \end{equation*}
  \textbf{Treatment of $T^c$}:
  For any $\tilde{\pmb{\varphi}}_h \in \RTBDMZero$ we have
  \begin{align*}
    T^c
    &= (\pmb{e}^{\pmb{\varphi}} , \nabla \times \pmb{\rho} - \tilde{\pmb{\psi}}_h^c )_{\ltwo} \\
    &= (\pmb{\varphi} - \tilde{\pmb{\varphi}}_h , \nabla \times \pmb{\rho} - \tilde{\pmb{\psi}}_h^c )_{\ltwo}
     + (\tilde{\pmb{\varphi}}_h - \pmb{\varphi}_h , \nabla \times \pmb{\rho} - \tilde{\pmb{\psi}}_h^c )_{\ltwo} \eqqcolon T^c_1 + T^c_2.
  \end{align*}
  \textbf{Treatment of $T^c_1$}:
  By the Cauchy-Schwarz inequality we have
  \begin{equation*}
    T^c_1 = (\pmb{\varphi} - \tilde{\pmb{\varphi}}_h , \nabla \times \pmb{\rho} - \tilde{\pmb{\psi}}_h^c )_{\ltwo}
    \leq \| \pmb{\varphi} - \tilde{\pmb{\varphi}}_h \|_{L^2(\Omega)} \| \nabla \times \pmb{\rho} - \tilde{\pmb{\psi}}_h^c \|_{L^2(\Omega)}.
  \end{equation*}
  \textbf{Treatment of $T^c_2$}:
  In order to treat $T^c_2$ we proceed as in the proof of Lemma~\ref{lemma:properties_of_Ih} and apply Lemma~\ref{lemma:helmholtz_decomp_zero} to split the discrete object $\tilde{\pmb{\varphi}}_h - \pmb{\varphi}_h \in \RTBDMZero$ on a discrete and a continuous level:
  \begin{align*}
    \tilde{\pmb{\varphi}}_h - \pmb{\varphi}_h &= \nabla \times \pmb{\mu} + \pmb{r},   \\
    \tilde{\pmb{\varphi}}_h - \pmb{\varphi}_h &= \nabla \times \pmb{\mu}_h + \pmb{r}_h,
  \end{align*}
  for certain $\pmb{\mu} \in \HZeroCurl$, $\pmb{r} \in \HZeroDiv$, $\pmb{\mu}_h \in \NedelecZero$, and $\pmb{r}_h \in \RTBDMZero$.
  We now choose $\tilde{\pmb{\psi}}_h^c = \pmb{\Pi}^{\mathrm{curl},0}_h \nabla \times \pmb{\rho} $ given by Lemma~\ref{lemma:helmholtz_decomp_zero}.
  Exploiting the definition of the operator $\pmb{\Pi}^{\mathrm{curl},0}_h$ we find
  \begin{align*}
    T^c_2
    &= (\tilde{\pmb{\varphi}}_h - \pmb{\varphi}_h , \nabla \times \pmb{\rho} - \tilde{\pmb{\psi}}_h^c )_{\ltwo} \\
    &= \underbrace{(\nabla \times \pmb{\mu}_h , \nabla \times \pmb{\rho} - \pmb{\Pi}^{\mathrm{curl},0}_h \nabla \times \pmb{\rho} )_{\ltwo}}_{=0}
     + (\pmb{r}_h , \nabla \times \pmb{\rho} - \pmb{\Pi}^{\mathrm{curl},0}_h \nabla \times \pmb{\rho} )_{\ltwo} \\
    &= (\pmb{r}_h - \pmb{r}, \nabla \times \pmb{\rho} - \pmb{\Pi}^{\mathrm{curl},0}_h \nabla \times \pmb{\rho} )_{\ltwo} + (\pmb{r}, \nabla \times \pmb{\rho} - \pmb{\Pi}^{\mathrm{curl},0}_h \nabla \times \pmb{\rho} )_{\ltwo} \\
    &\eqqcolon T_1 + T_2.
  \end{align*}
  \textbf{Treatment of $T_1$}:
  With the same notation as in the proof of Lemma~\ref{lemma:properties_of_Ih} and with exactly the same arguments we have
  \begin{equation*}
    \| \pmb{r} - \pmb{r}_h \|_{L^2(\Omega)} \lesssim \frac{h}{p_v} \norm{\pmb{r}}_{H^1(\Omega)}  \lesssim \frac{h}{p_v} \norm{\nabla \cdot (\tilde{\pmb{\varphi}}_h - \pmb{\varphi}_h) }_{L^2(\Omega)}.
  \end{equation*}
  By the Cauchy-Schwarz inequality we have
  \begin{equation*}
    T_1
    \lesssim \frac{h}{p_v} \norm{\nabla \cdot (\tilde{\pmb{\varphi}}_h - \pmb{\varphi}_h) }_{L^2(\Omega)} \| \nabla \times \pmb{\rho} - \pmb{\Pi}^{\mathrm{curl},0}_h \nabla \times \pmb{\rho} \|_{L^2(\Omega)}
    \lesssim \frac{h}{p_v} \norm{\nabla \cdot (\tilde{\pmb{\varphi}}_h - \pmb{\varphi}_h) }_{L^2(\Omega)} \| \nabla \times \pmb{\rho} \|_{L^2(\Omega)},
  \end{equation*}
  where the last estimate follows from the fact that
  \begin{equation*}
    \| \nabla \times \pmb{\rho} - \pmb{\Pi}^{\mathrm{curl},0}_h \nabla \times \pmb{\rho} \|_{L^2(\Omega)} \leq
    \| \nabla \times \pmb{\rho} - \nabla \times \tilde{\pmb{\rho}}_h \|_{L^2(\Omega)}
  \end{equation*}
  for any $\tilde{\pmb{\rho}}_h \in \NedelecZero$ since it is a projection.
  Finally inserting $\pmb{\varphi}$ and applying the triangle inequality as well as estimating $\norm{\nabla \cdot (\pmb{\varphi} - \pmb{\varphi}_h) }_{L^2(\Omega)}$ by $\| (e^u, \pmb{e}^{\pmb{\varphi}}) \|_b$ we find
  \begin{equation*}
    T_1
    \lesssim \frac{h}{p_v} \norm{\nabla \cdot (\pmb{\varphi} - \tilde{\pmb{\varphi}}_h) }_{L^2(\Omega)} \| \nabla \times \pmb{\rho} \|_{L^2(\Omega)} + \frac{h}{p_v} \| ( \pmb{e}^{\pmb{\varphi}} , e^u ) \|_b \| \nabla \times \pmb{\rho} \|_{L^2(\Omega)}.
  \end{equation*}
  \\
  \noindent
  \textbf{Treatment of $T_2$}: Note again that $\pmb{\rho} \in \HZeroCurl$ and the fact that $\pmb{\Pi}^{\mathrm{curl},0}_h$ maps into $\nabla \times \NedelecZero$.
  Therefore, we can write $\nabla \times \pmb{\rho} - \pmb{\Pi}^{\mathrm{curl},0}_h \nabla \times \pmb{\rho} = \nabla \times \widehat{\pmb{\rho}}$ for some $\widehat{\pmb{\rho}} \in \HZeroCurl$
  and the boundary terms consequently vanish in the following integration by parts
  \begin{equation*}
    T_2
    = (\pmb{r}, \nabla \times \widehat{\pmb{\rho}} )_{\ltwo}
    = (\nabla \times \pmb{r}, \widehat{\pmb{\rho}} )_{\ltwo}.
  \end{equation*}
  Finally, $T_2 = 0$, since
$ \nabla \times \pmb{r} = 0 $ by Lemma~\ref{lemma:helmholtz_decomp_zero}.
  \\
  \noindent
  \textbf{Collecting all the terms}: Collecting the terms together with the estimate $\norm{\nabla \times \pmb{\rho}}_{L^2(\Omega)} \lesssim \norm{\pmb{\psi}}_{L^2(\Omega)} \lesssim \norm{\pmb{e}^{\pmb{\varphi}}}_{L^2(\Omega)}$ from the Helmholtz decomposition and the regularity estimates of Lemma~\ref{theorem:duality_argument_phi} we find
  \begin{equation}\label{eq:e_phi_intermediate_estimate_2}
    \begin{alignedat}{1}
      (\pmb{e}^{\pmb{\varphi}} , \pmb{\psi} - \tilde{\pmb{\psi}}_h )_{\ltwo}
      \lesssim \Bigg[ \| \nabla z &- \tilde{\pmb{\psi}}_h^g \|_{L^2(\Omega)}
       + \| \pmb{\varphi} - \tilde{\pmb{\varphi}}_h \|_{L^2(\Omega)}\\
       &+ \frac{h}{p_v} \norm{\nabla \cdot (\pmb{\varphi} - \tilde{\pmb{\varphi}}_h) }_{L^2(\Omega)}
       + \frac{h}{p_v} \| ( \pmb{e}^{\pmb{\varphi}} , e^u ) \|_b \Bigg] \norm{\pmb{e}^{\pmb{\varphi}}}_{L^2(\Omega)}.
    \end{alignedat}
  \end{equation}
  Since $\tilde{\pmb{\psi}}_h^c =  \pmb{\Pi}^{\mathrm{curl},0}_h \nabla \times \pmb{\rho} \in \nabla \times \NedelecZero$ we have
  \begin{equation*}
    \| \nabla \cdot ( \pmb{\psi} - \tilde{\pmb{\psi}}_h ) ) \|_{L^2(\Omega)} = \| \nabla \cdot ( \nabla z - \tilde{\pmb{\psi}}_h^g ) ) \|_{L^2(\Omega)}.
  \end{equation*}
  Due to the regularity of $z \in H^3(\Omega)$ we can find $\tilde{\pmb{\psi}}_h^g \in \RTBDMZero$ such that
  \begin{equation*}
    \| \nabla z - \tilde{\pmb{\psi}}_h^g  \|_{\HDiv}
    \lesssim
      \frac{h}{p_v} \norm{\nabla z}_{\pmb{H}^1(\Omega, \div)}
      \lesssim \frac{h}{p_v} \norm{\nabla \cdot \pmb{\psi}}_{H^1(\Omega)}
      \lesssim \frac{h}{p_v} \norm{\pmb{e}^{\pmb{\varphi}}}_{L^2(\Omega)}
      \lesssim \frac{h}{p_v} \| ( \pmb{e}^{\pmb{\varphi}} , e^u ) \|_b.
  \end{equation*}
  Therefore, estimate (\ref{eq:e_phi_intermediate_estimate_2}) can be summarized as follows:
  \begin{equation}\label{eq:e_phi_intermediate_estimate_3}
    (\pmb{e}^{\pmb{\varphi}} , \pmb{\psi} - \tilde{\pmb{\psi}}_h )_{\ltwo}
    \lesssim \left[\frac{h}{p_v} \| (e^u, \pmb{e}^{\pmb{\varphi}}) \|_b
    + \| \pmb{\varphi} - \tilde{\pmb{\varphi}}_h \|_{L^2(\Omega)}
    + \frac{h}{p_v} \norm{\nabla \cdot (\pmb{\varphi} - \tilde{\pmb{\varphi}}_h) }_{L^2(\Omega)} \right] \norm{\pmb{e}^{\pmb{\varphi}}}_{L^2(\Omega)}.
  \end{equation}
  Again due to the regularity of $v \in H^3(\Omega)$ we can find $\tilde{v}_h \in \Sp$ such that
  \begin{equation*}
    \| v - \tilde{v}_h  \|_{H^1(\Omega)}
    \lesssim \frac{h}{p_s} \norm{v}_{H^2(\Omega)}
    \lesssim \frac{h}{p_s} \norm{\pmb{e}^{\pmb{\varphi}}}_{L^2(\Omega)}.
  \end{equation*}
  Finally, summarizing the estimates (\ref{eq:e_phi_intermediate_estimate_1}) and (\ref{eq:e_phi_intermediate_estimate_3}) and again using
  \begin{equation*}
    \| \nabla \cdot ( \pmb{\psi} - \tilde{\pmb{\psi}}_h ) ) \|_{L^2(\Omega)}
      = \| \nabla \cdot ( \nabla z - \tilde{\pmb{\psi}}_h^g ) ) \|_{L^2(\Omega)}
      \lesssim \frac{h}{p_v} \| ( \pmb{e}^{\pmb{\varphi}} , e^u ) \|_b
  \end{equation*}
  we find
  \begin{equation*}
      \norm{\pmb{e}^{\pmb{\varphi}}}_{L^2(\Omega)}^2
      \lesssim \left[\frac{h}{p} \| ( \pmb{e}^{\pmb{\varphi}} , e^u ) \|_b
      + \| \pmb{\varphi} - \tilde{\pmb{\varphi}}_h \|_{L^2(\Omega)}
      + \frac{h}{p} \norm{\nabla \cdot (\pmb{\varphi} - \tilde{\pmb{\varphi}}_h) }_{L^2(\Omega)} \right] \norm{\pmb{e}^{\pmb{\varphi}}}_{L^2(\Omega)}.
  \end{equation*}
  Canceling one power of $\norm{\pmb{e}^{\pmb{\varphi}}}_{L^2(\Omega)}$ then yields the first estimate. The second one follows again by the fact that the least squares approximation is the projection with respect to $b$ and the norm equivalence given in Theorem~\ref{theorem:norm_equivalence}.
\end{proof}

\begin{lemma}\label{lemma:convergence_of_dual_solution_grad_u}
  Let $\Gamma$ be smooth and $( \pmb{\varphi}_h , u_h )$ be the least squares approximation of $( \pmb{\varphi} , u )$. Furthermore, let  $e^u = u-u_h$ and $\pmb{e}^{\pmb{\varphi}} = \pmb{\varphi}-\pmb{\varphi}_h$.
  Let $(\pmb{\psi}, v) \in \HZeroDiv \times \HOne$ be the solution of the dual problem given by Theorem~\ref{theorem:duality_argument_grad_u} with $w = e^u$.
  Additionally, let $( \pmb{\psi}_h , v_h )$ be the least squares approximation of $( \pmb{\psi} , v )$ and
  denote $e^v = v-v_h$ and $\pmb{e}^{\pmb{\psi}} = \pmb{\psi}-\pmb{\psi}_h$.
  Then,
  \begin{equation*}
    \| ( \pmb{e}^{\pmb{\psi}} , e^v ) \|_b    \lesssim \norm{\nabla e^u}_{L^2(\Omega)} \quad \text{ and } \quad
    \norm{e^v}_{L^2(\Omega)}                          \lesssim \frac{h}{p} \norm{\nabla e^u}_{L^2(\Omega)} \quad \text{ and } \quad
    \| \pmb{e}^{\pmb{\psi}} \|_{L^2(\Omega)}  \lesssim \frac{h}{p} \norm{\nabla e^u}_{L^2(\Omega)}.
  \end{equation*}
\end{lemma}

\begin{proof}
  Theorem~\ref{theorem:duality_argument_grad_u} provides
  $\|\pmb{\psi}\|_{\pmb{H}^2(\Omega)} + \|\nabla \cdot \pmb{\psi}\|_{H^1(\Omega)} + \|v \|_{H^1(\Omega)} \lesssim \|\nabla w\|_{L^2(\Omega)}$.
  Stability of the least squares method (cf.\ (\ref{eq:cea-lemma})) yields
  \begin{equation*}
    \| ( \pmb{e}^{\pmb{\psi}} , e^v ) \|_b
    \lesssim \norm{\nabla e^u}_{L^2(\Omega)}.
  \end{equation*}
  By Lemma~\ref{lemma:e_u_suboptimal_l2_error_estimate} we have
  \begin{equation*}
    \norm{e^v}_{L^2(\Omega)} \lesssim h/p \| ( \pmb{e}^{\pmb{\psi}} , e^v ) \|_b,
  \end{equation*}
  which together with the above gives the second estimate.
  By Theorem~\ref{theorem:e_phi_optimal_l2_error_estimate} we have
  \begin{equation*}
    \| \pmb{e}^{\pmb{\psi}} \|_{L^2(\Omega)}
    \lesssim
      \frac{h}{p} \norm{v - \tilde{v}_h}_{H^1(\Omega)} +
      \| \pmb{\psi} - \tilde{\pmb{\psi}}_h \|_{L^2(\Omega)} +
      \frac{h}{p} \| \nabla \cdot (\pmb{\psi} - \tilde{\pmb{\psi}}_h) \|_{L^2(\Omega)}
  \end{equation*}
  for any $\tilde{v}_h \in \Sp$, $\tilde{\pmb{\psi}}_h \in \RTBDMZero$.
  The result follows immediately by again exploiting the regularity of the dual solution and the approximation properties of the employed spaces.
\end{proof}

\begin{theorem}\label{theorem:grad_e_u_optimal_l2_error_estimate}
  Let $\Gamma$ be smooth and $( \pmb{\varphi}_h , u_h )$ be the least squares approximation of $( \pmb{\varphi} , u )$.
  Furthermore, let  $e^u = u-u_h$. Then, for any $\tilde{\pmb{\varphi}}_h \in \RTBDMZero$, $\tilde{u}_h \in \Sp$,
  \begin{equation*}
    \norm{\nabla e^u}_{L^2(\Omega)}
    \lesssim \norm{ u - \tilde{u}_h }_{H^1(\Omega)} +
      \frac{h}{p} \| \pmb{\varphi} - \tilde{\pmb{\varphi}}_h \|_{L^2(\Omega)} +
      \frac{h}{p} \| \nabla \cdot (\pmb{\varphi} - \tilde{\pmb{\varphi}}_h ) \|_{L^2(\Omega)}.
  \end{equation*}
\end{theorem}

\begin{proof}
  As in Remark~\ref{remark:heuristic_arguments} with $( \pmb{e}^{\pmb{\psi}} , e^v )$ denoting the error of the FOSLS approximation of the dual solution given by Theorem~\ref{theorem:duality_argument_grad_u} (duality argument for the gradient of the scalar variable) applied to $w = e^u$
  we have for any $\tilde{\pmb{\varphi}}_h \in \RTBDMZero$, $\tilde{u}_h \in \Sp$
  \begin{align*}
    \norm{e^u}_{L^2(\Omega)}^2
    &= b( ( \pmb{\varphi} - \tilde{\pmb{\varphi}}_h , u - \tilde{u}_h ), ( \pmb{e}^{\pmb{\psi}} , e^v ) ) \\
    &=
    ( \nabla \cdot (\pmb{\varphi} - \tilde{\pmb{\varphi}}_h) + \gamma (u - \tilde{u}_h) ,
      \nabla \cdot \pmb{e}^{\pmb{\psi}} + \gamma e^v )_{\ltwo}
    + ( \nabla (u - \tilde{u}_h) + \pmb{\varphi} - \tilde{\pmb{\varphi}}_h , \nabla e^v + \pmb{e}^{\pmb{\psi}} )_{\ltwo}.
  \end{align*}
  We specifically choose $\tilde{\pmb{\varphi}}_h = \IhZero \pmb{\varphi}$.
  In the following we heavily use the properties of the operator $\IhZero$ given in Lemma~\ref{lemma:properties_of_Ih}.
  First we exploit the regularity of the dual solution using Lemma~\ref{lemma:convergence_of_dual_solution_grad_u} as well as the estimates of Theorem~\ref{theorem:duality_argument_grad_u}:
  \begin{equation*}
  \begin{alignedat}{2}
    ( \gamma (u - \tilde{u}_h) , \nabla \cdot \pmb{e}^{\pmb{\psi}} + \gamma e^v )_{\ltwo}
      &\lesssim \norm{ u - \tilde{u}_h }_{L^2(\Omega)} \| ( \pmb{e}^{\pmb{\psi}}, e^v ) \|_b \\
      &\lesssim \norm{ u - \tilde{u}_h }_{H^1(\Omega)} \norm{\nabla e^u}_{L^2(\Omega)} , \\
    ( \nabla (u - \tilde{u}_h) , \nabla e^v + \pmb{e}^{\pmb{\psi}} )_{\ltwo}
        &\lesssim \norm{\nabla (u - \tilde{u}_h) }_{L^2(\Omega)} \| ( \pmb{e}^{\pmb{\psi}}, e^v ) \|_b \\
        &\lesssim \norm{ u - \tilde{u}_h }_{H^1(\Omega)} \norm{\nabla e^u}_{L^2(\Omega)}, \\
    ( \pmb{\varphi} - \IhZero \pmb{\varphi} , \nabla e^v )_{\ltwo}
        &= - ( \nabla \cdot ( \pmb{\varphi} - \IhZero \pmb{\varphi} ) , e^v )_{\ltwo} \\
        &\leq \norm{\nabla \cdot ( \pmb{\varphi} - \IhZero \pmb{\varphi})}_{L^2(\Omega)} \norm{e^v}_{L^2(\Omega)} \\
        &\lesssim h/p \norm{\nabla \cdot ( \pmb{\varphi} - \IhZero \pmb{\varphi})}_{L^2(\Omega)} \norm{\nabla e^u}_{L^2(\Omega)}, \\
    ( \nabla \cdot (\pmb{\varphi} - \IhZero \pmb{\varphi}) , \gamma e^v )_{\ltwo}
        &\leq \norm{\nabla \cdot ( \pmb{\varphi} - \IhZero \pmb{\varphi})}_{L^2(\Omega)} \norm{e^v}_{L^2(\Omega)} \\
        &\lesssim h/p \norm{\nabla \cdot ( \pmb{\varphi} - \IhZero \pmb{\varphi})}_{L^2(\Omega)} \norm{\nabla e^u}_{L^2(\Omega)}, \\
    ( \pmb{\varphi} - \IhZero \pmb{\varphi} , \pmb{e}^{\pmb{\psi}} )_{\ltwo}
        &\lesssim \norm{  \pmb{\varphi} - \IhZero \pmb{\varphi} }_{L^2(\Omega)} \| \pmb{e}^{\pmb{\psi}} \|_{L^2(\Omega)} \\
        &\lesssim h/p \norm{  \pmb{\varphi} - \IhZero \pmb{\varphi} }_{L^2(\Omega)} \norm{\nabla e^u}_{L^2(\Omega)}, \\
    ( \nabla \cdot (\pmb{\varphi} - \IhZero \pmb{\varphi}) , \nabla \cdot \pmb{e}^{\pmb{\psi}} )_{\ltwo}
        &= ( \nabla \cdot (\pmb{\varphi} - \IhZero \pmb{\varphi}) , \nabla \cdot (\pmb{\psi} - \tilde{\pmb{\psi}}_h) )_{\ltwo} \\
        &\leq \norm{ \nabla \cdot (\pmb{\varphi} - \IhZero \pmb{\varphi}) }_{L^2(\Omega)} \| \nabla \cdot (\pmb{\psi} - \tilde{\pmb{\psi}}_h) \|_{L^2(\Omega)} \\
        &\lesssim h/p \norm{  \nabla \cdot (\pmb{\varphi} - \IhZero \pmb{\varphi}) }_{L^2(\Omega)} \norm{\nabla e^u}_{L^2(\Omega)}.
  \end{alignedat}
\end{equation*}
  Canceling one power of $\norm{\nabla e^u}_{L^2(\Omega)}$, collecting the terms, and using the estimate for $\IhZero$ we arrive at the asserted estimate.
\end{proof}

As a tool in the proof of our main theorem (Theorem~\ref{theorem:e_u_optimal_l2_error_estimate}) we need to analyze the error of the FOSLS approximation of the dual solution.
This is summarized in

\begin{lemma}\label{lemma:convergence_of_dual_solution}
  Let $\Gamma$ be smooth and $( \pmb{\varphi}_h , u_h )$ be the least squares approximation of $( \pmb{\varphi} , u )$. Furthermore, let  $e^u = u-u_h$ and $\pmb{e}^{\pmb{\varphi}} = \pmb{\varphi}-\pmb{\varphi}_h$.
  Let $(\pmb{\psi}, v) \in \HZeroDiv \times \HOne$ be the solution of the dual problem given by Theorem~\ref{theorem:duality_argument} with $w = e^u$.
  Additionally, let $( \pmb{\psi}_h , v_h )$ be the least squares approximation of $( \pmb{\psi} , v )$ and
  denote $e^v = v-v_h$ and $\pmb{e}^{\pmb{\psi}} = \pmb{\psi}-\pmb{\psi}_h$.
  Then,
  \begin{equation*}
    \| ( \pmb{e}^{\pmb{\psi}} , e^v ) \|_b \lesssim \frac{h}{p}     \norm{e^u}_{L^2(\Omega)} \quad \text{ and } \quad
    \norm{e^v}_{L^2(\Omega)}                       \lesssim \left(\frac{h}{p}\right)^{2} \norm{e^u}_{L^2(\Omega)}.
  \end{equation*}
  Furthermore,
  \[
  \| \pmb{e}^{\pmb{\psi}} \|_{L^2(\Omega)} \lesssim
  \begin{cases}
      h \norm{e^u}_{L^2(\Omega)}   & \mbox{ if } \RTBDMZero = \pmb{\mathrm{RT}}_{0}^{0}(\mathcal{T}_h), \\
      \left(\frac{h}{p}\right)^{2} \norm{e^u}_{L^2(\Omega)} & \textit{else}.
   \end{cases}
   \]
\end{lemma}

\begin{proof}
  %Since $(\pmb{\psi}_h, v_h)$ is just the $b$ projection of $(\pmb{\psi}, v)$ onto the finite element spaces the error is optimal in the $b$ norm.
  Theorem~\ref{theorem:duality_argument} gives $\pmb{\psi} \in \pmb{H}^3(\Omega)$, $\nabla \cdot \pmb{\psi} \in H^2(\Omega)$ and $v \in H^2(\Omega)$
  with norms bounded by $\|e^v\|_{L^2(\Omega)}$.
  Therefore we have in view of optimality of the FOSLS method in the $b$-norm
  \begin{equation*}
    \| ( \pmb{e}^{\pmb{\psi}} , e^v ) \|_b
    \stackrel{(\ref{eq:cea-lemma})}{\leq} \| ( \pmb{\psi} - \tilde{\pmb{\psi}}_h , v - \tilde{v}_h ) \|_b
    \lesssim h/p \norm{e^v}_{L^2(\Omega)},
  \end{equation*}
  where the first estimate holds for any $\tilde{v}_h \in S_p$, $\tilde{\pmb{\psi}}_h \in \RTBDMZero$ and
  the second one follows with the same arguments as in the proof of Lemma~\ref{lemma:e_u_suboptimal_l2_error_estimate}.
  By Lemma~\ref{lemma:e_u_suboptimal_l2_error_estimate} we have
  \begin{equation*}
    \norm{e^v}_{L^2(\Omega)} \lesssim h/p \| ( \pmb{e}^{\pmb{\psi}} , e^v ) \|_b,
  \end{equation*}
  which together with the above gives the second estimate.
  By Theorem~\ref{theorem:e_phi_optimal_l2_error_estimate} we have
  \begin{equation*}
    \| \pmb{e}^{\pmb{\psi}} \|_{L^2(\Omega)}
    \lesssim
      \frac{h}{p} \norm{v - \tilde{v}_h}_{H^1(\Omega)} +
      \| \pmb{\psi} - \tilde{\pmb{\psi}}_h \|_{L^2(\Omega)} +
      \frac{h}{p} \| \nabla \cdot (\pmb{\psi} - \tilde{\pmb{\psi}}_h) \|_{L^2(\Omega)}
  \end{equation*}
  for any $\tilde{v}_h \in \Sp$, $\tilde{\pmb{\psi}}_h \in \RTBDMZero$.
  The result follows immediately by again exploiting the regularity of the dual solution and the approximation properties of the employed spaces.
\end{proof}

\begin{theorem}\label{theorem:e_u_optimal_l2_error_estimate}
  Let $\Gamma$ be smooth and $( \pmb{\varphi}_h , u_h )$ be the least squares approximation of $( \pmb{\varphi} , u )$.
  Furthermore, let  $e^u = u-u_h$. Then,
  for any $\tilde{\pmb{\varphi}}_h \in \RTBDMZero$, $\tilde{u}_h \in \Sp$,
  \[
  \norm{e^u}_{L^2(\Omega)}
   \lesssim
    \begin{cases}
      h \norm{ u - \tilde{u}_h }_{H^1(\Omega)} +
        h \| \pmb{\varphi} - \tilde{\pmb{\varphi}}_h \|_{L^2(\Omega)} +
        h \| \nabla \cdot (\pmb{\varphi} - \tilde{\pmb{\varphi}}_h ) \|_{L^2(\Omega)} & \mbox{for } \pmb{\mathrm{RT}}_{0}^{0}(\mathcal{T}_h), \\
      h \norm{ u - \tilde{u}_h }_{H^1(\Omega)} +
        h^2 \| \pmb{\varphi} - \tilde{\pmb{\varphi}}_h \|_{L^2(\Omega)} +
        h \| \nabla \cdot (\pmb{\varphi} - \tilde{\pmb{\varphi}}_h ) \|_{L^2(\Omega)} & \mbox{for } \pmb{\mathrm{BDM}}_{1}^{0}(\mathcal{T}_h), \\
      \frac{h}{p} \norm{ u - \tilde{u}_h }_{H^1(\Omega)} +
        \left( \frac{h}{p} \right)^2 \| \pmb{\varphi} - \tilde{\pmb{\varphi}}_h \|_{L^2(\Omega)} +
        \left( \frac{h}{p} \right)^2 \| \nabla \cdot (\pmb{\varphi} - \tilde{\pmb{\varphi}}_h ) \|_{L^2(\Omega)} & \textit{else}.
     \end{cases}
  \]
\end{theorem}

\begin{proof}
  As in Remark~\ref{remark:heuristic_arguments} with $( \pmb{e}^{\pmb{\psi}} , e^v )$ denoting the FOSLS approximation of the dual solution given by Theorem~\ref{theorem:duality_argument} applied to $w = e^u$
  we have for any $\tilde{\pmb{\varphi}}_h \in \RTBDMZero$, $\tilde{u}_h \in \Sp$
  \begin{align*}
    \norm{e^u}_{L^2(\Omega)}^2
    &= b( ( \pmb{\varphi} - \tilde{\pmb{\varphi}}_h , u - \tilde{u}_h ), ( \pmb{e}^{\pmb{\psi}} , e^v ) ) \\
    &=
    ( \nabla \cdot (\pmb{\varphi} - \tilde{\pmb{\varphi}}_h) + \gamma (u - \tilde{u}_h) ,
      \nabla \cdot \pmb{e}^{\pmb{\psi}} + \gamma e^v )_{\ltwo}
    + ( \nabla (u - \tilde{u}_h) + \pmb{\varphi} - \tilde{\pmb{\varphi}}_h , \nabla e^v + \pmb{e}^{\pmb{\psi}} )_{\ltwo}.
  \end{align*}
  We specifically choose $\tilde{\pmb{\varphi}}_h = \IhZero \pmb{\varphi}$.
  In the following we heavily use the properties of the operator $\IhZero$ given in Lemma~\ref{lemma:properties_of_Ih}.
  First we exploit the regularity of the dual solution using Lemma~\ref{lemma:convergence_of_dual_solution} as well as the estimates of Theorem~\ref{theorem:duality_argument}:
  \begin{equation*}
  \begin{alignedat}{2}
    ( \gamma (u - \tilde{u}_h) , \nabla \cdot \pmb{e}^{\pmb{\psi}} + \gamma e^v )_{\ltwo}
      &\lesssim \norm{ u - \tilde{u}_h }_{L^2(\Omega)} \| ( \pmb{e}^{\pmb{\psi}}, e^v ) \|_b \\
      &\lesssim h/p \norm{ u - \tilde{u}_h }_{H^1(\Omega)} \norm{e^u}_{L^2(\Omega)} , \\
    ( \nabla (u - \tilde{u}_h) , \nabla e^v + \pmb{e}^{\pmb{\psi}} )_{\ltwo}
        &\lesssim \norm{\nabla (u - \tilde{u}_h) }_{L^2(\Omega)} \| ( \pmb{e}^{\pmb{\psi}}, e^v ) \|_b \\
        &\lesssim h/p \norm{ u - \tilde{u}_h }_{H^1(\Omega)} \norm{e^u}_{L^2(\Omega)}, \\
    ( \pmb{\varphi} - \IhZero \pmb{\varphi} , \nabla e^v )_{\ltwo}
        &= - ( \nabla \cdot ( \pmb{\varphi} - \IhZero \pmb{\varphi} ) , e^v )_{\ltwo} \\
        &\leq \norm{\nabla \cdot ( \pmb{\varphi} - \IhZero \pmb{\varphi})}_{L^2(\Omega)} \norm{e^v}_{L^2(\Omega)} \\
        &\lesssim (h/p)^2 \norm{\nabla \cdot ( \pmb{\varphi} - \IhZero \pmb{\varphi})}_{L^2(\Omega)} \norm{e^u}_{L^2(\Omega)}, \\
    ( \nabla \cdot (\pmb{\varphi} - \IhZero \pmb{\varphi}) , \gamma e^v )_{\ltwo}
        &\leq \norm{\nabla \cdot ( \pmb{\varphi} - \IhZero \pmb{\varphi})}_{L^2(\Omega)} \norm{e^v}_{L^2(\Omega)} \\
        &\lesssim (h/p)^2 \norm{\nabla \cdot ( \pmb{\varphi} - \IhZero \pmb{\varphi})}_{L^2(\Omega)} \norm{e^u}_{L^2(\Omega)}, \\
    ( \pmb{\varphi} - \IhZero \pmb{\varphi} , \pmb{e}^{\pmb{\psi}} )_{\ltwo}
        &\lesssim \norm{  \pmb{\varphi} - \IhZero \pmb{\varphi} }_{L^2(\Omega)} \| \pmb{e}^{\pmb{\psi}} \|_{L^2(\Omega)} \\
        &\lesssim
        \begin{cases}
            h \norm{  \pmb{\varphi} - \IhZero \pmb{\varphi} }_{L^2(\Omega)} \norm{e^u}_{L^2(\Omega)}   & \mbox{if } \RTBDMZero = \pmb{\mathrm{RT}}_{0}^{0}(\mathcal{T}_h), \\
            \left(\frac{h}{p}\right)^{2} \norm{  \pmb{\varphi} - \IhZero \pmb{\varphi} }_{L^2(\Omega)} \norm{e^u}_{L^2(\Omega)} & \textit{else},
         \end{cases}. \\
    ( \nabla \cdot (\pmb{\varphi} - \IhZero \pmb{\varphi}) , \nabla \cdot \pmb{e}^{\pmb{\psi}} )_{\ltwo}
        &= ( \nabla \cdot (\pmb{\varphi} - \IhZero \pmb{\varphi}) , \nabla \cdot (\pmb{\psi} - \tilde{\pmb{\psi}}_h) )_{\ltwo} \\
        &\leq \norm{ \nabla \cdot (\pmb{\varphi} - \IhZero \pmb{\varphi}) }_{L^2(\Omega)} \| \nabla \cdot (\pmb{\psi} - \tilde{\pmb{\psi}}_h) \|_{L^2(\Omega)} \\
        &\lesssim
        \begin{cases}
            h \norm{  \nabla \cdot (\pmb{\varphi} - \IhZero \pmb{\varphi}) }_{L^2(\Omega)} \norm{e^u}_{L^2(\Omega)}   & \mbox{if } p_v = 1, \\
            \left(\frac{h}{p}\right)^{2} \norm{  \nabla \cdot (\pmb{\varphi} - \IhZero \pmb{\varphi}) }_{L^2(\Omega)} \norm{e^u}_{L^2(\Omega)} & \textit{else}.
        \end{cases}
  \end{alignedat}
  \end{equation*}
  Canceling one power of $\norm{e^u}_{L^2(\Omega)}$, collecting the terms, and using the estimate for $\IhZero$ we arrive at the asserted estimate.
\end{proof}

\begin{remark}
  Before stating the general corollary with prescribed right-hand side $f \in H^s(\Omega)$ we highlight the improved convergence result.
  Consider $f \in L^2(\Omega)$.
  For the classical conforming finite element method one observes convergence $O(h^2)$ due to the Aubin-Nitsche trick.
  More precisely, let $u_h^{\mathrm{FEM}}$ be the solution to the model problem obtained by classical FEM, then there holds
  \begin{equation*}
    \| u - u_h^{\mathrm{FEM}} \|_{L^2(\Omega)} \lesssim h^2 \norm{u}_{H^2(\Omega)} \lesssim h^2 \norm{f}_{L^2(\Omega)}.
  \end{equation*}
  As elaborated in Section~\ref{section:introduction} this rate could not be obtained for the FOSLS method by previous results, since further regularity of the vector variable $\pmb{\varphi}$ would be necessary.
  Results like \cite[Lemma~{3.4}]{bochev-gunzburger05} and \cite[Thm.~{4.1}]{jespersen77} are essentially a duality argument like Theorem~\ref{theorem:duality_argument} and the strategy of Lemma~\ref{lemma:e_u_suboptimal_l2_error_estimate}.
  Without further analysis the estimate of Lemma~\ref{lemma:e_u_suboptimal_l2_error_estimate}, does not give any further powers of $h$, since the $b$-norm is equivalent to the $\HDiv \times \HOne$ norm.
  Theorem~\ref{theorem:e_u_optimal_l2_error_estimate} ensures, at least if the space $\RTBDMZero$ is not of lowest order, i.e. $p_v > 1$, that the FOSLS method converges
also as $O(h^2)$.
  More precisely, the estimate in Theorem~\ref{theorem:e_u_optimal_l2_error_estimate} together with the approximation properties of the employed finite element spaces and $p_v > 1$ and $p_s \geq 1$ gives
  \begin{equation*}
    \norm{e^u}_{L^2(\Omega)} \lesssim
      h^2 \norm{ u }_{H^2(\Omega)} +
      h^2 \| \pmb{\varphi} \|_{H^1(\Omega)} +
      h^2 \| \nabla \cdot \pmb{\varphi} \|_{L^2(\Omega)}
      \lesssim h^2 \norm{f}_{L^2(\Omega)}.
  \end{equation*}
  So in fact the optimal rate in the sense of the beginning of Section~\ref{section:error_analysis} is achieved.
  If the lowest order case $p_v = 1$ also achieves optimal order is yet to be answered.
  Numerical experiments in Section~\ref{section:numerical_examples}, however, indicate it to be true.
\eremk
\end{remark}

We summarize the results for general right-hand side $f \in H^s(\Omega)$.
This summary is essentially the estimates given by the Theorems~\ref{theorem:e_phi_optimal_l2_error_estimate}, \ref{theorem:grad_e_u_optimal_l2_error_estimate}, and \ref{theorem:e_u_optimal_l2_error_estimate} together with the approximation properties of the employed finite element spaces.

\begin{corollary}\label{corollary:summary_of_estimates_for_f_in_higher_order_sobolev_space}
  Let $\Gamma$ be smooth and $f \in H^s(\Omega)$ for some $s \geq 0$.
  Then the solution to $(\ref{eq:model_problem_first_order_system})$ satisfies $u \in H^{s+2}(\Omega)$, $\pmb{\varphi} \in \pmb{H}^{s+1}(\Omega)$ and $\nabla \cdot \pmb{\varphi} \in H^s(\Omega)$.
  Let $( \pmb{\varphi}_h , u_h )$ be the least squares approximation of $( \pmb{\varphi} , u )$.
  Furthermore, let  $e^u = u-u_h$ and $\pmb{e}^{\pmb{\varphi}} = \pmb{\varphi}-\pmb{\varphi}_h$.
  Then, for the lowest order case $p_v = 1$,
  \begin{equation*}
    \norm{e^u}_{L^2(\Omega)} \lesssim
    h^{\min(s+1, 2)} \norm{f}_{H^s(\Omega)}.
  \end{equation*}
  For $p_v > 1$ there holds
  \begin{equation*}
    \norm{e^u}_{L^2(\Omega)} \lesssim \left(\frac{h}{p}\right)^{\min(s+1, p_s, p_v+1) + 1} \norm{f}_{H^s(\Omega)}.
  \end{equation*}
  Furthermore, the estimate
  \begin{equation*}
    \norm{\nabla e^u}_{L^2(\Omega)} \lesssim \left(\frac{h}{p}\right)^{\min(s+1, p_s, p_v+1)} \norm{f}_{H^s(\Omega)}.
  \end{equation*}
  holds. Finally, we have
  \\
  \begin{center}
    \begin{tabular}{|c|c|}                   \hline
      $ \RTBDMZero = \RTZero $ & $ \RTBDMZero = \BDMZero $    \\ \hline
      $ \norm{\pmb{e}^{\pmb{\varphi}}}_{L^2(\Omega)} \lesssim \left(\frac{h}{p}\right)^{\min(s+1, p_s + 1, p_v) } \norm{f}_{H^s(\Omega)} $ &
      $ \norm{\pmb{e}^{\pmb{\varphi}}}_{L^2(\Omega)} \lesssim \left(\frac{h}{p}\right)^{\min(s+1, p_s + 1, p_v + 1) } \norm{f}_{H^s(\Omega)} $. \\ \hline
    \end{tabular}
  \end{center}
\end{corollary}

\begin{proof}
  The regularity result follows immediately by standard arguments together with the fact that $ \pmb{\varphi} = - \nabla u $.
  We now analyze the quantities in the estimates of the Theorems~\ref{theorem:e_phi_optimal_l2_error_estimate}, \ref{theorem:grad_e_u_optimal_l2_error_estimate} and \ref{theorem:e_u_optimal_l2_error_estimate}:
  \begin{align*}
    &\norm{ u - \tilde{u}_h }_{H^1(\Omega)}
      \lesssim (h/p)^{\min(s+1, p_s)} \norm{u}_{H^{s+2}(\Omega)} \lesssim (h/p)^{\min(s+1, p_s)} \norm{f}_{H^{s}(\Omega)}, \\
    &\| \pmb{\varphi} - \tilde{\pmb{\varphi}}_h \|_{L^2(\Omega)}
      \lesssim
        \begin{cases}
           (h/p)^{\min(s+1, p_v)}   \norm{ \pmb{\varphi} }_{H^{s+1}(\Omega)} \lesssim (h/p)^{\min(s+1, p_v)}   \norm{f}_{H^{s}(\Omega)} & \mbox{for } \RTZero, \\
           (h/p)^{\min(s+1, p_v+1)} \norm{ \pmb{\varphi} }_{H^{s+1}(\Omega)} \lesssim (h/p)^{\min(s+1, p_v+1)} \norm{f}_{H^{s}(\Omega)} & \mbox{for } \BDMZero,
        \end{cases} \\
    &\| \nabla \cdot (\pmb{\varphi} - \tilde{\pmb{\varphi}}_h ) \|_{L^2(\Omega)}
      \lesssim (h/p)^{\min(s, p_v)} \norm{ \nabla \cdot \pmb{\varphi} }_{H^{s}(\Omega)} \lesssim h^{\min(s, p_v)} \norm{f}_{H^{s}(\Omega)}.
  \end{align*}
  The estimates of the Theorems~\ref{theorem:e_phi_optimal_l2_error_estimate}, \ref{theorem:grad_e_u_optimal_l2_error_estimate}, and \ref{theorem:e_u_optimal_l2_error_estimate} together with the above estimates give, after straightforward calculations, the asserted rates.
\end{proof}

We close this section with some remarks concerning sharpness of the estimates of Corollary~\ref{corollary:summary_of_estimates_for_f_in_higher_order_sobolev_space}:

\begin{remark}\label{remark:sharpness_of_estimates}
  Let the assumptions of Corollary~\ref{corollary:summary_of_estimates_for_f_in_higher_order_sobolev_space} be satisfied.
  From a best approximation point of view, since $u \in H^{s+2}(\Omega)$, we have
  \begin{align*}
    \inf_{\tilde{u}_h \in \Sp} \norm{u - \tilde{u}_h}_{L^2(\Omega)} &= O(h^{\min(s+1, p_s)+1}) \\
    \inf_{\tilde{u}_h \in \Sp} \norm{\nabla (u - \tilde{u}_h) }_{L^2(\Omega)} &= O(h^{\min(s+1, p_s)}) \\
    \inf_{\tilde{\pmb{\varphi}}_h \in \RTBDMZero} \| \pmb{\varphi} - \tilde{\pmb{\varphi}}_h \|_{L^2(\Omega)}
      &=
        \begin{cases}
           O(h^{\min(s+1, p_v)} )   & \mbox{if }\RTBDMZero = \RTZero, \\
           O(h^{\min(s+1, p_v+1)} ) & \mbox{if }\RTBDMZero = \BDMZero.
        \end{cases}
  \end{align*}
  Excluding the lowest order case $p_v = 1$ we have, choosing $p_v \geq p_s - 1$, sharpness of the estimates for $e^u$ and $\nabla e^u$.
  This can be easily seen, since the rates guaranteed by Corollary~\ref{corollary:summary_of_estimates_for_f_in_higher_order_sobolev_space}
  for $\norm{e^u}_{L^2(\Omega)}$ and $\norm{\nabla e^u}_{L^2(\Omega)}$ are the same as the ones from a best approximation argument.
  The estimates are therefore sharp.
  The lowest order case $p_v = 1$ seems to be suboptimal, as the numerical examples in Section~\ref{section:numerical_examples} suggest.
  In all other cases, i.e., $p_v > 1$ and $p_v < p_s - 1$, our numerical examples suggest sharpness of the estimates, in both the setting of a smooth solution as well as one with finite Sobolev regularity,
  but not achieving the best approximation rate.
  Since in the least squares functional the term $\norm{ \nabla u_h + \pmb{\varphi}_h }_{L^2(\Omega)}$ enforces $\nabla u_h$ and $\pmb{\varphi}_h$ to be $\textit{close}$,
  it is to be expected that an insufficient choice of $p_v$ limits the convergence rate.
  A theoretical justification concerning the observed rates in the cases $p_v = 1$ as well as $p_v > 1$ and $p_v < p_s - 1$ is yet to be studied.
  In conclusion, when the application in question is concerned with approximation of $u$ or $\nabla u$ in the $L^2(\Omega)$ norm,
  the best possible rate with the smallest number of degrees of freedom is achieved with the choice $p_v = p_s - 1$ regardless of the choice of $\RTBDMZero$.
  Therefore, it is computationally favorable to choose Raviart-Thomas elements over Brezzi-Douglas-Marini elements.
  Turning now to $\norm{\pmb{e}^{\pmb{\varphi}}}_{L^2(\Omega)}$ similar arguments guarantee sharpness of the estimates.
  In this case when $p_s + 1 \geq p_v$ and $p_s + 1 \geq p_v + 1$, for the choice of Raviart-Thomas elements and Brezzi-Douglas-Marini elements respectively.
  Again the other cases are open for theoretical justification.
  However, both theoretical as well as the numerical examples in Section~\ref{section:numerical_examples} suggest the choice of Brezzi-Douglas-Marini elements over Raviart-Thomas elements,
  when application is concerned with approximation of $\pmb{\varphi}$ in the $L^2(\Omega)$ norm.
\eremk
\end{remark}

%% file: 04_numerical_examples.tex
\section{Numerical examples}\label{section:numerical_examples}

All our calculations are performed with the $hp$-FEM code NETGEN/NGSOLVE by J.~Sch\"oberl,~\cite{schoeberlNGSOLVE,schoeberl97}.
The curved boundaries are implemented using second order rational splines.

In the following we will perform two different numerical experiments.
\begin{enumerate}
  \item
    For the first one we choose $f \in C^{\infty}(\overline{\Omega})$.
    The suboptimal estimate $\norm{e^u}_{L^2(\Omega)} \lesssim  h/p  \lVert ( \pmb{e}^{\pmb{\varphi}} , e^u ) \rVert_b$
    of Lemma~\ref{lemma:e_u_suboptimal_l2_error_estimate} suffices to deduce optimal rates.
    Therefore, we only present three graphs in this section
    in order to highlight two aspects of the least squares approach:
    On the one hand the optimal choice of the
    employed polynomial degrees $p_s$ and $p_v$.
    On the other hand the superiority of Brezzi-Douglas-Marini elements over
    Raviart-Thomas elements when approximating the vector valued variable.
    For completeness we present other convergence plots in Appendix~\ref{section:appendix}.
  \item
    To showcase our new convergence result we then choose
    $f \in \cap_{\varepsilon >0} H^{1/2-\varepsilon}(\Omega)$, but $f \notin H^{1/2}(\Omega)$ with
    $u \in \cap_{\varepsilon >0} H^{5/2-\varepsilon}(\Omega)$ and
    $\pmb{\varphi} \in \cap_{\varepsilon >0} \pmb{H}^{3/2-\varepsilon}(\Omega)$.
    We again only present a selection of graphs focusing on the new convergence results,
    other convergence plots can be found in Appendix~\ref{section:appendix}.
\end{enumerate}

In all graphs, the actual numerical results are given by red dots.
The rate that is guaranteed by Corollary~\ref{corollary:summary_of_estimates_for_f_in_higher_order_sobolev_space} is plotted in black together with the number written out near the bottom right.
Furthermore, in blue the reference line for the best rate possible with the employed space $\Sp$ or $\RTBDMZero$ is plotted, depending on the quantity of interest, i.e.,
for $\norm{e^u}_{L^2(\Omega)}$ the blue reference line corresponds to $h^{\min(s+1, p_s)+1}$,
for $\norm{\nabla e^u}_{L^2(\Omega)}$ the blue reference line corresponds to $h^{\min(s+1, p_s)}$ and
for $\norm{\pmb{e}^{\pmb{\varphi}}}_{L^2(\Omega)}$ the blue reference line corresponds to $h^{\min(s+1, p_v)}$ for $\RTBDMZero = \RTZero$ and $h^{\min(s+1, p_v + 1)}$ for $\RTBDMZero = \BDMZero$.

\begin{example}\label{example:numerics_smooth_solution}
  We consider as the domain $\Omega$ the unit sphere in $\mathbb{R}^2$.
  The exact solution is the smooth function $u(x,y) = \cos(2 \pi (x^2 + y^2))$.
  The numerical results are plotted in Figures~\ref{figure:l2_error_u_RT} and \ref{figure:l2_error_u_BDM} for $\norm{e^u}_{L^2(\Omega)}$,
  in Figures~\ref{figure:l2_error_grad_u_RT} and \ref{figure:l2_error_grad_u_BDM} for $\norm{\nabla e^u}_{L^2(\Omega)}$, and
  in Figures~\ref{figure:l2_error_phi_RT} and \ref{figure:l2_error_phi_BDM} for $\norm{\pmb{e}^{\pmb{\varphi}}}_{L^2(\Omega)}$.
  There are some remarks to be made:
  \begin{itemize}
    \item
      Consider Figure~\ref{figure:l2_error_u_RT} depicting $\norm{e^u}_{L^2(\Omega)}$
      using Raviart-Thomas elements.
      The rates guaranteed by Corollary~\ref{corollary:summary_of_estimates_for_f_in_higher_order_sobolev_space} are achieved in the numerical experiment.
      The important subfigures are the ones in the subdiagonal of the discussed figure, i.e., corresponding to the choice $p_v = p_s - 1$.
      Here, apart from the lowest order case, the best possible rate with the smallest number of degrees of freedom is achieved.
      Above this subdiagonal, i.e., $p_v \geq p_s$, additional degrees of freedom will not increase the rate of convergence, since by pure best approximation arguments the rate of convergence is limited by the polynomial degree $p_s$ of the scalar variable.
      Below this subdiagonal, i.e., $p_v < p_s - 1$, we notice that the rate of convergence is also limited by the polynomial degree $p_v$ of the vector variable.
      Note that the results for $\norm{e^u}_{L^2(\Omega)}$ in Corollary~\ref{corollary:summary_of_estimates_for_f_in_higher_order_sobolev_space} are independent of the choice of the vector valued finite element space, which is also confirmed by our experiments.
      Additional convergence plots can be found in Appendix~\ref{section:appendix}.
    \item
      Consider Figures~\ref{figure:l2_error_phi_RT} and \ref{figure:l2_error_phi_BDM} depicting $\norm{\pmb{e}^{\pmb{\varphi}}}_{L^2(\Omega)}$.
      Apart from similar observations as for the scalar variable,
      it is notable that a difference in the approximation properties
      of the different spaces for the vector variable is observed, as predicted by Corollary~\ref{corollary:summary_of_estimates_for_f_in_higher_order_sobolev_space}.
      Consider wanting to achieve a rate of $h^5$.
      The combination of spaces with the smallest number of degrees of freedom corresponds to $\pmb{\mathrm{BDM}}^0_{4}(\mathcal{T}_h) \times S_{4}(\mathcal{T}_h)$ and $ \pmb{\mathrm{RT}}^0_{4}(\mathcal{T}_h) \times S_{4}(\mathcal{T}_h)$ respectively, highlighting the superiority of the Brezzi-Douglas-Marini elements when approximating $\pmb{\varphi}$.
      For further discussion see again Remark~\ref{remark:sharpness_of_estimates}.
  \end{itemize}
\end{example}

\begin{figure}[H]
	\centering
	\includegraphics[width=\textwidth]{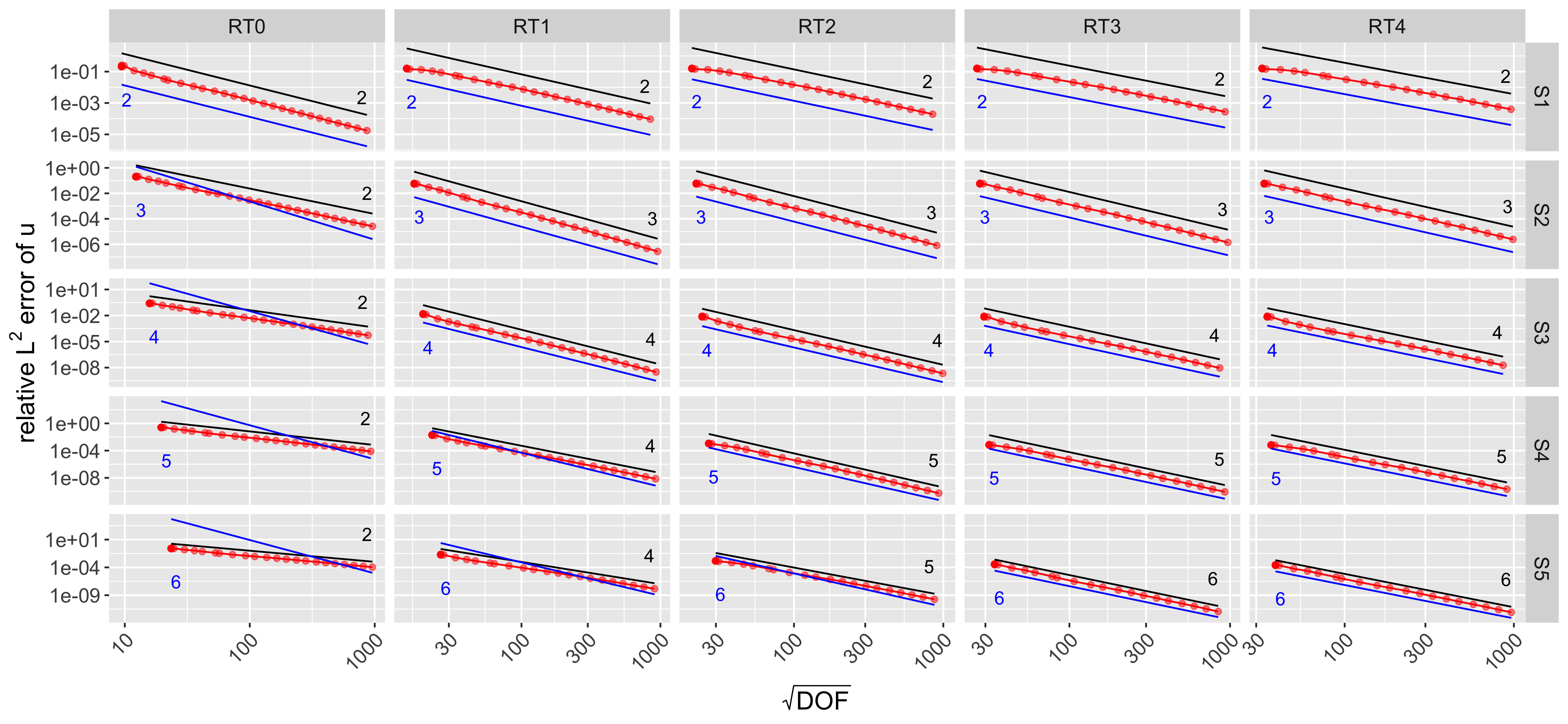}
	\caption{(cf.~Example~\ref{example:numerics_smooth_solution})
  Convergence of $\norm{e^u}_{L^2(\Omega)}$ vs.\ $\sqrt{\operatorname{DOF}} \sim 1/h$ employing $\RTBDMZero = \RTZero$.
	}
	\label{figure:l2_error_u_RT}
\end{figure}

\begin{figure}[H]
	\centering
	\includegraphics[width=\textwidth]{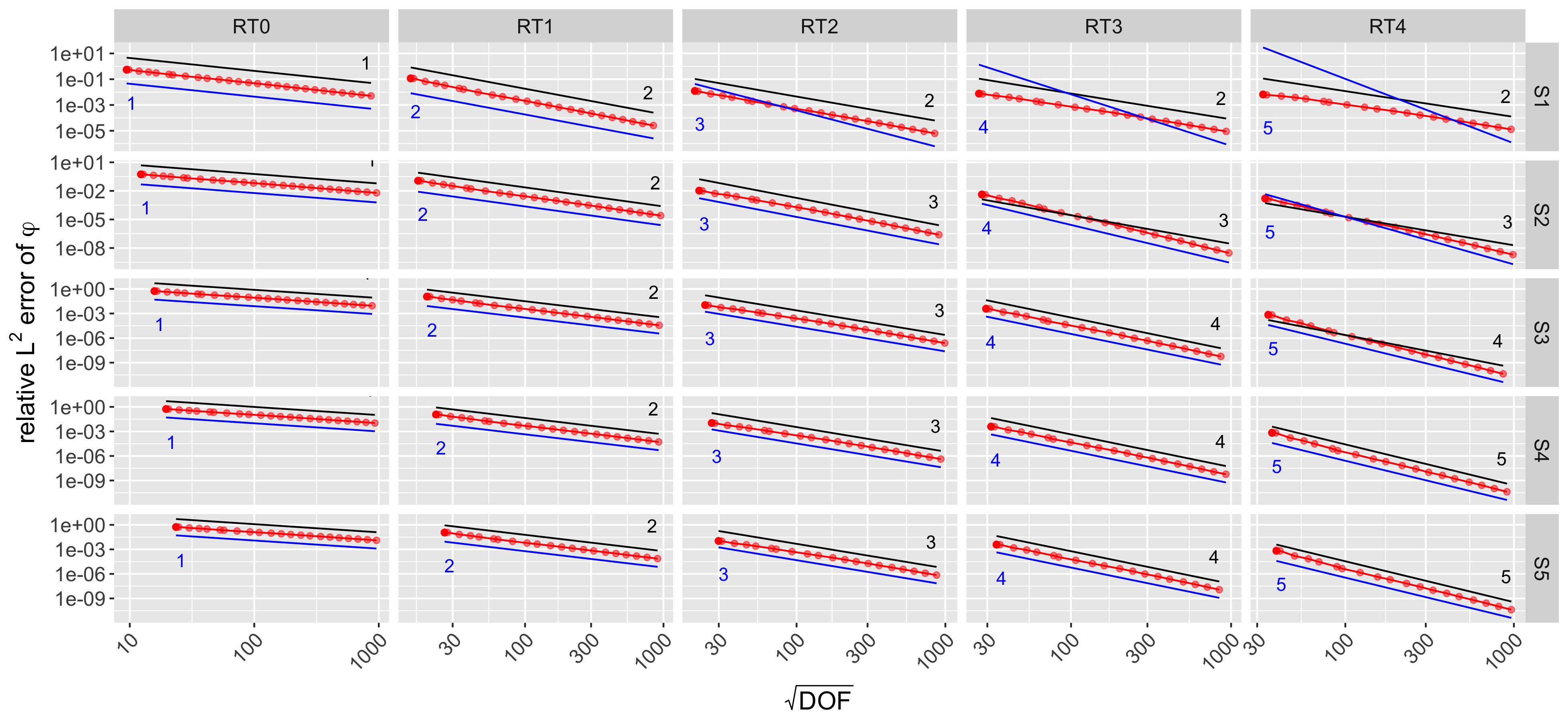}
	\caption{(cf.\ Example~\ref{example:numerics_smooth_solution})
  Convergence of $\norm{\pmb{e}^{\pmb{\varphi}}}_{L^2(\Omega)}$ vs.\ $\sqrt{\operatorname{DOF}} \sim 1/h$ employing $\RTBDMZero = \RTZero$.
	}
	\label{figure:l2_error_phi_RT}
\end{figure}

\begin{figure}[H]
	\centering
	\includegraphics[width=\textwidth]{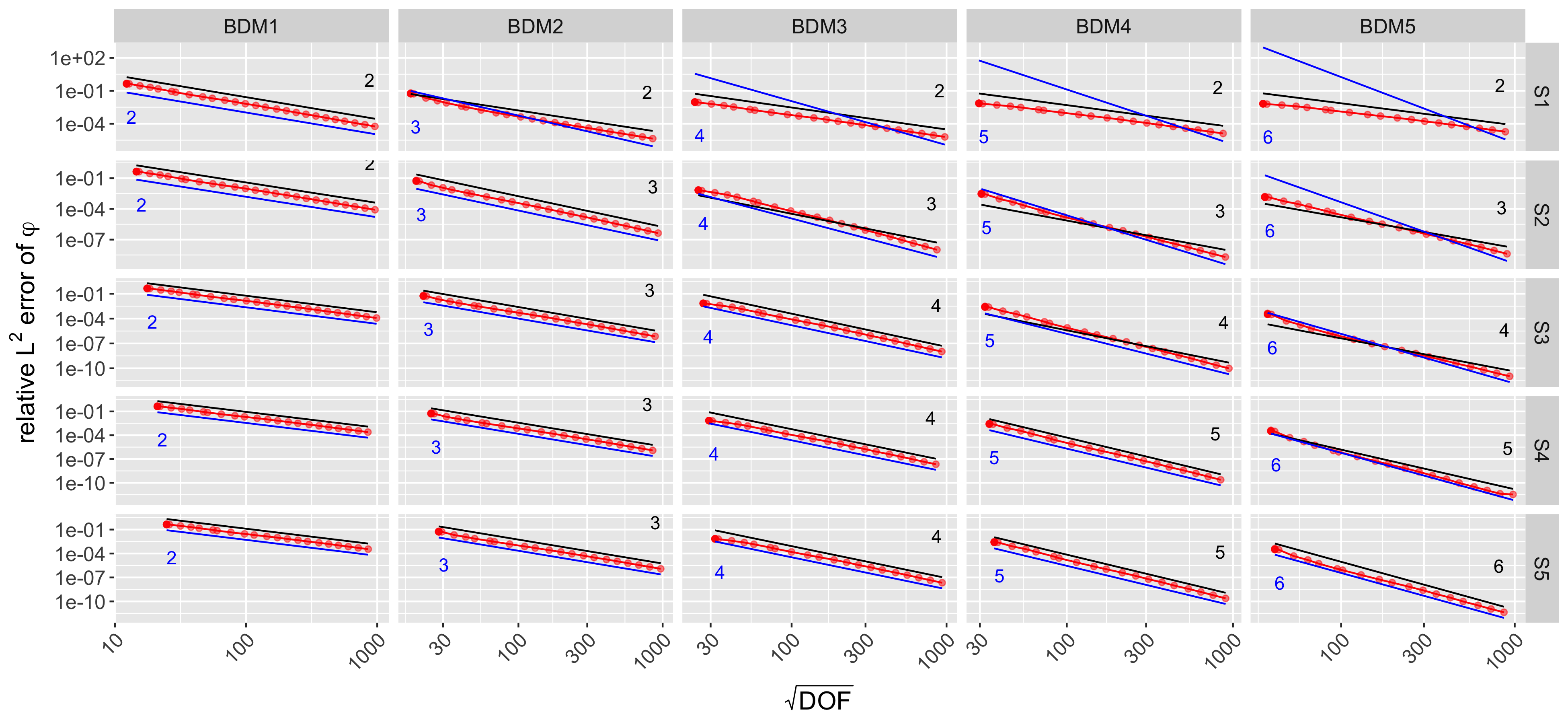}
	\caption{(cf.\ Example~\ref{example:numerics_smooth_solution})
    Convergence of $\norm{\pmb{e}^{\pmb{\varphi}}}_{L^2(\Omega)}$ vs.\ $\sqrt{\operatorname{DOF}} \sim 1/h$ employing $\RTBDMZero = \BDMZero$.
	}
	\label{figure:l2_error_phi_BDM}
\end{figure}

\begin{example}\label{example:numerics_singular_solution}
  For our second example we consider again the case of $\Omega$ being the unit sphere in $\mathbb{R}^2$.
  The exact solution $u(x,y)$ is calculated corresponding to the right-hand side $f(x,y) = \mathbbm{1}_{[0,1/2]}(\sqrt{x^2 + y^2})$.
  Therefore $u \in \cap_{\varepsilon >0} H^{5/2-\varepsilon}(\Omega)$.
  The numerical results for the choice of Raviart-Thomas elements are plotted in Figure~\ref{figure:l2_error_u_RT_singular} for $\norm{e^u}_{L^2(\Omega)}$,
  in Figure~\ref{figure:l2_error_grad_u_RT_singular} for $\norm{\nabla e^u}_{L^2(\Omega)}$ and
  in Figure~\ref{figure:l2_error_phi_RT_singular} for $\norm{\pmb{e}^{\pmb{\varphi}}}_{L^2(\Omega)}$.
  Apart from the remarks already made in Example~\ref{example:numerics_smooth_solution} we note that we observe the improved convergence result
  when dealing with limited Sobolev regularity of the data. Furthermore, in the lowest order case $p_v = 1$ the guaranteed rate seems to be suboptimal.
  The plots for the choice of Brezzi-Douglas-Marini elements are presented in Appendix~\ref{section:appendix}.
\end{example}

\begin{figure}[H]
	\centering
	\includegraphics[width=\textwidth]{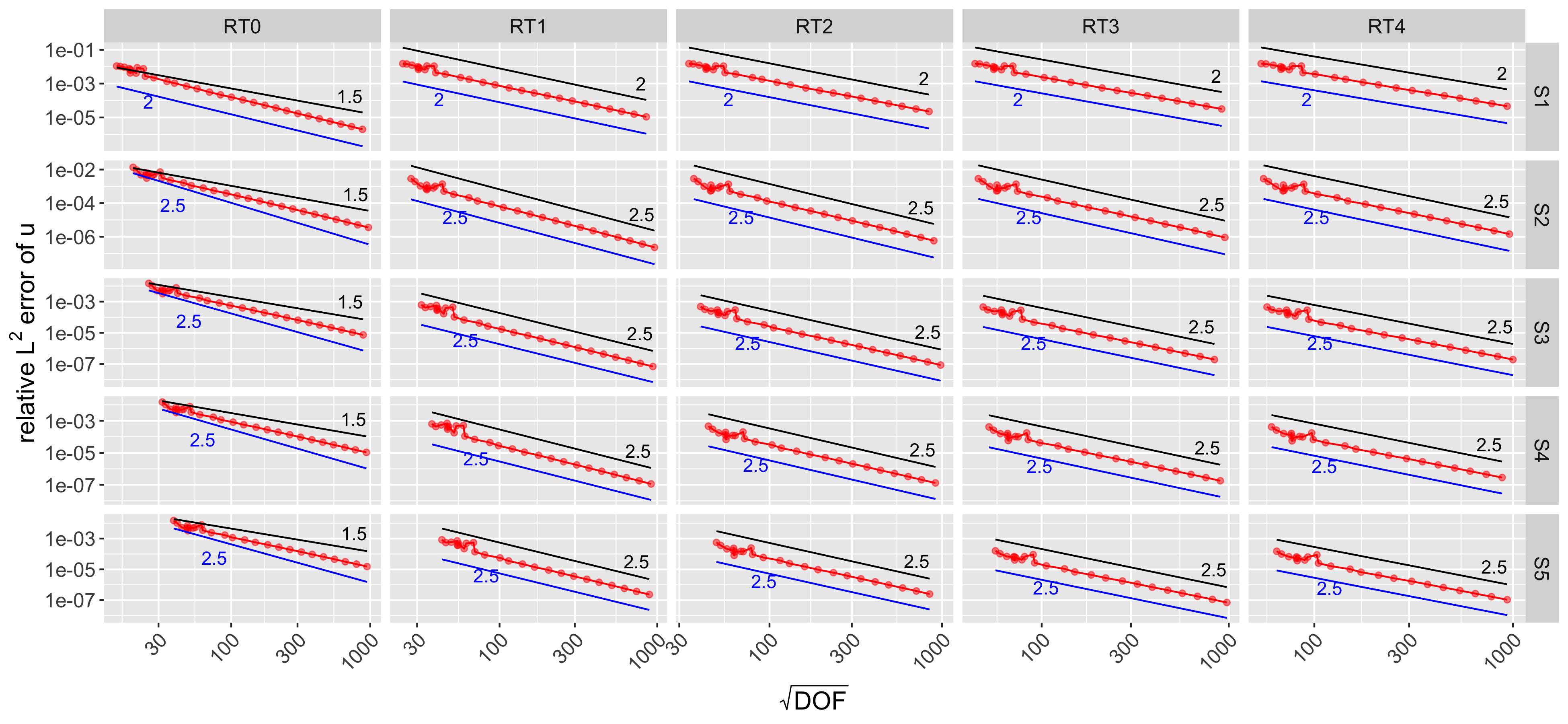}
	\caption{(cf.\ Example~\ref{example:numerics_singular_solution})
  Convergence of $\norm{e^u}_{L^2(\Omega)}$ vs.\ $\sqrt{\operatorname{DOF}} \sim 1/h$ employing $\RTBDMZero = \RTZero$.
	}
	\label{figure:l2_error_u_RT_singular}
\end{figure}

\begin{figure}[H]
	\centering
	\includegraphics[width=\textwidth]{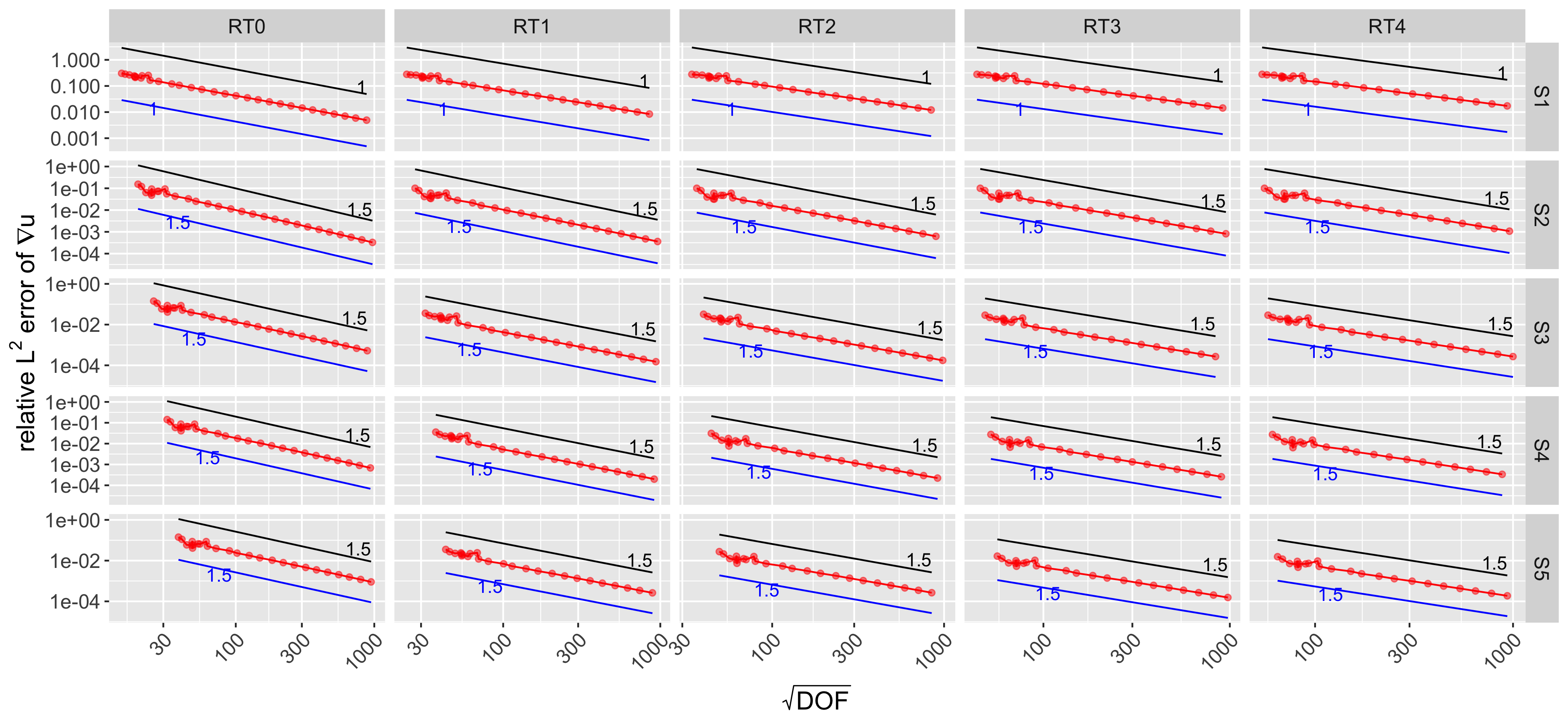}
	\caption{(cf.\ Example~\ref{example:numerics_singular_solution})
  Convergence of $\norm{\nabla e^u}_{L^2(\Omega)}$ vs.\ $\sqrt{\operatorname{DOF}} \sim 1/h$ employing $\RTBDMZero = \RTZero$.
	}
	\label{figure:l2_error_grad_u_RT_singular}
\end{figure}

\begin{figure}[H]
	\centering
	\includegraphics[width=\textwidth]{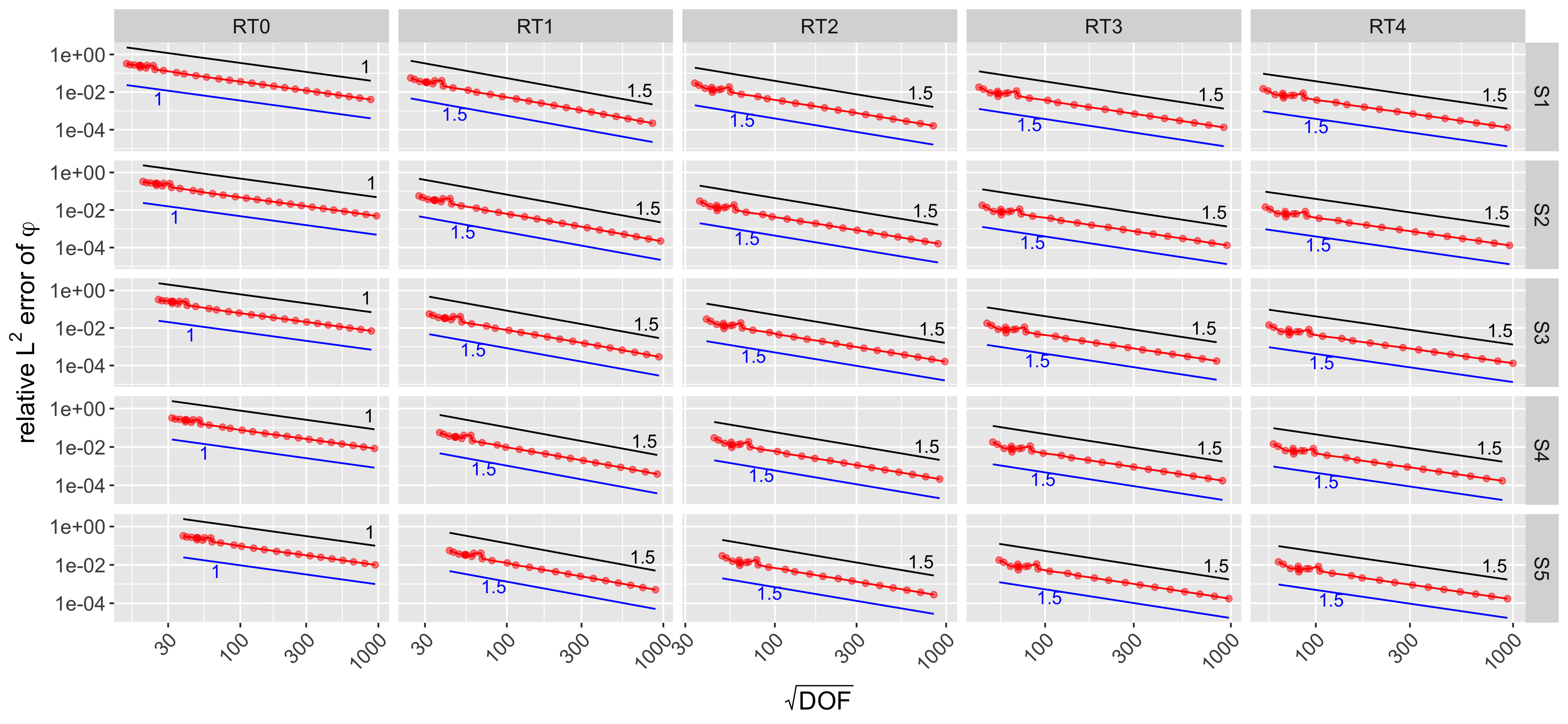}
	\caption{(cf.\ Example~\ref{example:numerics_singular_solution})
  Convergence of $\norm{\pmb{e}^{\pmb{\varphi}}}_{L^2(\Omega)}$ vs.\ $\sqrt{\operatorname{DOF}} \sim 1/h$ employing $\RTBDMZero = \RTZero$.
	}
	\label{figure:l2_error_phi_RT_singular}
\end{figure}

%% file: 05_appendix.tex
\section{Appendix}\label{section:appendix}

For completeness we present additional convergence plots below.
In Figure~\ref{figure:l2_error_u_BDM} we plot $\norm{e^u}_{L^2(\Omega)}$
employing Brezzi-Douglas-Marini elements
for the problem considered in Example~\ref{example:numerics_smooth_solution}.
The Figures~\ref{figure:l2_error_grad_u_RT} and \ref{figure:l2_error_grad_u_BDM} depicting $\norm{\nabla e^u}_{L^2(\Omega)}$ are essentially the same just one order less than $\norm{e^u}_{L^2(\Omega)}$.
The numerical results for the finite regularity solution
considered in Example~\ref{example:numerics_singular_solution} are plotted in
Figure~\ref{figure:l2_error_u_BDM_singular} for $\norm{e^u}_{L^2(\Omega)}$,
in Figure~\ref{figure:l2_error_grad_u_BDM_singular} for $\norm{\nabla e^u}_{L^2(\Omega)}$ and
in Figure~\ref{figure:l2_error_phi_BDM_singular} for $\norm{\pmb{e}^{\pmb{\varphi}}}_{L^2(\Omega)}$.

\begin{figure}[H]
	\centering
	\includegraphics[width=\textwidth]{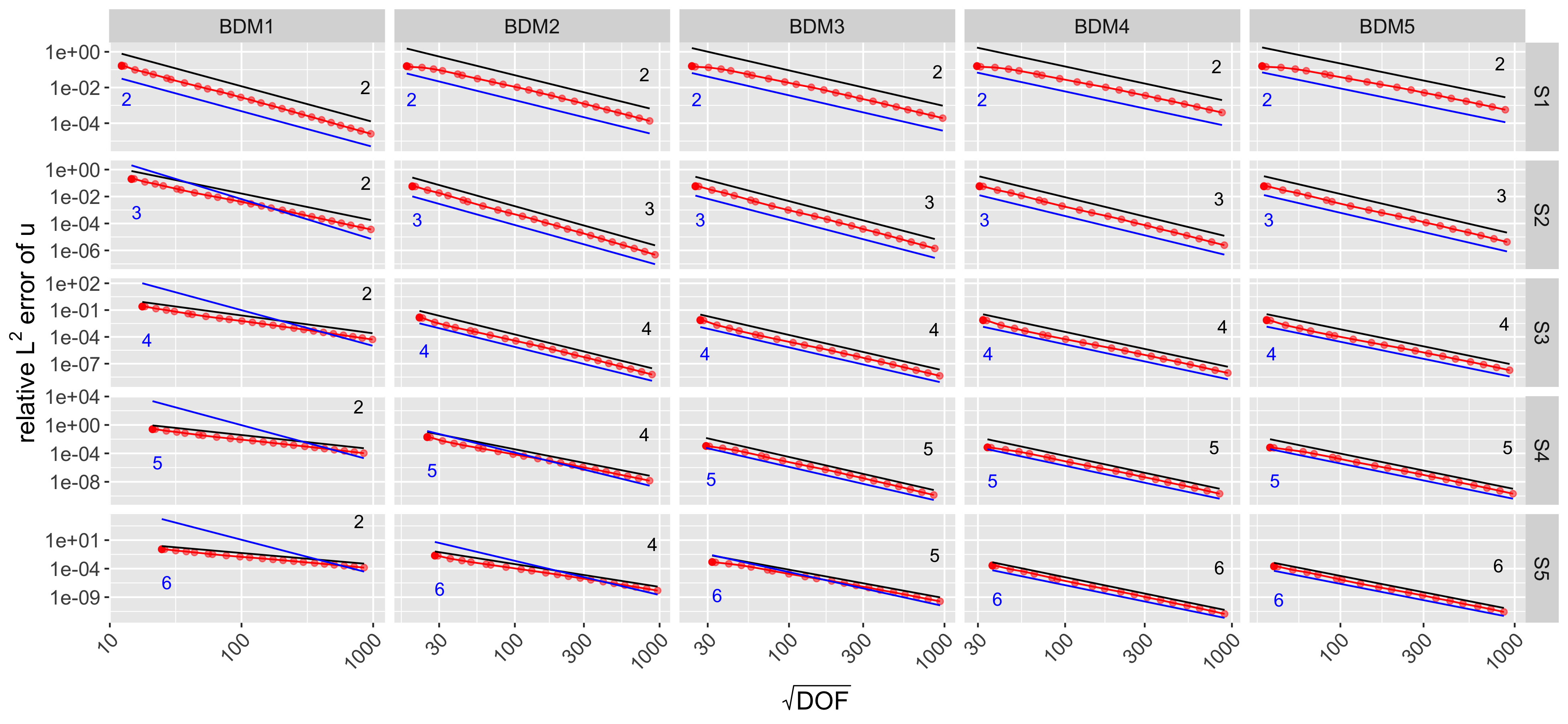}
	\caption{(cf.\ Example~\ref{example:numerics_smooth_solution})
    Convergence of $\norm{e^u}_{L^2(\Omega)}$ vs.\ $\sqrt{\operatorname{DOF}} \sim 1/h$ employing $\RTBDMZero = \BDMZero$.
        }
	\label{figure:l2_error_u_BDM}
\end{figure}

\begin{figure}[H]
	\centering
	\includegraphics[width=\textwidth]{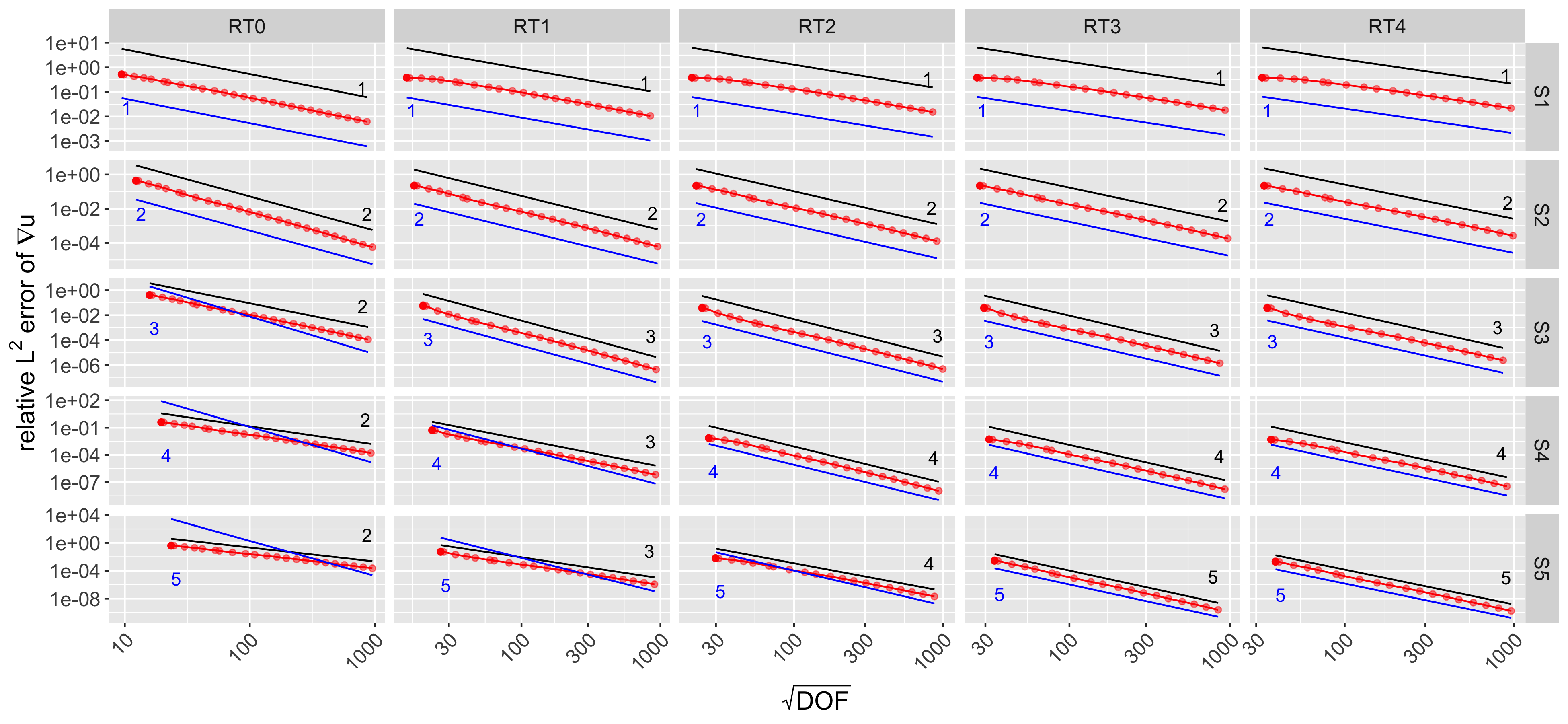}
	\caption{
(cf.\ Example~\ref{example:numerics_smooth_solution})
  Convergence of $\norm{\nabla e^u}_{L^2(\Omega)}$ vs.\ $\sqrt{\operatorname{DOF}} \sim 1/h$ employing $\RTBDMZero = \RTZero$.
	}
	\label{figure:l2_error_grad_u_RT}
\end{figure}

\begin{figure}[H]
	\centering
	\includegraphics[width=\textwidth]{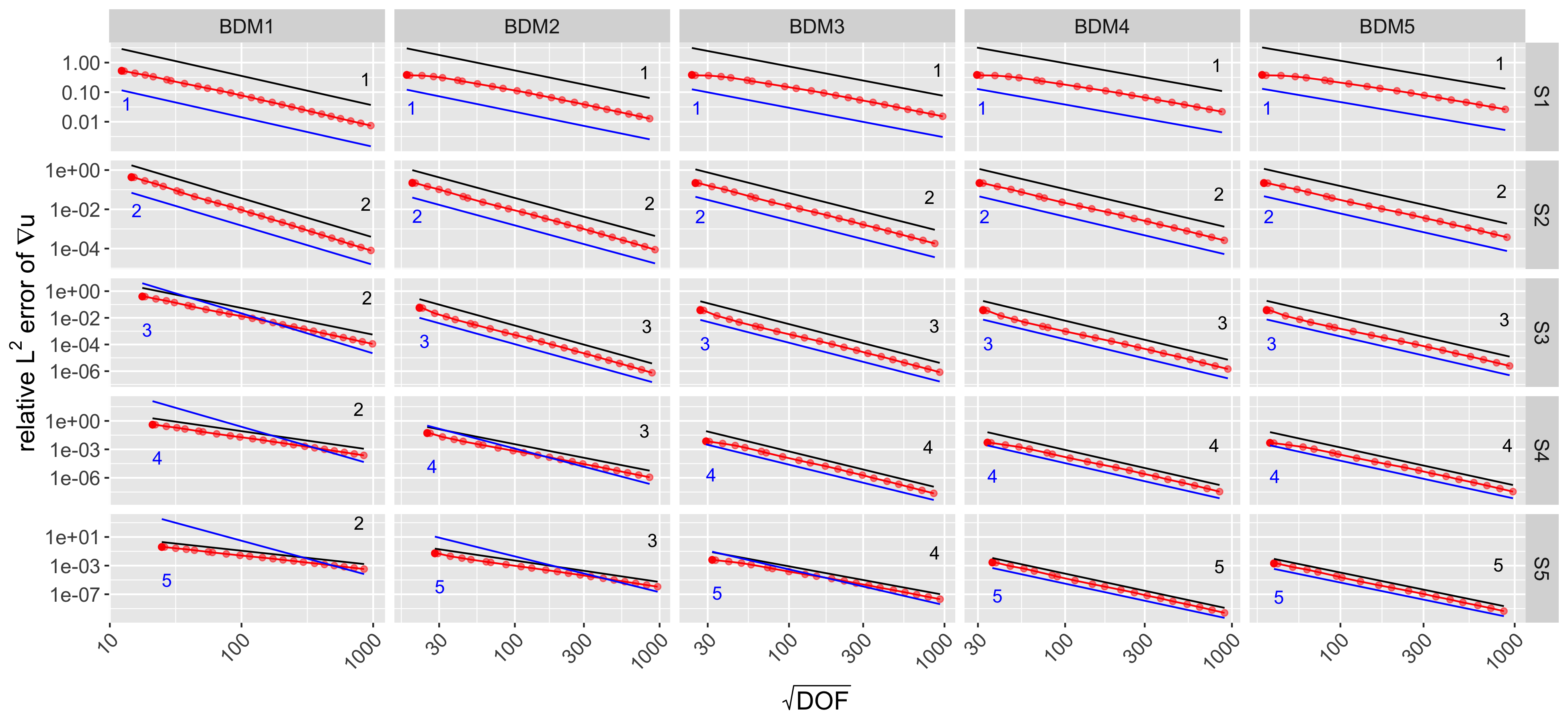}
	\caption{(cf.\ Example~\ref{example:numerics_smooth_solution})
    Convergence of $\norm{\nabla e^u}_{L^2(\Omega)}$ vs.\ $\sqrt{\operatorname{DOF}} \sim 1/h$ employing $\RTBDMZero = \BDMZero$.
	}
	\label{figure:l2_error_grad_u_BDM}
\end{figure}

\begin{figure}[H]
	\centering
	\includegraphics[width=\textwidth]{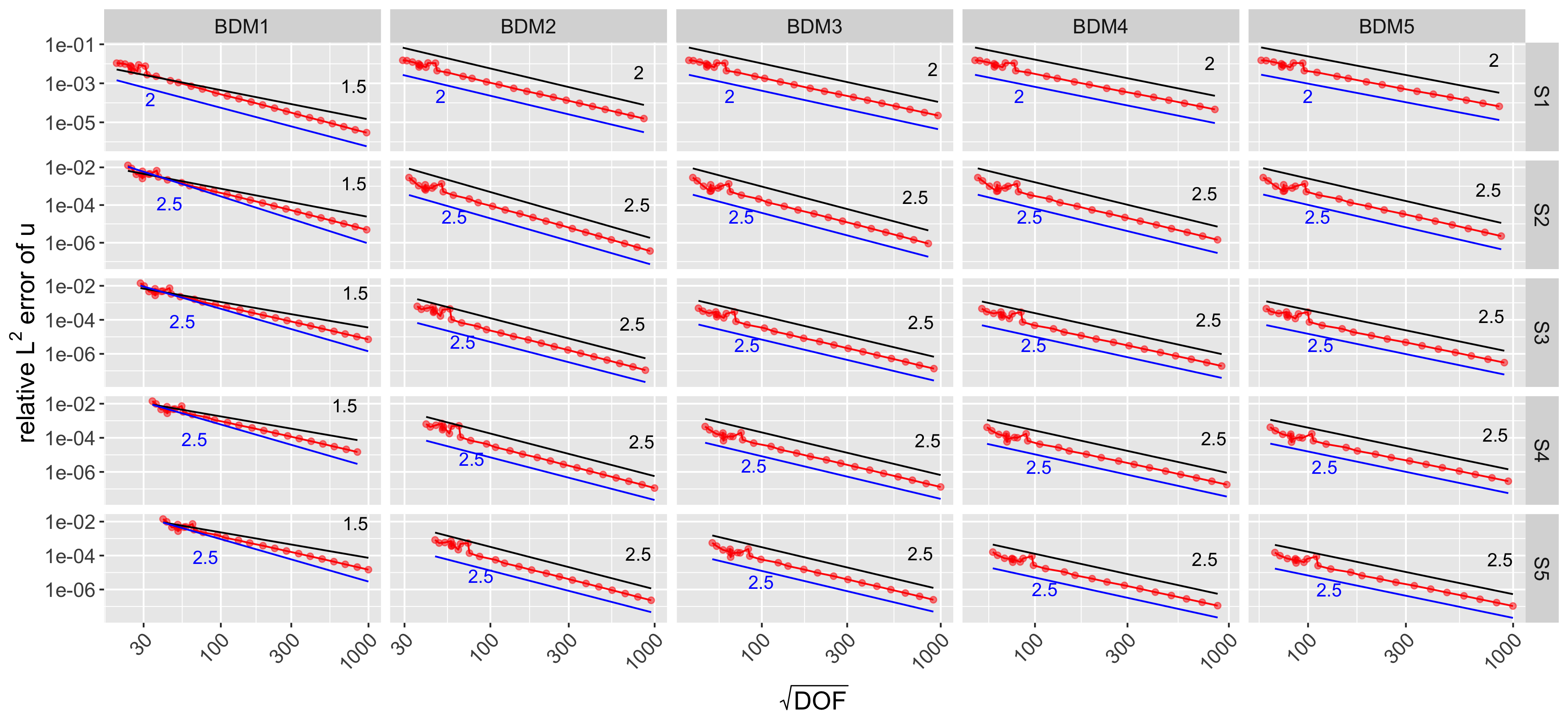}
	\caption{(cf.\ Example~\ref{example:numerics_singular_solution})
    Convergence of $\norm{e^u}_{L^2(\Omega)}$ vs.\ $\sqrt{\operatorname{DOF}} \sim 1/h$ employing $\RTBDMZero = \BDMZero$.
	}
	\label{figure:l2_error_u_BDM_singular}
\end{figure}

\begin{figure}[H]
	\centering
	\includegraphics[width=\textwidth]{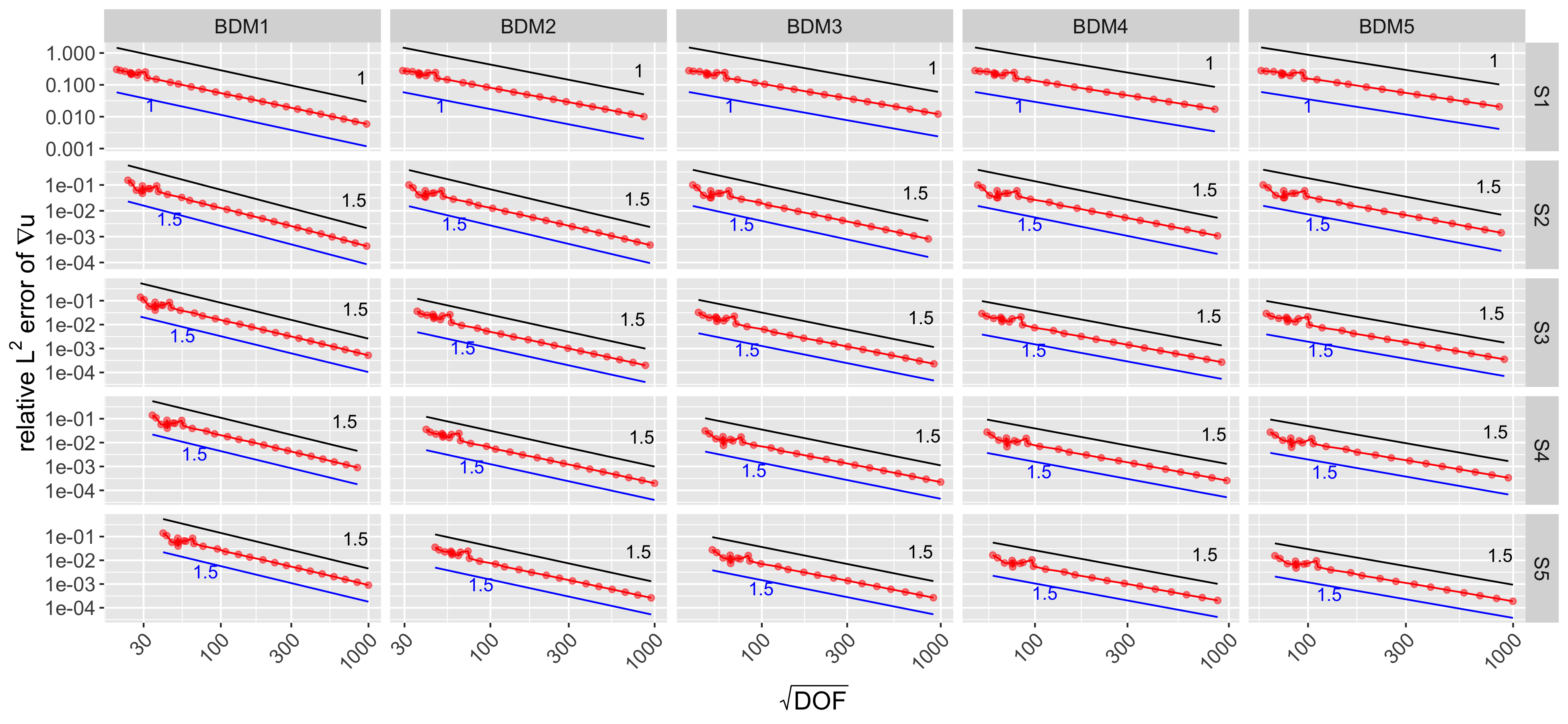}
	\caption{(cf.\ Example~\ref{example:numerics_singular_solution})
    Convergene of $\norm{\nabla e^u}_{L^2(\Omega)}$ vs.\ $\sqrt{\operatorname{DOF}} \sim 1/h$ employing $\RTBDMZero = \BDMZero$.
	}
	\label{figure:l2_error_grad_u_BDM_singular}
\end{figure}

\begin{figure}[H]
	\centering
	\includegraphics[width=\textwidth]{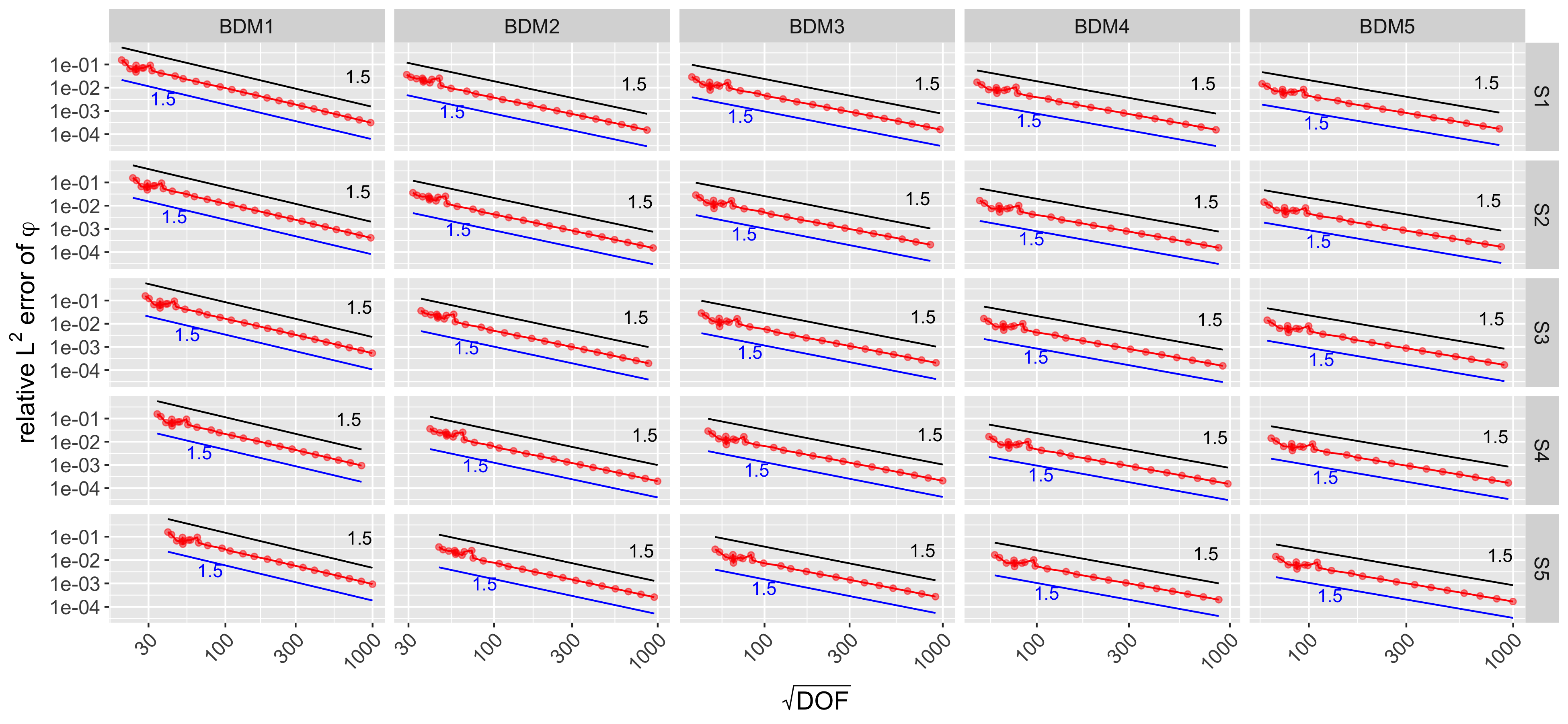}
	\caption{(cf.\ Example~\ref{example:numerics_singular_solution})
    Convergence of $\norm{\pmb{e}^{\pmb{\varphi}}}_{L^2(\Omega)}$ vs.\ $\sqrt{\operatorname{DOF}} \sim 1/h$ employing $\RTBDMZero = \BDMZero$.
	}
	\label{figure:l2_error_phi_BDM_singular}
\end{figure}